\documentclass[a4paper,10pt]{amsart}
\usepackage[utf8]{inputenc}
\usepackage{lmodern}
\usepackage[T1]{fontenc}
\usepackage{microtype} 	
\usepackage[all]{xy}
\usepackage{relsize}

\usepackage{amsmath,amssymb,amsfonts,amsthm,latexsym,mathrsfs}
\usepackage{mathtools}
\usepackage{ stmaryrd } 
\usepackage{xcolor}
\usepackage{hyperref}
\usepackage{cleveref}
\hypersetup{colorlinks=true,linkcolor=gray,citecolor=gray}
\usepackage[hyperpageref]{backref} 
\newtheorem{lemma}{Lemma}[section]
\newtheorem{proposition}[lemma]{Proposition}
\newtheorem{theorem}[lemma]{Theorem}

\newtheorem{corollary}[lemma]{Corollary}

\newtheorem{maintheorem}{Theorem}

\theoremstyle{definition}

\newtheorem{definition}[lemma]{Definition}
\newtheorem{conjecture}[lemma]{Conjecture}
\newtheorem{question}[lemma]{Question}
\theoremstyle{remark}
\newtheorem{remark}[lemma]{Remark}
\newtheorem{example}[lemma]{Example}

\newtheorem*{remark*}{Remark}
\newtheorem*{problem*}{Problem}
\newtheorem*{example*}{Example}

\DeclareMathOperator{\im}{im}

\newcommand\Ima{{\rm Im}}

\newcommand{\GL}{\operatorname{GL}}
\newcommand{\SL}{\operatorname{SL}}
\newcommand{\PSL}{\operatorname{PSL}}
\newcommand{\Ma}{\operatorname{M}}

\newcommand\Aut{\operatorname{Aut}}
\newcommand\Out{\operatorname{Out}}
\newcommand\Inn{\operatorname{Inn}}
\newcommand\Hom{\operatorname{Hom}}
\newcommand\Map{\operatorname{Map}}
\newcommand\Coker{\operatorname{coker}}
\newcommand\Span{\operatorname{span}}
\newcommand\supp{\operatorname{supp}}

\newcommand\Irr{\operatorname{Irr}}
\newcommand\Lin{\operatorname{Lin}}

\newcommand\Res{\rm Res}

\newcommand\Inf{\operatorname{Inf}}
\newcommand\Tra{\operatorname{Tra}}

\newcommand\Char{{\rm char}}

\newcommand{\ZZ}{\mathcal{Z}}
\renewcommand{\O}{\mathcal{O}}
\newcommand{\U}{{\mathcal U}}
\newcommand\rank{\rm rank \,}

\newcommand\ot{\otimes}
\newcommand\op{\oplus}
\newcommand\mc{\mathcal}
\newcommand\ov{\overline}

\newcommand\wt{\widetilde}
\newcommand{\qa}[3]{\left(\frac{#1, #2}{#3}\right)}

\newcommand{\N}{{\mathbb N}}
\newcommand{\Z}{{\mathbb Z}}

\newcommand{\Q}{{\mathbb Q}}

\newcommand{\C}{{\mathbb C}}

\newcommand\G{\Gamma}

\newcommand\Bic{\operatorname{Bic}}
\newcommand\lcm{\operatorname{lcm}}
\newcommand\PCI{\operatorname{PCI}}

 \newcommand\restr[2]{{
   \left.\kern-\nulldelimiterspace 
   #1 
   \right|_{#2} 
   }}

\definecolor{Vino}{rgb}{0.256,0,0}

\makeatletter
\@namedef{subjclassname@2020}{\textup{2020} Mathematics Subject Classification}
\makeatother

\usepackage[margin=1.3in]{geometry}

\thanks{The first author is grateful to Fonds Wetenschappelijk Onderzoek vlaanderen - FWO (grant 88258), and le Fonds de la Recherche Scientifique - FNRS (grant 1.B.239.22) for financial support. Part of this work was written while the first author was in residence at the Simons Laufer Mathematical Sciences Institute (formerly MSRI) in Berkeley, California, during the Spring 2024 semester}

\begin{document}
\title{Units of twisted group rings and their correlations to classical group rings}
\author{Geoffrey Janssens}
\author{Eric Jespers}
\author{Ofir Schnabel}

\address{(Geoffrey Janssens) \newline Institut de Recherche en Math\'ematiques et Physique, UCLouvain, 1348 Louvain-la-Neuve, Belgium and \newline
Department of Mathematics and Data Science, Vrije Universiteit Brussel,
Pleinlaan $2$, 1050 Elsene \newline E-mail address: {\tt geoffrey.janssens@uclouvain.be}}
\address{(Eric Jespers) \newline Department of Mathematics and Data Science, Vrije Universiteit Brussel,
Pleinlaan $2$, 1050 Elsene, Belgium \newline E-mail address: {\tt eric.jespers@vub.be}}
\address{(Ofir Schnabel) \newline Department of Mathematics, ORT Braude College, 2161002 Karmiel, Israel \newline E-mail address: {\tt ofirsch@braude.ac.il}}

\begin{abstract}
This paper is centered around the classical problem of extracting properties of a finite group $G$ from the ring isomorphism class 
of its integral group ring $\mathbb{Z} G$. This problem is considered via describing the unit group $\mathcal{U}( \mathbb{Z} G)$ generically for a 
finite group.  Since the $`90s$ several well known generic constructions of units are known to generate a subgroup of finite index in $\mathcal{U}(\mathbb{Z } G)$ if $\mathbb{Q} G$ does not have so-called exceptional simple epimorphic images, e.g. $M_2 (\mathbb{Q})$. However it remained a major open problem to find a {\it generic} construction under the presence of the latter type of simple images. In this article we obtain such generic construction of units. Moreover, this new construction also exhibits new properties, such as providing  generically free subgroups of large rank. As an application we answer positively for several classes of groups recent conjectures on the rank and the periodic elements  of the abelianisation $\mathcal{U}(\mathbb{Z} G)^{ab}$. To obtain all this, we investigate the group ring $R \Gamma$ of an extension $\Gamma$ of some normal subgroup $N$ by a group $G$, over a domain $R$. More precisely, we obtain a direct sum decomposition of the (twisted) group algebra of $\Gamma$ over the fraction field $F$ of $R$ in terms of various twisted group rings of $G$ over finite extensions of $F$. Furthermore, concrete information on the kernel and cokernel of the associated projections is obtained. Along the way we also launch the investigations of the unit group of twisted group rings and of $\mathcal{U}( R\Gamma)$ via twisted group rings. 
\end{abstract}

\maketitle

\newcommand\blfootnote[1]{%
  \begingroup
  \renewcommand\thefootnote{}\footnote{#1}%
  \addtocounter{footnote}{-1}%
  \endgroup
  }

\vspace{-0,9cm}
\tableofcontents

\section{Introduction}
This paper contributes to the study of a finite group $\G$ via its representation theory over a 
number field\footnote{Most of the results will in fact hold in the generality where $F$ is any field with 
$\Char(F) \nmid | \G|.$ } $F$ and its ring of integers $R$. The overarching question is which group 
theoretical invariants of $\Gamma$ are determined by the $R$-algebra $R\G$, or in other words by the regular $R\G$-module. 
From the vast literature, 
see for example
\cite{EricAngel1,EricAngel2,KarUnit,MR798076, SehgalBook93, SehgalSurvey03}, 
one can somehow
 distil an approach in two steps to that. Firstly one considers the decomposition of the semisimple group algebra $F\Gamma$ given by Wedderburn-Artin's theorem:
$$F\G \cong \Ma_{n_1}(D_1) \oplus \cdots  \oplus \Ma_{n_q}(D_q),$$
where $D_i$ are finite dimensional division $F$-algebras. If one now chooses an order $\mathcal{O}_i$ in 
each $D_i$, then we obtain two orders in $F\G$, namely $R\G$ and 
$\bigoplus_{i = 1}^q \Ma_{n_i}(\mathcal{O}_i)$. Because $F$ is a number field, in particular $R$ is a `nice' Dedekind domain\footnote{The property that the unit group of two orders in a common semisimple algebra are commensurable requires that $R/I$ is finite for every non-zero ideal $I$ of $R$ (e.g. see \cite[Lemma 4.6.9.(3)]{EricAngel1}). The use of `nice' refers to this extra property.}, it is well known that the two 
orders share many properties. In particular their unit groups have a common subgroup of finite index 
(see for example \cite{EricAngel1,SehgalBook93}).
The aim of the first step is to obtain as much information as possible on the $F$-character degrees $n_i$ 
and the form of the division algebras $D_i$. Consequently, it is to be expected that the number 
theoretical properties of the number field $F$ will play a special role here.

Recall that there is a bijection between the matrix components above, the absolutely irreducible $F$-characters of $\G$ 
and the primitive central idempotents of $F\G$ (a set denoted $\PCI (F\G)$). For  a survey on the construction of primitive central idempotents and a description of its associated matrix components we refer to \cite{EricAngel1} and for some of the  most recent progress on this  topic to  \cite{MR4328652,MR3995003,MR3883245}. A setback of making the switch from $R\G$ to 
$\bigoplus_{i = 1}^q \Ma_{n_i}(\mathcal{O}_i)$ is that one is somehow replacing the group $\G$ by the larger group 
$\prod_{e \in \PCI (F\G)} \G e$. Due to this, one loses information on `ties'. The role of the second 
step consists then to focus on the given group $\G$. This might be via representation theoretical methods or 
group theoretical ones. To clarify the latter, we now set $F= \Q$ and so $R = \Z$. Then it is known \cite{MR4356272}
 that the group isomorphism type of the unit group $\U (\Z \G)$ and the ring isomorphism 
type of $\Z \G$ contain the same information\footnote{Thus $\Z \G \cong \Z H$ if and only if 
$\U (\Z \G) \cong \U ( \Z H)$. It is folklore that it holds more generally for any $\G$-adapted coefficient 
ring $R$, i.e. where all prime divisors of $|\G|$ are not divisible in $\U(R)$. In short, this follows in that 
case from the linear independence of finite subgroups of $\U(RG)$.}. The gain of this is that $\U (\Z \G)$ is 
an arithmetic subgroup in some linear reductive algebraic group (in particular it is a finitely presented 
group), allowing the use of classical but strong methods in algebraic groups or geometric group theory. 

From this unit group point of view and using \cite{Johnson}, the first step describes a full list of invariants determining the commensurability class of $\U(\Z \Gamma)$, whereas the second step aims to filter till its isomorphism class. A usual approach to the latter consists of constructing an `interesting' torsion-free subgroup $N$ of finite index in $\U(\Z \Gamma)$. On one hand such $N$ would (ideally) be normal and the associated finite quotient would reflect properties of the finite subgroups of $\U (\Z \Gamma)$. On the other hand, importantly, the construction of $N$ needs to be generic in the sense that it does not require knowing the isomorphism type of the group basis $\Gamma$. 

In this article we will contribute to both steps in a novel way and it can be summarized as follows:

(1) In the first step we make a shift in the traditional philosophy by regrouping simple components into certain twisted group algebras which arise by viewing $\G$ as a non-trivial extension. \Cref{Sectie decompositie} till \Cref{section correlations} is devoted to constructing in general such decomposition and giving down-to-earth descriptions of the projections and their (co)kernels. Below we will give an overview of those results. 

(2) Classically, the main generic constructions of units in $\U(\Z \G)$ are Bass and bicyclic units. Using the solutions to the subgroup congruence problem for $\SL_n(D)$, with $n \geq 2$, it is known since the mid `90s \cite{JesLea} that if $\Q \Gamma$ has no simple $2\times 2$-components\footnote{More precisely, if $\Q \Gamma$ has no simple component $\Ma_2(D)$ with $D$ containing an order with finite unit group. Equivalently, if no $\SL_n(D)$ has $S$-rank one for some finite set of place $S$ of $\mathcal{Z}(D)$ containing the archimedean.} and the only simple $1\times 1$ components are commutative, then the group generated by the Bass and (generalized) bicyclic units is of finite index. In particular the subgroup $\mathcal{B}$ generated by them could serve as $N$ in the explanation above. Therefore, a major open problem in group rings of the last decennia has been to find generic constructions if $2\times 2$ components are present. In \Cref{section h-units} we give the first such generic construction, which we call $H$-units. Moreover these elements can also contribute to other components, hence even without such problematic components, when combined with $\mathcal{B}$ it may yield a `stronger $N$'. As an illustration thereof we answer in the positive for various infinite families of groups conjectures on the rank and the torsion of the abelianisation $\U (\Z \Gamma)^{ab}.$ This is possible due to the nice properties of $H$-units, e.g. they yield free groups of large ranks. The proof of all this builds on the work for step (1) and \Cref{sectie twisted bicyclic en de rest} till \Cref{section elem ab description}.

We will now explain the main results of this article in more detail, starting with the contributions to (classical) questions in (untwisted) group rings. 

\vspace{0,2cm}
\noindent{\it A new generic construction of units and their contribution to the structure of $\U(\Z G)$.}

 For any groups $W \leq H$ and $K$ and any group homomorphism $f : H \rightarrow K$ it is easily observed that $[K : f(W)]$ is finite exactly when $[H : \langle \ker(f), W \rangle]$ is finite. For the unit group of $\Z G$, a useful incarnation thereof is for $f$ the norm map $nr$ of $\Z G$ which maps an element on a tuple recording the reduced norm of each projection onto a simple component of $\Q G$. By definition $\SL_1(\Z G) =  \ker(nr)$, see (\ref{def reduced norm}) and (\ref{def of SL1}) for precise definitions. Doing so, see \cite[Proposition 5.5.1]{EricAngel1}, one has that 
$$\langle \SL_1(\Z G), \U (\mathcal{Z}(\Z G)) \rangle \text{ is of finite index in } \U (\Z G).$$
Another important incarnation of the above is for $f$ the mapping of $\U (\Z G)$, via $GL_1(\Z G)$, into its Whitehead group $ K_1 (\Z G) := \GL (\Z G)^{ab}.$ This enables one to replace $\U (\mathcal{Z}(\Z G))$ by any subgroup of $\U (\Z G)$ that maps to a finite index subgroup of $ K_1 (\Z G)$. Thanks to a theorem of Bass and Milnor, an example of such a group working for any finite group is given by the group generated by the so-called Bass units. For background we refer to \cite[Section 1]{EricAngel1}. 

In conclusion, the problem to generically construct a finite index subgroup of $\U (\Z G)$ is reduced to $\SL_1(\Z G)$, 
i.e. the elements of norm $1$. For any tuple $(g,h)$ with $g \notin N_G(\langle h \rangle)$ one can construct 
the elements $1 + (1 - h) g \sum_{j =1} ^{o(h)} h^j$ and $1 +  (\sum_{j =1} ^{o(h)} h^j) g (1-h)$ which are unipotent units and hence in $\SL_1(\Z G)$. 
These elements are called {\it bicyclic units} and the group they generate is  denoted $\Bic (G)$. 
A theorem of Jespers and Leal \cite{JesLea} says that $\Bic(G)$ is of finite index in $\SL_1( \Z G)$ 
under some restrictions on the simple components $\Ma_n(D)$ of $\Q G$. For $ n \geq 2$ the restriction 
is that there is no $\Ma_2(D)$ with $D \in \{ \Q, \Q(\sqrt{-d}) , \qa{-a}{-b}{\Q} \}$, for $a,b, d \in \Z_{>0}$ and $ \qa{-a}{-b}{\Q}$ denotes a quaternion algebra. Therefore these $\Ma_2(D)$ are called {\it exceptional} of type (II) (see paragraphs behind \Cref{Jespers-LEal for twisted bicyclic} for complete definition).

In \Cref{section h-units} we produce a new generic construction of elements in $\SL_1(\Z G)$.  For any triple $(g,h,Q)$ with $g,h \in G, Q \leq G$ satisfying the two conditions in \Cref{def h-units} we construct a group $\mathcal{H}(g,h,Q)$. The elements therein are called\footnote{As explained in \Cref{remark on the name}, their name refers to the crucial role of the second cohomology group, and in particular twisted group rings, to both discover the elements and the proof of the subsequent theorem.} $H$-units and interestingly the generators are usually not unipotent. For every quadruple of numbers $(x_1,x_2,y_1,y_2)\in \N^4$ satisfying the equations (\ref{Admissable quadruple cond}) there is an associated $H$-unit denoted $v_{(x_1,x_2,y_1,y_2)}$. The free group of rank $n$ we denote by $F_n$.

\begin{maintheorem}[\Cref{prop H-unit is H-unit}]\label{H-unit prop th intro}
Let $(g,h,Q)$ be a triple as in \Cref{def h-units}. Then,
\begin{enumerate}
\item[(1)] $\mathcal{H}(g,h,Q)$ is a finitely generated subgroup of $\SL_1(\Z G)$ and $v_{(x_1,x_2,y_1,y_2)}^{-1} = v_{(-x_1,-x_2,y_2,y_1)},$
\item[(2)] $\mathcal{H}(g,h,Q) \neq 1$ if and only if $[g,h] \notin \langle g \rangle Q.$
\end{enumerate}
Moreover, for $\mathcal{H}(g,h,Q) \neq 1$, 
\begin{enumerate}
\item[(3)] if $o(gQ)|Q| =2$, then $\mathcal{H}(g,h,Q) \cong F_3 \times C_2$, and
\item[(4)] if $o(gQ)|Q| > 2$, then $\mathcal{H}(g,h,Q) \cong F_{n}$ with $n = 1 + \frac{(o(gQ)|Q|)^3}{6}\prod_{p} (1 - \frac{1}{p^2})$, where the product runs over the prime divisors $p$ of $o(gQ)|Q|$.
\end{enumerate}
\end{maintheorem}
More precisely, see \Cref{H-unit contain principal congruence in one comp}, we give a concrete description of $\mathcal{H}(g,h,Q)$ and not only of its isomorphism type. Nevertheless, the structure of the group $\langle \mathcal{H}(g_i,h_i,Q_i) \mid i \in I \rangle$ generated by the $H$-units corresponding to several triples $(g_i,h_i,Q_i)$ is still mysterious to us.  There are several natural questions, in particular \Cref{question grp gen by two triples}:  when it is a direct product of the groups $\mathcal{H}(g_i,h_i,Q_i)$? 

From \Cref{subsection H-units infnite index ext} on we consider the case that $3 \nmid |G$| and $\Q G$ has  exceptional components of the type $\Ma_2(\Q)$. As a first application of Theorem~\ref{H-unit prop th intro} we show that the Jespers-Leal theorem can be extended to include the difficult case $\Ma_2(\Q)$. This is recorded in \Cref{extend Jespers-Leal remark} and follows from the proof of \Cref{H-units give finit index th intro}, resolving in this case the problem of constructing generically a finite index subgroup. By $\mathcal{B}(G)$ we denote the subgroup generated by the Bass and bicyclic unit and by $\mathcal{H}(G)$ the subgroup generated by $\mathcal{H}(g,h,Q)$ for any admissible triple $(g,h,Q)$.

\begin{maintheorem}[\Cref{H-units for 2-grps give finite index}]\label{H-units give finit index th intro}
Let $G$ be a $2$-group such that the only exceptional components of $\Q G$ are of the form $\Ma_2(\Q)$, then $\langle \Bic (G) , \mathcal{H}(G) \rangle $ is of finite index in $\SL_1(\Z G)$. Consequently,  $\langle \mathcal{B} (G) , \mathcal{H}(G) \rangle$ is of finite index in $\U (\Z G)$.
\end{maintheorem}

We have chosen to focus on $\Ma_2(\Q)$ as it is by far the most frequent exceptional component, 
cf. \cite[Appendix A]{BJJKT}, and such a component naturally yields triples $(g,h,Q)$ as in \Cref{def h-units}. 
However, the statement of \Cref{H-units give finit index th intro} sometimes also holds 
in the presence of other exceptional components, but we have not tried to pursue this line of investigations.\vspace{0.2cm}

Our {\it second application} is about the abelianisation $\U (\Z G)^{ab} \cong \Z^n \times T$, with $T$ a finite abelian group. The number $n$ is called the rank of the abelianisation of $\U (\Z G)$. Recall that $\U(\Z G) = \pm V(\Z G)$ with $V(\Z G)$ the group of invertible elements with augmentation one. Recently the following questions got some attention:\vspace{0.1cm}

\begin{enumerate}
\item[(R1)] Is the rank, as abelian group, of $\ZZ (\mathcal{U}(\Z G))$ and  $\U(\Z G)^{ab}$ equal? In particular if  $\ZZ (\mathcal{U}(\Z G))$ is finite, is $\U(\Z G)^{ab}$ also finite? (See \cite[Question 7.8 and Proposition 7.9]{BJJKT}.) \vspace{0.1cm}

\item[(P)] Let $p$ be a prime. If $V(\Z G)^{ab}$ contains an element of order $p$, does $G^{ab}$ also contain an element of order $p$? (See \cite[page 2]{BMM}.) \vspace{0.2cm} 
\end{enumerate}

Question (P) is one of three questions formulated by B\"achle, Maheshwary and Margolis \cite{BMM} and the labelling refers to theirs. A stronger version of (P) was also formulated in \cite{BMM}, namely that $\exp (V(\Z G)^{ab}) = \exp(G^{ab}).$ Currently, the only infinite family for which (P) has been proven are the dihedral groups $D_{2p}$ with $p$ prime \cite[Theorem C]{BMM}. 

We consider the `most degenerate' classes of $2$-groups, namely those where all the matrix components are exceptional and of the form $\Ma_2(\Q).$ Such groups have been classified by Jespers, Leal and del R\'io \cite{JesLealRio, JesRioCrelle}, namely $G \cong K \times C_2^n$ with $K$ a finite group that is a member of seven (infinite) families of groups $\mathcal{G}_1,  \ldots,  \mathcal{G}_7$, recalled in \Cref{subsection H-units infnite index ext}. However our proofs do not require the precise classification. In the next result we denote by $cl_{\G}(H)$ the normal closure of some subgroup $H$ in the larger group $\G.$ This result shows that in presence of exceptional components there are some natural obstructions towards conjectures (R1) and (P). 

\begin{maintheorem}[\Cref{Theorem abelianisation of only M2(Q) components} and \Cref{families of grps with H-unit new}]\label{The abel th intro}
Let $G = K \times C_2^n$, with $K$ a group in $\mathcal{G}_1 \cup \ldots \cup \mathcal{G}_7$, and $\pi$ the natural epimorphism of $\U(\Z G)$ onto $\U(\Z G)^{ab}$. Then 
$$
\rank  \, \U(\Z G)^{ab}  = \rank \ZZ (\mathcal{U}(\Z G)) + \rank  \pi(\langle \mc{H}(G)_{un}\rangle) 
$$
where $\mc{H}(G)_{un} = \{ x \in \mc{H}(G) \mid x \text{ is unipotent } \}.$ Furthermore,
$$\exp(V(\Z G)^{ab}) = \lcm\left( \exp(G^{ab}), \exp \left(\frac{V(\Z G)}{\ZZ(V(\Z G))\; cl_{\U(\Z G)}(\langle \Bic(G), \pm G \rangle)}\right)^{ab} \right). $$
Moreover, 
\begin{enumerate}
    \item The $H$-units $\mathcal{H}(G)$ are of finite index in $\SL_1(\Z G)$ despite that $\Bic(G)$ might be of infinite index. 
    \item If $G$ satisfies $(\ref{prop star})$, then $\ZZ(\U(\Z G)) \; cl_{\U(\Z G)}(\langle \Bic(G), \pm G \rangle)$ is of finite index in $\U( \Z G)$ and  both (R1) and (P) have a positive answer.
\end{enumerate}
\end{maintheorem}

Recall that $\Span_{\Q}\{ ge \mid g\in G \} = \Ma_n(D)$ for $e$ a pritimive central idempotent of $\Q G$ and $(\Q G)e \cong \Ma_n(D).$ Condition (\ref{prop star}) appearing in part (2) of \Cref{The abel th intro} is about whether the generators of the finite group $Ge$ have pre-images in $G$ of the same order. If this is the case, then part (2) tells that the obstructions (i.e. the second factor in each formula) vanish. 

Alternatively (\ref{prop star}) can be interpreted as a condition on the proportions $o(g)/o(ge)$ for $g \in G$. Inspired by \Cref{The abel th intro}, in upcoming work by the first author, it will be proven that if these proportions are `large' then both obstructions are non-trivial and hence (\ref{prop star}) is a non-artificial condition.

As a final application, we recover in \Cref{families of grps with H-unit new} a result of del R\'io and Ruiz 
\cite[Theorem 1.1]{RioRuiz} saying that $M := \prod_{e \in \PCI(\Q G)} \SL_1(\Z G) \cap (1-e + \Q Ge)$ 
is the largest direct product of free groups in $\mathcal{U}(\Z G)$. However, using $H$-units, our proof is uniform, 
i.e. we do not use the classification for $G$, and yield more explicit generators. Namely, 
$M = \langle \mathcal{H}(g_i,h_i,Q_i) \rangle \cong F_n^{q}$ for some triples $(g_i,h_i,Q_i)$, $1 \leq i\leq q$ and 
$n$ explicit. Furthermore, in some cases the bicyclic units and $H$-units together yield a normal complement of 
$G$ in $V(\Z G)$. As formulated in \Cref{question when generic constr normal complement}, it would be interesting to 
investigate this phenomenon further. 

\vspace{0,2cm}
\noindent {\it Decomposition in twisted group rings as valuable substitute for the Wedderburn-Artin decomposition.} 
The statements and proofs of \Cref{H-units give finit index th intro} and \Cref{The abel th intro} are in 
the framework of untwisted group rings. However, interestingly, they heavily depend on \Cref{H-unit prop th intro} 
whose proof crucially requires twisted group rings. Recall that $F$ is a number field and $R$ is its ring of integers. 
The starting observation is that an exceptional component of $F[\Gamma]$ corresponds to an irreducible $F$-representation of $\Gamma$, say $\varphi$. Furthermore a representation allows in a natural way to view $\G$ as an extension 
\begin{equation}\label{extension intro}
1\rightarrow N \rightarrow \Gamma \overset{\lambda}\rightarrow G\rightarrow 1
\end{equation}
where $N= \ker(\varphi)$ and $G = \Ima(\varphi)$. Corresponding to this is an algebra decomposition $F\Gamma \cong FG \oplus FG(1-\widehat{N})$ where $\widehat{N}$ is some central idempotent. The second observation is that the summand $FG$ somehow originates from the trivial representation of $N$. In \Cref{Sectie decompositie} we show how the irreducible representations of $N$ can be used to concretely decompose $F\G$ in terms of certain twisted group rings and crossed products of the smaller group $G$.

Before going into details, we first recall the definition of a twisted group ring. For a $2$-cocycle $\alpha \in Z^2(G, R^*)$, where $G$ acts trivially on $R^*$, the twisted group ring $R^\alpha[G]$ of $G$ over $R$ with respect to $\alpha$ is the free $R$-module with basis $\{u_g\}_{g \in G}$ where the multiplication is defined via
\[u_g u_h = \alpha(g, h) u_{gh} \ \ \text{for all} \ \ g,h \in G \]
and any $u_g$ commutes with the elements of $R$.
Note that the ring structure of $R^{\alpha}[G]$ depends only on the cohomology class $[\alpha]\in H^2(G,R^*)$ of $\alpha$ and not on the particular $2$-cocycle. Of importance is that there is a 1-1 correspondence between
$\alpha$-projective representations of $G$ and $R^\alpha[G]$-modules. As such projective representations, although not explicitly used, are recurrent objects behind the scenes (see \cite[\S 2.1]{ginosar2012semi} for detailed explanation).

In order to keep the introduction notationally lighter, we will restrict ourselves to abelian extensions (i.e. $N$ is abelian in~\eqref{extension intro}). However from \Cref{sectie prelim} till \Cref{section correlations} we will work with general extensions. When $N$ is abelian, the extension (\ref{extension intro}) corresponds\footnote{The necessary background is briefly introduced in \Cref{sectie prelim}.} to a cohomology class $[\alpha] \in H^2_{\sigma}(G, N)$ where $\sigma$ is the action of $G$ on $N$. Via $\sigma$ the group $G$ acts on the set $\Lin(N,F)$ of linear $F$-characters of $N$ and we denote the orbit space by $\Lin(N,F)/G$. Now, for a $G$-invariant linear character $\chi$ of $N$ over $F$, the transgression of $\chi$ with respect to $\alpha$ is a $2$-cocycle $T_{\alpha}(\chi) \in Z^2(G, F^*)$ which is defined by $T_{\alpha}(\chi)(g,h) := \chi(\alpha(g,h)).$ Via so-called inflation one can extend a cocycle of $G$ to one of $\Gamma$, see \Cref{sectie prelim} for details.

\begin{maintheorem}[\Cref{decomp voor elke abelse en centrale extensie} and \Cref{decomp voor elke extensie}]\label{Theorem A intro}
With notations as above, $N$ abelian and $[\beta]\in H^2(\Gamma, F^*)$ inflated from $G$. We have that 
\begin{equation}\label{decomp twisted grp ring intro}
F^{\beta}[\G] \cong \bigoplus_{[\chi] \in \Lin(N,F)/G} (F(\chi) E_{\chi}) * G
\end{equation}
for some concrete idempotents $E_{\chi}$ and explicit skewing and twisting of the crossed product $(F(\chi) E_{\chi}) * G$. In particular, if $\chi \in \Lin (N, F)^{G}$ is a $G$-invariant character, then\footnote{With $\beta . T_{\alpha}(\chi)$ is meant the $2$-cocycle of $G$ with values in $F(\chi)^*$ defined pointwise, i.e. $(\beta . T_{\alpha}(\chi))(g,h) = \beta(g,h). T_{\alpha}(\chi)(g,h) = \beta(g,h) . \chi(\alpha(g,h))$.} 
$$ (F(\chi) E_{\chi}) * G \cong F(\chi)^{\beta . T_{\alpha}(\chi)}[G].$$
\end{maintheorem}
In particular we recover the case where $N$ is central, obtained in \cite[Theorem 5.3]{MaScLast} by Margolis and Schnabel. However even in that case our methods give a new, more explicit, proof. 
It is interesting to pause a second on the classical case when $\beta$ is  trivial, i.e. $F^{\beta}[\G] = F\G$ is a group ring. The above theorem then tells that even if one is solely interested in non-twisted group rings, one should still study twisted group rings over finite extensions of the chosen number field $F$. In particular, at this point one has another point of view on the first step alluded to at the start of the introduction. 

Important for the applications later on is that the decomposition (\ref{decomp twisted grp ring intro}) is 
not simply an abstract one. Among others, the projections $p_{\chi}$ onto the direct summands have a 
down-to-earth description. For this we need to fix a section $\mu: G \rightarrow \G$ of $\lambda$ in (\ref{extension intro}). 
For example, for $\chi \in \Lin(N,F)^{\G/N}$, \Cref {ring homom from Tra} says that the projection $p_{\chi}$ viewed over $R$ agrees with the ring epimorphism 
$$
\Psi _{\chi, \beta}:R^{\beta}[\Gamma] \rightarrow R[\chi]^{\beta . T_{\alpha}(\chi)}[G]: r \, u_{n . \mu (g)} \mapsto r\chi (n) v_{g}
$$
where the sets $\{ u_h \mid h \in \G \}$ and $\{ v_g \mid g \in G \}$  are the bases of the mentioned twisted group rings.

\noindent This ring morphism induces a group morphism 
$$\wt{\Psi}_{\chi, \beta}:\U (R^{\beta}[\Gamma]) \rightarrow \U (R[\chi]^{\beta . T_{\alpha}(\chi)}[G]).$$ 
Another reason that (\ref{decomp twisted grp ring intro}) is a reasonable alternative for Wedderburn-Artin's decomposition is that the kernel and cokernel can be worked with.

\begin{maintheorem}[\Cref{size and comparission of coker} and \Cref{size and comparission of ker}]
Let $\G$ be some extension as in (\ref{extension intro}), $[\beta]\in H^2(\Gamma, F^*)$ inflated from $G$ and $\chi \in \Lin (N, F)^{G}$. Also let $R$ be an order in the number field $F$. Then :\footnote{Again, in order to avoid more notations in the introduction, some parts of the statements are left vague.} 
\begin{enumerate}
\item $\Coker (\wt{\Psi}_{\chi, \beta})$ is finite.
\item If $N$ is central, then $\{ \text{torsion units in } \ker(\wt{\Psi}_{\chi, \beta})\} = \{ \chi(a)^{-1}a \mid a \in N  \}$.
\end{enumerate}
Moreover, we obtain conditions for $\ker (\wt{\Psi}_{\chi, \beta})$ to be finite and some computational reduction for determining $|\Coker (\wt{\Psi}_{\chi, \beta})|$.
\end{maintheorem}

Along the way we obtain a version for twisted group algebras of certain classical theorems of Higman and Berman-Higman (see for example \cite[Proposition 1.5.1 and Theorem 1.5.6]{EricAngel1}). More precisely, in \Cref{finitetwisted} we describe when the unit group of a twisted group ring is finite and in \Cref{Berman-Higman twisted} we show that torsion units must have trace zero. Also, in \Cref{proposition on isom of tga} we answer a question of Margolis and Schnabel \cite[Remark 3.2.]{MaScLast}, in case of a torsion $2$-cocycle, on when $F^{\alpha}[G] \cong F^{\alpha^{j}}[G]$.

\vspace{0,2cm}
\noindent{\it Units in twisted group rings and a full description in the elementary abelian case}
 
By \Cref{Theorem A intro}, describing the unit group of twisted group rings is interlaced with the classical problem of describing $\U(R\G)$. In \Cref{sectie twisted bicyclic en de rest} we launch the investigations of generic constructions of units in twisted group rings and investigate generators of $\Coker (\wt{\Psi}_{\chi, \beta}).$
 
More precisely, we consider $\U(R^{\gamma}[G])$ where $R$ is the ring of integers in a cyclotomic field $F =\Q(\zeta_n) $, 
with $\zeta_n$ some primitive root of unity, and $[\gamma] \in H^2(G, R^*)$. Firstly, we construct 
in \Cref{algemene def twisted bicyclic} a class of units $\mathcal{H}_{\gamma}(G)$ which truly makes use of the twisting 
$\gamma \in Z^2(G,R^*)$. These elements can be thought of as some deformations of the classical bicyclic units in 
non-twisted group rings. An intriguing and crucial feature of these elements is that the generators live in 
$ \Coker \left( \wt{\Psi}_{\chi}: \U (R \G) \rightarrow \U (R^{T_{\alpha}(\chi)}[G]) \right)$, 
see \Cref{non-trivial elements in cokernel} and \Cref{min twists in cokernel question}.

In other words, the generators of $\mathcal{H}_{\gamma}(G)$ are intrinsic to twisted group rings. In fact, the newly constructed $H$-units arose as the pullback along the transgression map $\Psi_{\chi}$ of words in the generators of $\mathcal{H}_{\gamma}(G)$ for $G$ an elementary abelian $2$-group. The difficulty hereby is two-fold: (i) the inverse under $\Psi_{\chi}$ of a unit is usually not a unit due to $1 \neq \ker(\wt{\Psi}_{\chi})$; (ii)  the generators of $\mathcal{H}_{\gamma}(G)$ are not attained. In fact the length of the words depend on $|\Coker (\wt{\Psi}_{\chi})|$ which on its turn depends on $N = \ker(\lambda: \Gamma \twoheadrightarrow G).$

An important step for investigating units in twisted group rings is \Cref {Jespers-Leal generalized} and \Cref{Jespers-LEal for twisted bicyclic} saying that $\langle \Bic (G), \mathcal{H}_{\gamma}(G) \rangle$ contains enough elementary matrices of each simple component. This generalizes the much used theorem of Jespers and Leal \cite{JesLea2}. 

Finally, in Sections \ref{section elem ab description} and \ref{section desciprtion D8 times elem abelian 2-group} 
we apply all the above machinery to study the case that $\G = G \times C_2^m$ for some $m$ and some 
$[\gamma] \in H^2(\Gamma, \Z ^*)$ inflated from $[\delta]\in H^2(G, \Z ^*)$. Note that the number of simple components of 
the group algebra increases exponentially with $m$, which for the investigations of unit groups makes these extensions more subtle  as it first looks likes. 
In \Cref{subsection reduing elem ab 2-subgroups} we give a description of 
$\U(\Z^{\gamma}[G\times C_2^m])$ in terms of $\U(\Z^{\delta}[G])$, see 
\Cref{th reducing off an elemantary abelian subgroup}.  As an application, we are able to deduce the following result. Recall that $G$ is said to have a  normal complement in $\U (\Z [G])$ if $\U (\Z [G])=N\rtimes (\pm G)$ for some normal subgroup $N$. If a torsion free complement $N$ exists then the integral isomorphism problem has a positive answer for $G$ (see for example \cite[Proposition (30.4)]{SehgalBook93}). Also recall that $G$ satisfies the Higman subgroup property if every finite subgroup $H \leq V(\Z G)$ is isomorphic to a subgroup of $G$ (this property was asked for the first time in Graham Higman's thesis).

\begin{maintheorem}[\Cref{Properties that get inherited with times elem ab}]\label{main theorem over extending with elem ab}
Let $G$ be a finite group and $[\gamma] \in H^2(G\times C_2^m, \Z^*)$ inflated from a cohomology class $[\delta] \in H^2(G, \Z^*)$. If $G$ has a (torsion-free) complement in $\U (\Z^{\delta}[G])$ or satisfies the Higman subgroup property, then the same holds for $G \times C_2^m$ and $\U (\Z^{\gamma}[G\times C_2^m])$.
\end{maintheorem}

Subsequentely, in \Cref{section elem ab description}, all the protagonists are computed explicitly in the case that 
$\Gamma =D_8\times C_2^m$ (where $D_8$ denotes the dihedral group of order $8$) and $G=C_2 ^{m+2}$. 
More precisely, we consider $D_8$ as some extension $[\alpha] \in H^2(C_2 \times C_2, C_2)$ and look at the 
projections $\wt{\Psi}_{\chi} : \U( \Z [D_8 \times C_2^m]) \rightarrow \U (\Z^{T_{\alpha}(\chi)}[C_2^{m+2}])$.  
Together with the general description of $\ker(\wt{\Psi}_{\chi})$ and \Cref{main theorem over extending with elem ab} 
we are able to pullback a precise description of $\U( \Z^{T_{\alpha}(\chi)}[C_2^{m+2}])$ obtained 
in \Cref{the groups U_i for elemen ab}. All together, the main achievement here is an unexpected uniform 
description of $\U( \Z [D_8 \times C_2^m])$ for all $m$. The case that $m=0,1,2$ has been dealt with in 
earlier papers \cite{MR1221725,Jespers1995,Li,MarSeh}. 

Along the article we formulated several questions we believe to be of interest. 
Some we would like to attract the attention to are
\Cref{question torsion and nilp interplay}, \Cref{question idemp in case of self-extensions}, 
\Cref{question on abelianisation} and \Cref{conjecture torsion units in kernel transgression}.

In conclusion, (a subclass of the) $H$-units arose by applying all the machinery above to an extension of $C_2 \times C_2$. These units might extend the bicyclic units with infinite index and are particularly useful in presence of an exceptional component of the form $\Ma_2(\Q)$. However we have not yet enough generic constructions that, without any condition on the $2\times 2$ simple components, give a finite index subgroup for any finite group. Nevertheless, extending other groups in \cite[Appendix A]{BJJKT}, the methods developed in this article should yield other generic constructions. Besides, this work can also simply be seen as an invitation to the study of units in twisted group rings as a problem of independent interest.

\vspace{0,4cm}
\noindent {\it Notational conventions:} 
\begin{enumerate}
    \item All rings, denoted $R$, will be assumed unital and associative. 
      \item $\mathcal{U}(R)$ denotes the unit group of $R$. If $\U(R)$ appears as the coefficients of a cohomology group, i.e. $H^{j}(G, \U(R))$, we will instead write $R^{*}$ (thus $H^{j}(G, R^*)$). 
      \item If $f : R \rightarrow S$ is a ring homomorphism, then we denote the induced map on the unit groups by $\wt{f}: \mc{U}(R) \rightarrow \mc{U}(S)$.
    \item Except stated otherwise, with `an order $R$' we will mean a $\Z$-order. 
    \item We will use the convention $g^h = h^{-1}g h$ and $\text{conj}(h)(g) = g^h$ for conjugation. 
    \item A commutator is $[g,h] = g^{-1}h^{-1}gh.$
    \item $cl_{\G}(H)$ is the normal closure of some subgroup $H$ in the larger group $\G.$
\end{enumerate}

\vspace{0,2cm}
\subsection*{Acknowledgment.}
We thank Leo Margolis for numerous important comments on an earlier version of this paper. We would also like to express our gratefulness to the anonymous referee for many suggestions. We are further indebted to \'Angel del R\`io for providing us the proof of \Cref{label about dividing order} and to Yuval Ginosar for sharing with us \Cref{the group depends on the cocycle} and valuable discussions. Finally we warmly thank Marco Trombetti for useful conversations and Alexander Hulpke for computational help which both strongly helped for intuition in \Cref{Sectie ab conjectures}. The first author would also like to thank \v{S}pela \v{S}penko for all the encouragement during the long walk through the dessert that this project represented. \newline During the early stages of this project, in $2017$, the first author was a visitor at the Technion - Israel Institute of Technology and he would like to thank the institute for the pleasant working conditions. Thereafter the third author visited several times the VUB and he thanks the host institution and the Department of  Mathematics and Data Science at VUB.

\section{Cohomological preliminaries}\label{sectie prelim}
We now recall the required minimum on group cohomology for which details can be found in the book \cite{BrownBook}. Given a finite group $G$ with an action $\varphi: G \rightarrow \Aut(A)$ on an abelian group $A$, extensions\footnote{Thus $A$ is normal and $G \cong \G/A$. In other words (\ref{De start CE setting prelim}) is a short exact sequence.} 
\begin{equation}\label{De start CE setting prelim}
1\rightarrow A \rightarrow \Gamma \overset{\lambda}\rightarrow G\rightarrow 1.
\end{equation}
of $A$ by $G$, whose induced action is $\varphi$, are parametrised by the second 
group cohomology group $H_{\varphi}^2(G,A)$. If the action of $G$ on $A$ is trivial, 
in which case we write $H^2(G,A)$,  it is well known that the cohomology group parametrizes the central extensions (i.e. $\mc{Z}(A) \subseteq \Gamma$). 
Given a commutative ring $R$ and a $2$-cocycle 
$\alpha \in Z^2(G,R^*)$ one can define 
the {\it twisted group ring} $R^{\alpha}[G]$ whose basis we will denote by 
$\{ u_g\mid g\in G\}$. Up to $R$-algebra isomorphism, $R^{\alpha}[G]$ does not depend on $\alpha$ but only on the cohomology class $[\alpha]\in H^2(G,R^*)$. 
A $2$-cocycle $\alpha$ is said to be  {\it normalized} if $u_1$ is the identity element 
of $R^{\alpha}[G]$, i.e. if $\alpha (1,g)=1=\alpha (g,1)$ for all $g\in G$. 
Notice also that then the ring $R=Ru_1$ is central in the twisted group ring $R^{\alpha}[G]$.

\begin{example}
A cohomology class $[\alpha]\in H^2(G,\Z^*)$ over $\Z$ corresponds to a central extension
\begin{equation}\label{eq:centralext}
[\alpha]:\quad 1\rightarrow C_2\rightarrow \Gamma \rightarrow G\rightarrow 1.
\end{equation}
As a result, for an abelian group $G$ generated by $\{g_1,g_2,\ldots ,g_r\}$, $[\alpha]\in H^2(G, \Z^*)$ is determined by the values in $\{\pm 1\}$ of
\begin{equation}\label{eq:defrelations}
u_{g_i}^{\circ (g_i)}, \text{ }1\leq i\leq r, \text{ and the commutators }[u_{g_i},u_{g_j}], \text{ }1\leq i<j\leq r,
\end{equation}
in the corresponding twisted group ring $\Z^{\alpha}[G]$.
\end{example}

Let  $Q$ be a normal subgroup of $G$, let $M$ be an abelian group which we equip with a {\it trivial} $G$-action (i.e. $M$ is a trivial $G$-module) and let $\pi: G \rightarrow G /Q$ be the
quotient map. Then, for any $\delta \in Z^{i}(G/Q, M) \subset \Map(G/Q, M)$ we can define
$\gamma \in Z^{i}(G,M)$ by
$$\gamma (x_1,\ldots,x_i)=\delta(\pi (x_1), \ldots, \pi (x_i)).$$
The map from $Z^{i}(G/Q,M)$ to $Z^{i}(G,M)$ sending $\delta$ to
$\gamma$ induces
a map
$$\Inf:H^{i}(G/Q,M)\rightarrow H^{i}(G,M)$$
which is called the \textit{inflation map}. Sometimes we will want to emphasize $Q$ and use the notation $\Inf_Q$. Notice also, that for any subgroup $H$ of $G$, there is a natural restriction map from $ H^i(G,M)$ to\footnote{If $H=Q$ is normal, then $\Ima(\Res)$ is easily seen to be contained in the subgroup $H^{i}(Q,M)^{G/Q}$ of $G/Q$-invariant cocycles.} $H^i(H,M)$. In case of $M$ a trivial $G$-module, the Lyndon-Hochschild-Serre spectral sequence yields a very concrete exact sequence between the lower cohomology groups. 
 \begin{lemma}[Inflation-Restriction exact sequence]\label{Lyndon-Hochschild-Serre first five terms}
Let $Q$ be a normal subgroup of $G$ and $M$ a trivial $G$-module. Then one has an exact sequence
$$0\rightarrow H^1(G/Q,M) \overset{\Inf}\rightarrow H^1(G,M) \overset{\Res}\rightarrow H^1(Q,M)^{G/Q}\overset{d}\rightarrow H^2(G/Q,M) \overset{\Inf}\rightarrow H^2(G,M)$$
where the first connecting map $d$ is the transgression map and $H^1(\cdot,\cdot) = \Hom(\cdot,\cdot)$.
 \end{lemma}
 
In case of the abelian extension (\ref{De start CE setting prelim}), i.e. taking $\Gamma$ and $A$ (as respectively $G$ and $Q$) the first transgression map is a morphism from $\operatorname{Hom}(A,M)^{G}$ to $H^2(G,M)$ whose definition we now recall. For $\chi \in \operatorname{Hom}(A,M)^G$ one can define a $2$-cocycle
$\Tra(\chi)\in Z^2(G, M)$ via
$$\Tra(\chi)(g_1,g_2)=\chi (\alpha (g_1,g_2))$$ for any $g_1,g_2\in G$. The cohomology class $[\Tra(\chi)]$ does not depend on the choice of $\alpha \in [\alpha]$.

\begin{definition}\label{def:Tra}
	With the above notations, the map 
	$$\Tra_{\alpha}: \operatorname{Hom}(A,M)^G \rightarrow H^2(G,M): \chi \mapsto T_{\alpha}(\chi):=  [\Tra(\chi)]$$  is called the (first) \emph{transgression map} associated to $\alpha$.
\end{definition}

Using notations as in (\ref{De start CE setting prelim}), let $Q$ be a normal subgroup of $\Gamma$ such that $A \cap Q = 1$. Then $Q \cong \lambda(Q)$ and from now on {\it we will implicitly identify the both} and speak about $G/Q$. 

\begin{remark}\label{CE again a CE}
Notice that $G/\lambda(Q) \cong G/Q$ and $AQ/Q \cong A/A\cap Q = A$. Hence, if $\alpha=\Inf_Q(\gamma)$ then 
$\Gamma/Q$ is the central extension
\begin{equation}\label{CE over N}
1\rightarrow A \rightarrow \Gamma/Q \overset{\lambda}\rightarrow G/Q \rightarrow 1,
\end{equation}
corresponding to $[\gamma]$. 
\end{remark}

In \Cref{subsection corrolations} we will consider various such normal subgroups $Q$ and the (co)kernel of the transgression map associated to $\G/Q$. Therefore we also write $\Tra_Q$ for the transgression map associated to the abelian extension $(\ref{CE over N})$. When the extension or $Q$ are clear from the context we simply write $\Tra$.

We record a corollary of $\alpha(g,1) = \alpha(1,1) = \alpha(1,g)$ for all $g\in \Gamma$, entailed by the $2$-cocycle condition, which will repeatedly be used without further notice.
\begin{corollary}\label{cor imabe inf yields central}
Let $G$ be a group, $Q$ a normal subgroup and suppose that $[\alpha]\in H^2(G,R^*)$ is inflated from a cohomology class $[\gamma]\in H^2(G/Q,R^*)$. Then, for $g \in G$ and $x \in Q$, $u_g$ and $u_x$ commute in $R^{\alpha}[G]$ whenever $g$ and $x$ commute in $G$. In particular, $u_a$ is central in $R^{\alpha}[G]$ for any $a\in \mathcal{Z}(Q)$. 
\end{corollary}
Finally, we mention a lemma that might be known (however, to our knowledge, does  not appear in the literature) and which will be useful in Proposition~\ref{proposition on isom of tga}.
\begin{lemma}\label{lemma prime divisors of alpha}
Let $G$ be a finite group, let $F$ be a field and let $\alpha\in Z^2(G,F^*)$ be of finite order. Assume that $m$ is a number relatively prime to the order of $[\alpha]\in H^2(G,F^*)$. Then there exists a cocycle $\beta \in [\alpha]$ such that $m$ is relatively prime with the order of $\beta$. 
\end{lemma}
\begin{proof}
Consider the greatest common divisor of $m$ and the order of $\alpha$, which we denote by $d$ and it exists as $o(\alpha)$ is assumed to be finite. Denote $o(\alpha) = d\cdot \ell$ with $(d,\ell)=1.$ Since the order of $\alpha$ is finite, notice that $H=\langle \text{Im}(\alpha)\rangle$ is a finite subgroup of $F^*$ and hence it is cyclic. Therefore $H$ can be decomposed as $H=H_1\times H_2$ where $|H_2| = \ell$ and $|H_1|=d$. Hence $\alpha$ admits a natural decomposition as $\alpha=\alpha_{H_1}\times \alpha _{H_2}\in Z^2(H_1,F^*)\times Z^2(H_2,F^*)$. Since $m$ is relatively prime to the order of $[\alpha]$ and $\exp (H^2(H_1,F^*)) \mid |H_1| =d$, $[\alpha_{H_1}]$ is cohomological trivial. Therefore there exists $\lambda :G\rightarrow H_1$ such that $\alpha _{H_1}(g,h)=\lambda (g)\lambda (h)\lambda (gh)^{-1}$ for any $g,h\in G$. Consequently, $\alpha (g,h)=\lambda (g)\lambda (h)\lambda (gh)^{-1}\alpha_{H_2}(g,h)$ that is $\alpha$ is cohomologous to $\alpha _{H_2}$ and clearly $m$ is relatively prime to the order of $\alpha _{H_2}$.
\end{proof}
 {\it \underline{Standing assumptions:}} In the sequel of the paper all cohomology groups will be with respect to a trivial action, except stated otherwise. Also, given a cohomology class $[\alpha]$, we always assume a chosen representative $\alpha$ which we assume to be normalized.

\section{Decomposition of twisted group algebras for general extensions and transgression}\label{Sectie decompositie}

In this section, given a central extension as in (\ref{De start CE setting prelim}), we will recover the decomposition \cite[Theorem 5.3.]{MaScLast} of the twisted group algebra $F^{\beta}[\G]$ for any $[\beta]\in H^2(\Gamma, F^*)$ inflated from $G$ and $F$ a field with $\Char(F) \nmid |\G|$. Note that loc. cit. was a generalisation of \cite[Theorem 3.2.9.]{KarProj}. Our aim is to use other methods than \cite{KarProj, MaScLast} which allow to work with general extensions (i.e. with $A$ not necessarily central or even abelian) and to interpret in \Cref{background transgression} the projection maps as a kind of transgression morphisms. The most general structural result will be \Cref{decomp voor elke extensie}, but \Cref{decomp voor elke abelse en centrale extensie} which focuses on abelian extensions will be more precise. These results proves \Cref{Theorem A intro} from the introduction.

\subsection{Concrete decomposition}\label{subsectie decompositie}
Consider an arbitrary short exact sequence
\begin{equation}\label{De algemene extensie}
1\rightarrow N \rightarrow \Gamma \overset{\lambda}\rightarrow G\rightarrow 1.
\end{equation}
Fix also a section $\mu$ of $\lambda$. In particular $\alpha(g,h) := \mu(g)\mu(h)\mu(gh)^{-1} \in N$ for all $g,h \in G$ and conjugation with $\mu(g)$ gives an outer automorphism of $N$:
\begin{equation}\label{2-cocycle condition arb extension}
\sigma: G \rightarrow \Out (N) := \Aut (N) / \Inn (N): g \mapsto \text{conj}(\mu(g)^{-1}).
\end{equation}

\noindent It is well known that the associativity of $\G$ means that $\alpha(g,h) \alpha(gh, k) = \sigma_g(\alpha (h,k)) \alpha(g,hk).$

Let {\it $F$ be any field with $\Char(F) \nmid |N|$} and $[\beta] \in \im (\Inf: H^2(G, F^*) \rightarrow H^2(\G, F^*))$. Note that 
\begin{equation}\label{naive decomp}
F^{\beta}[\G] \cong (FN) * G =\sum_{g\in G} (FN) v_{g}
\end{equation}
is a crossed product of $G$ over $FN$ with the following 
operations :\footnote{In $F^{\beta}[\G]$ an arbitrary $F$-basis element has the form 
$u_a u_{\mu(g)}$ for some $a \in N, g\in G$. We are identifying it with $a v_g$ 
where $a$ is considered as a coefficient in $F[N]$. Implicitly we are extending the 
action $\sigma$ of $G$ on $N$ to $FN$.} 
$$v_g v_h =\beta(g,h) \alpha (g,h) v_{gh} \qquad \text{(twisting)}$$
and, for $x \in FN$,
  $$v_g x =\sigma_g (x) v_g. \qquad \text{(skewing)}$$
Before giving a more precise description, recall that there is a bijective correspondence between $\PCI (FN)$, the set of the primitive central idempotents of $FN$, and the simple $FN$-modules up to isomorphism. Moreover, via Galois descent, $\PCI (FN)$ can be computed from the set of the complex irreducible representations $\Irr(N,\C)$. 
Also recall that $G$, interpreted as  $\G /N$, acts on $\PCI (FN)$. The orbit space will be denoted $\PCI (FN)/G$. 

Let $K$ be a splitting field of $N$ containing $F$ and let $\Irr (N,F)$ be a set of $K$-characters of $N$ containing the character of exactly one composition factor of $V \ot_F K$ for each simple $FN$-module $V$. Recall that the choice of the composition factor does not affect the character by \cite[Theorem 3.3.1]{EricAngel1}.

By definition, one has a bijection between  $\Irr (N,F)$ and  the simple $FN$-modules up to isomorphism,  and thus with $\PCI (FN)$. Denote by $e_{\chi}$ the unique primitive central idempotent of $FN$ corresponding to $\chi \in \Irr(N,F)$ (see Section 3.3 in \cite{EricAngel1}). Note that $G$ also acts on the irreducible $K$-characters of $N$ via $\chi^g(n) := \chi(\sigma_g(n))$. The bijection can be chosen such that $e_{\chi}^g := e_{\chi^g}$ and hence $G$ acts on $\Irr (N,F)$.

Next, define 
 $$\Lin (N,F) := \{ \chi \in \Irr(N,F) \mid \chi(1) = 1 \}.$$
It is easily verified that 
$$\chi \in \Lin (N,F) \text{ if and only if } FNe_{\chi} \text{ is commutative.}$$
Moreover, for $\chi \in \Lin (N,F)$ one has that
\begin{equation} \label{effect of projection}
n e_{\chi} = \chi(n) e_{\chi} \quad \text{ and } 
\quad FNe_{\chi} \cong F(\chi).
\end{equation}
where $F(\chi)$ is the smallest field containing $F$ and $\im (\chi)$.
\begin{proposition}\label{decomp voor elke extensie}
With notation as above, we have the following:
$$F^{\beta}[\G] \cong \bigoplus_{[\chi] \in \Irr(N,F)/G} \big( (FN E_{\chi}) * G \big),$$
where $E_{\chi} := \sum_{\chi^g \in [\chi]} e_{\chi}^{g}$ and 
the skewing of $(FN E_{\chi}) * G$ is given by $\sigma$ and the 
twisting by\footnote{The notation needs clarification: we implicitly consider both $\beta$ 
and $\alpha$ as having image in $FN$ and hence 
$(\beta \cdot \alpha)(g,h) := \beta(g,h) . \alpha(g,h)$. Subsequently, 
$E_{\chi} . (\beta \cdot \alpha)$ also multiplies the result with the central 
idempotent $E_{\chi}$.} $E_{\chi} . (\beta \cdot \alpha)$. 
Moreover, if \footnote{We denote by $\Lin(N,F)^{\G / N}$ the set of $\G/N$-invariant linear characters.} $\chi \in \Lin(N,F)^{\G / N},$
\begin{equation}\label{component for invariant character}
(FNE_{\chi}) * G \cong F(\chi)^{\beta. T_{\alpha}(\chi)}[G].
\end{equation}
\end{proposition}
\begin{proof}
Since $[\beta] \in \im (\Inf _N)$ we have that $F^{\Res(\beta)}[N] = FN =\bigoplus_{\chi \in \Irr(N,F)} FN e_{\chi}.$
In general the orthogonal idempotents $e_{\chi}$ are not central in 
$F^{\beta}[\G]$, however the element $E_{\chi}$ (being the sum of the orbit 
under conjugation by $G$) is.   These elements are again two-by-two 
orthogonal idempotents, hence (\ref{naive decomp}) can be rewritten as 
$F^{\beta}[\Gamma] =\bigoplus_{[\chi] \in \Irr(N,F)/G} (FN E_{\chi}) *G$. 
Since $E_{\chi}$ is central we see that indeed only the twisting changes and 
in the way asserted. This proves the first part of the result.

Next note that $\chi \in \Irr(N,F)^{\G/N}$ being an invariant character exactly means that the orbit of $e_{\chi}$ is a singleton and hence $E_{\chi} = e_{\chi}$. Hence if moreover $\chi \in \Lin(N,F)$, then by (\ref{effect of projection}) $FNE_{\chi}\cong F(\chi).$ Also, from the earlier reformulation of the associativity of $\G$, it readily 
follows that  $\alpha_{\chi} := \Tra_{\alpha}(\chi)$ satisfies the $2$-cocycle condition
$$\alpha_{\chi}(g,h) \alpha_{\chi}(gh, k) = \alpha_{\chi}(h,k) \alpha_{\chi}(g,hk).$$
In particular, $(FNE_{\chi}) * G$ is a twisted group algebra where, by the above, 
the twisting is indeed $E_{\chi} . \beta \cdot  \alpha = \beta \cdot \Tra_{\alpha}(\chi)$.
\end{proof}
In order to obtain a more insightful decomposition we {\it now assume that $N$ is abelian.} In this case the extension (\ref{De algemene extensie}) 
corresponds to the cohomology 
class\footnote{Since $N$ is not necessarily central we might thus have a non-trivial action on the coefficients 
given by $\sigma$. } $[\alpha] \in H^2_{\sigma}(G, N)$ and 
$\Irr (N,F) \subseteq \Hom (N, \overline{F}^*)$ where $\overline{F}$ denotes the algebraic closure of $F$. Moreover, since by assumption $\Char(F) \nmid |N|$, the theorem of 
Perlis and Walker \cite[Theorem 3.3.6.]{EricAngel1} tells us that
\begin{equation}\label{Perlis-Walker decomposition}
FN = \bigoplus_{d\mid |N|} F(\zeta_d)^{\op a_d},
\end{equation}
where $a_d = k_d \frac{[\Q [\zeta_d] :\Q]}{[F[\zeta_d]:F]}$, with $k_d$ the number 
of cyclic subgroups of $N$ of order $d$ and $\zeta_d$ a primitive $d$-th root of 
unity. \Cref{decomp voor elke extensie} now readily translates into the following:

\begin{theorem}\label{decomp voor elke abelse en centrale extensie}
With notations as above we have that:
\begin{enumerate}
\item If $N$ is abelian, then 
$$F^{\beta}[\G] \cong \bigoplus_{[\chi] \in \Lin(N,F)/G} \big( \bigoplus_{\chi\prime \in [\chi]}  F(\chi\prime) e_{\chi\prime} * G \big) = \bigoplus_{[\chi] \in \Lin(N,F)/G} (F(\chi) E_{\chi}) * G $$
with the skewing of $(F(\chi) E_{\chi}) * G$ given by $\sigma$ and the twisting by $E_{\chi} . (\beta \cdot \alpha)$.
\item If $N$ is central\footnote{As mentioned earlier, this case was already obtained in \cite[Theorem 5.3]{MaScLast}.}, then 
$$F^{\beta}[\G] \cong \bigoplus_{\chi \in \Lin(N,F)} F(\chi)^{\beta . T_{\alpha}(\chi)}[G].$$
\end{enumerate}
\end{theorem}
\begin{proof}
The first part is a direct application of \Cref{decomp voor elke extensie} and remarking that now $\Lin(N,F) = \Irr(N,F)$ and $F(\chi) = F(\chi\prime)$ for two conjugated characters $\chi$ and $\chi\prime$. When $N$ is central then $\Irr(N,F) = \Lin (N,F)^{\G/N}$ and $e_{\chi}^g = e_{\chi}$. Therefore the second part follows from the first and the last assertion in \Cref{decomp voor elke extensie}.
\end{proof}

\subsection{Transgression and natural morphisms as projections}\label{background transgression}

Consider again the general extension (\ref{De algemene extensie}). 
The projection in \Cref{decomp voor elke extensie} of $F^{\beta}[\G]$ on $FNE_{\chi} * G$, 
denoted $p_{\chi}$, is given by multiplying with the central idempotent $E_{\chi}$. 
If $\chi \in \Lin(N,F)^{\G /N}$ then this idempotent $E_{\chi}$ and, in particular, the isomorphism (\ref{component for invariant character}) can be made explicit and, moreover, we recover some classical constructions (see \Cref{Trans and nat hom as example}). More generally, let {\it $R$ be a domain} and $F$ its field of fractions. For $\chi \in \Lin(N,F)^{\G/N}$ we need its $R$-linear extension
$$\chi_R: RN \rightarrow R[\chi] : \sum_{a \in N} r_a a \mapsto \sum_{a \in N} r_a \chi(a) $$
which is an $R$-algebra map whose kernel we denote 
\begin{equation}\label{definition Ichi}
I_{\chi} := \ker (\chi_R).
\end{equation}

\begin{proposition}\label{ring homom from Tra}
Let $\G, N, G, F$ and $\mu$ be as in \Cref{subsectie decompositie}, 
$[\beta] \in \im (\Inf: H^2(G, R^*) \rightarrow H^2(\G, R^*))$ 
and $\chi \in \Lin(N,F)^{\G/N}$.  Then, the projection $p_{\chi}$ restricted to $R^{\beta}[\Gamma]$ 
agrees with the map defined by (for $n\in N$ and $g\in G$)
\begin{equation}\label{generalized transgression} 
\Psi _{\chi, \beta}:R^{\beta}[\Gamma] \rightarrow R[\chi]^{\beta . T(\chi)}[G]: 
r \, u_{n . \mu (g)} \mapsto r\chi (n) v_{g}.
\end{equation}
which is a ring epimorphism with $\ker(\Psi_{\chi, \beta}) = R^{\beta}[\Gamma] I_{\chi} = R^{\beta}[\Gamma] (1-e_{\chi}) \cap R^{\beta}[\Gamma]$. In particular,
$$R^{\beta}[\Gamma] / R^{\beta}[\Gamma] I_{\chi} \cong R^{\beta \cdot T(\chi)}[G].$$
Moreover, if $\im(\chi) \subseteq R^*$, then $I_{\chi}$ is a free $R$-module with $R$-basis $\{u_a - \chi(a) u_1 \mid 1 \neq a \in A \}$.
\end{proposition}
When $\beta$ is trivial we will simply write $\Psi_{\chi}$. If $R$ is a field and $N$ central then most of the previous result is known\footnote{If $\chi$ is the trivial character and $N$ arbitrary, then the statement has to be compared with \cite[Lemma 3.2.12]{KarProj}.} (e.g. if $\beta$ is trivial in \cite[Theorem 3.2.8.]{KarProj} and for $\beta \in \im(\Inf_N)$ in \cite[Proposition 5.2.]{MaScLast}).

\begin{example}\label{Trans and nat hom as example}
\begin{itemize}
\item If $\chi$ is the trivial character, i.e. $\chi(n)=1$ for all $n \in N$, then $\chi_R$ is the augmentation map of $RN$. In that case, writing $[\beta] = \Inf ([\gamma])$,
$$\Psi_{\chi, \beta} : R^{\beta}[\Gamma] \rightarrow R^{\gamma}[G]: r\,  u_{n .  \mu(g) }\mapsto r\,  v_g$$
is the so-called natural homomorphism with respect to $N$ (i.e. induced from the canonical map from $\Gamma$ to $G$, denoted $\omega_N$. If $\beta$ is trivial it also is called the relative augmentation with respect to $N$). We will denote this specific case by $\omega_{\beta,N}$.

\item If $\Gamma$ is a central extension, i.e. $N \subseteq \ZZ( \G)$, then $\Lin(N,F)^{\G/N} = \Irr(N,F)$ which is a subgroup of $\Hom (N, \ov{F}^*)$ and $\Psi_{\chi, \beta}$ is the ring morphism associated to the transgression map $\Tra$ as in \cite[Section 5]{MaScLast}. Therefore we call the morphism (\ref{generalized transgression}) the {\it generalized transgression morphism}.
\end{itemize}
\end{example}

The reader may wish now to look at \Cref{example order 16 everything worked out} where we will give an example of a group of order $16$, along with its decomposition in twisted group algebras and the associated generalized transgression maps.

\begin{proof}[Proof of \Cref{ring homom from Tra}.]
A direct calculation verifies that  $\Psi_{\chi, \beta}$ is multiplicative when $\chi$ 
is an invariant character. Concretely, 
$ (u_{n_1} u_{\mu(g_1)}) . (u_{n_2} u_{\mu(g_2)})  = \beta(g_1,g_2)
u_{n_1}u_{\sigma_{g_1}(n_2)} u_{\mu(g_1g_2)}$ which is mapped by $\Psi_{\chi, \beta}$ 
to 
$\beta(g_1,g_2) \chi(n_1\sigma_{g_1}(n_2)) v_{g_1g_2} = \beta(g_1,g_2) \chi(n_1)
\chi(n_2) v_{g_1g_2}$. The latter is $\Psi_{\chi, \beta} 
(u_{n_1} u_{\mu(g_1)}). \Psi_{\chi, \beta} (u_{n_2} u_{\mu(g_2)}) $, 
as needed. Now it is clear that it is an epimorphism.

Next, because $\chi \in \Lin(N,F)^{\G/N}$, multiplying with the idempotent $E_{\chi}$, i.e. the value of the projection, was already explicitly mentioned in (\ref{effect of projection}). By comparing we see that $\Psi_{\chi, \beta}$ indeed agrees with $p_{\chi}$. 

Concerning the kernel, it follows from \Cref{decomp voor elke extensie} that 
$\ker (\Psi_{\chi, \beta}) =  R^{\beta}[\Gamma] (1-e_{\chi}) \cap
R^{\beta}[\Gamma].$ 
For the other description, since $\restr{\Psi_{\chi, \beta}}{RN} = \chi_R$, we already have that $R^{\beta}[\Gamma]I_\chi \subseteq \ker(\Psi_{\chi, \beta})$. Also, an element $y$ of $R^{\beta}[\Gamma]$ can be uniquely written as $y =\sum_{g\in G} y_g u_{\mu(g)}$, with each $y_g\in RN$. Since the elements $v_g$ are $R[\chi]$-linearly independent, the concrete form of $\Psi_{\chi, \beta}$ implies that $y\in  \ker(\Psi_{\chi, \beta})$ if and only if each $y_g\in \ker(\Psi_{\chi, \beta})$, or equivalently each $y_g\in I_{\chi}$.  
Consequently, $\ker(\Psi_{\chi, \beta })=\sum_{g\in G}I_{\chi} u_{\mu(g)} \subseteq R^{\beta}[\Gamma] I_{\chi}$, as asserted. 

Finally, suppose $\im(\chi)$ is contained in $R^*$ and denote the free $R$-module generated by the set $\{u_a - \chi(a)u_1 \mid a \in N \}$ by $M$. Clearly $M$ is contained in the kernel $I_{\chi}$ of $\chi_{R}:RN\rightarrow R$. Conversely, if $x = \sum_{a\in N} r_a u_a \in \ker(\chi_R)$, then $\sum_{a \in N} r_a \chi(a)u_1 = 0$ and hence $x =  x - \sum_{a \in N} r_a \chi(a)u_1 \in M$. It follows that  $I_{\chi}=M$.
\end{proof}

In \Cref{cokernel of trans subsection} we will be in the setting of \Cref{CE again a CE} where 
$\chi$ is a (linear) character of the abelian subgroup $A$, but where once we will 
be working with $\Gamma$ and once with some $\Gamma /Q$ for $A \cap Q = 1$. In order 
to distinguish, we will sometimes write 
\begin{equation}\label{map psi wrt to a normal sbgrp Q}
\Psi _{\chi,Q}: R^{\beta}[\Gamma/Q] \rightarrow R^{\beta. T(\chi)}[G/Q]
\end{equation} 
and 
hence assume $[\beta]$ is understood from the context.

\subsection{Refined decomposition in case of an $\alpha$-representation group of $G$}\label{subsectie representation group}
 Often in this article we will be concentrating on a fixed $2$-cocycle $\alpha \in Z^2(G, F^*)$ of $G$. In that case it is useful to consider the following group which will be recurrent:

\begin{equation}\label{The extension of G with values of cocycle}
G_{\alpha} :=  \{ \alpha(a,b)u_c \mid a,b,c \in G \} = \langle u_g \mid g \in G \rangle \leq \U(F^{\alpha}[G]).
\end{equation}

If $\alpha$ is of finite order\footnote{Since $G$ is finite we know that the 
cohomology class $[\alpha]$ has finite order dividing $|G|$, however there has 
not to be a representative of finite order. For example the cohomology class of 
$C_2 = \langle x \rangle$ over $\Q$ defined by $u_x^2= 2$ has order $2$ but 
any $2$-cocycle representative has infinite order. Such a representative however 
exists when the values are in a $|G|$-divisible group.}, then every element 
in $G_{\alpha}$ is of the form $\zeta^{i}u_g$  with $\zeta$ a $o(\alpha)$-primitive 
root of unity. Hence it is a central extension of $\langle \zeta \rangle$ by $G$ 
with $\lambda: G_{\alpha} \rightarrow G : \zeta^{i}u_g \mapsto g.$ Considering 
the canonical section $\mu(g) = u_g$ one has that $\alpha(g,h)= \mu(g) \mu(h) \mu(gh)^{-1}$. Thus $\Hom(\langle \zeta \rangle, F^*) = \langle \chi \rangle$ 
with $\chi(\zeta) = \zeta$ and $\alpha = \Tra(\chi)$. 
Therefore we may now apply \Cref{decomp voor elke abelse en centrale extensie} to recover \cite[Proposition 3.3.8.]{KarProj}. 

\begin{corollary}\label{decomposition of FGalpha}
Let $G$ be a finite group, $F$ a number field and $\alpha \in Z^2(G, F^* )$ a cocycle of finite order. Then
$$F[G_{\alpha}] \cong \bigoplus_{i = 0}^{o(\alpha)-1} F^{\alpha^{i}}[G].$$
\end{corollary}

Note that the isomorphism class of $G_{\alpha}$ depends on the chosen cocycle. It turns out that even for cohomologous cocycles of the same finite order the associated groups might be non isomorphic. The following example was communicated to us by Y. Ginosar.
 
\begin{example}\label{the group depends on the cocycle}
 Denote  
$$G=C_2\times C_2=\langle x\rangle \times \langle y\rangle$$ 
and let $F=\mathbb{Q}(i)$. Next, we wish to define two coycles  
$\alpha, \alpha ^{\shortmid}\in Z^2(G,F^*)$. We will take cocycles which are normalized, that is for any $g\in G$
$$\alpha (g,e_G)=1=\alpha (e_G,g),\quad 
\alpha ^{\shortmid}(g,e_G)=1=\alpha ^{\shortmid}(e_G,g).$$
Further, define
$$\alpha (x,x)=\alpha (x,y)=\alpha (x,xy)=1$$
$$ \alpha(y,y)=1, \, \alpha (y,x)=\alpha (y,xy)=-1$$
$$\alpha (xy,x)=1, \, \alpha (xy,y)=\alpha (xy,xy)=-1.$$
and also 
$$\alpha ^{\shortmid}(x,x)=-1,\quad \alpha^{\shortmid} (x,y)=1,\quad \alpha^{\shortmid} (x,xy)=-1$$
$$ \alpha^{\shortmid}(y,y)=-1, \quad \alpha^{\shortmid} (y,x)=-1,\quad \alpha^{\shortmid} (y,xy)=1$$
$$\alpha^{\shortmid} (xy,x)=1, \quad \alpha^{\shortmid} (xy,y)=-1,\quad \alpha^{\shortmid} (xy,xy)=-1.$$

Now choose a basis  $\{u_g\}_{g\in G}$ for $F^{\alpha}G$
and a basis $\{v_g\}_{g\in G}$ for $F^{\alpha ^{\shortmid}}G$. 
 Notice that  both cocycles 
 $\alpha, \alpha ^{\shortmid}\in Z^2(G,F^*)$
 are of order $2$ and also notice that the cohomology classes which corresponds to these cocycles are
$$[\alpha]:\quad u_x^2=1,\quad u_y^2=1,\quad [u_x,u_y]=-1$$
and
$$[\alpha ^{\shortmid}]:\quad v_x^2=-1,\quad v_y^2=-1,\quad [v_x,v_y]=-1.$$
It follows that 
$$G_{\alpha}=\langle u_g \mid g\in G\rangle \cong D_8 \quad \text{ and } 
\quad 
 G_{\alpha^{\shortmid}}=\langle v_g \mid g\in G\rangle \cong Q_8$$ and therefore $G_{\alpha}^*\not \cong G_{\alpha^{\shortmid}}^*$.
The crucial point is that these cocycles are cohomologous 
over\footnote{The coboundary here is a map $f:G\rightarrow \mathbb{Q}(i)$ 
determined by $f(x)=i=f(y)$, $f(xy)=-1$ and $f(1)=1$. Indeed 
$\alpha ^{\shortmid}(g_1,g_2)=f(g_1)f(g_2)f(g_1g_2)^{-1}\alpha (g_1,g_2)$ 
for all $g_1,g_2\in G$.} $F=\mathbb{Q}(i)$ but not over $\mathbb{Q}$. 
\end{example}

\section{When is the unit group of a twisted group ring finite?}\label{unit group finite section}
Let $F$ be a number field, $R$ a $\Z$-order in $F$ and $\alpha \in Z^2(G, R^*)$ a {\it fixed} (normalized) $2$-cocycle. We refer to \cite[Section 4.6]{EricAngel1} for the necessary background on orders. In this section we determine when the unit group $\U(R^{\alpha}[G])$ is finite. 

Generally speaking, \cite[Corollary 5.5.8.]{EricAngel1}, if $\O$ a is $\Z$-order in 
a finite dimensional semisimple $\Q$-algebra $A$, then $\U (\O)$ is finite if and only 
if every simple component of $A$ is either $\Q$, a quadratic imaginary field extension 
of $\Q$ or a totally definite quaternion algebra over $\Q$. 
Recall that a totally definite quaternion algebra over $\Q$ is a 4-dimensional $\Q$-algebra with  basis $1,i,j,k$ so that $ij=k=-ji$, $i^2,j^2\in \Q$ and $i^2<0$, $j^2<0$. We will denote this algebra by $\qa{u}{v}{\Q}$ where $i^2 = u$ and $j^2=v.$

As a by-product, if $\U(R^{\alpha}[G])$ is finite and $x \in \U(R^{\alpha}[G])$ is of finite order, say $n$, then $n$ must\footnote{The $F$-subalgebra $F[x]$ of $F^{\alpha}G$ is a commutative semisimple subalgebra and hence a direct sum of cyclotomic extensions of $\Q (\xi_m)$, with $\xi_m$ a root of unity of order $m$, a divisor of $n$. Moreover, 
by the above,  $[\Q (\xi_m):\Q]\leq 2$  and one of the $m$ must be equal to 
$n$. Now the Dirichlet unit theorem yields the claim as a $\Z$-order in $\Q(\xi_n)$ 
is finite exactly when $n$ divides $4$ and $6$.} divide $4$ or $6$. Consequently, 
the exponent of $G_{\alpha}$ is a divisor of $4$ or $6$ if  $\U(R^{\alpha}[G])$ is finite. Moreover, also $\U(R)$ 
would need to be finite and hence by the above $F = \Q \text{ or } \Q(\sqrt{-d})$ 
with $d >0$.  In the former case $R= \Z$ is the only order in $F$. We obtain 
the following characterisation, generalizing a classical result of Higman 
\cite[Theorem 1.5.6]{EricAngel1} for untwisted group rings.

\begin{theorem}\label{finitetwisted}
Let $G$ be a finite group, $F$ be a number field, $R$ a $\Z$-order in $F$ and $\alpha \in Z^2(G, R^*)$ non-trivial normalized cocycle. Then the following are equivalent:
\begin{enumerate}
\item $\U(R^{\alpha}[G]) $ is finite,
\item $\U(R[G_{\alpha}])$ is finite,
\item $\U(R^{\alpha^{i}}[G])$ is finite for all $i$.
\end{enumerate}
One of the above holds if and only if one of the following conditions is satisfied: 
\begin{enumerate} 
\item[(i)] $G_{\alpha}$ is an abelian group of exponent $4$ and $F = \Q$ or $\Q(\sqrt{-1})$,
\item[(ii)] $G_{\alpha}$ is an abelian group of exponent $3$ and $F = \Q(\sqrt{-3})$,
\item[(iii)] $G_{\alpha}$ is an abelian group of exponent $6$ and $F = \Q$ or $\Q(\sqrt{-3})$ ,
\item[(iv)] $G_{\alpha}$ is a non-abelian Hamiltonian\footnote{A group is called {\it Hamiltonian} if every subgroup is normal.} 2-group and $F=\Q$.
\end{enumerate}
\end{theorem}

%
Recall that Higman's result says that $\U (\Z G)$ is finite if and only if $G$ is abelian of exponent dividing $4$ or $6$ or $G\cong Q_8 \times C_2^n$ for some $n \in \N.$ To put into perspective, it is good to recall Baer-Dedekind's classification theorem \cite[Theorem 1.8.5.]{PolSeh} which says that $G$ is Hamiltonian (i.e. all subgroups are normal) if and only if $G$ is abelian or $G \cong Q_8 \times C_2^n \times A$ with $A$ an odd order abelian group. Therefore, Higman's result says that $\U (\Z G)$ is finite exactly for the Hamiltonian groups with exponent dividing $4$ or $6$. Note that the list of possibilities stated above for a non-trivial cocycle is more restrictive than for a trivial cocycle, which is the reason that we excluded the trivial case in \Cref{finitetwisted}

Also note that the group $G_{\alpha}$ is abelian exactly when the cocycle $\alpha$ is symmetric (i.e. $\alpha(g,h) = \alpha(h,g)$ for all $h,g \in G$) and $G$ is abelian. In \Cref{When is unit group non-abelian finite} below we will give a concrete interpretation in terms of $G$ and $\alpha$ for case (iv).

\begin{proof}
Clearly in any of the cases (1)-(3) $\langle \im(\alpha) \rangle$ is a finitely 
generated torsion subgroup of $R^* \subset F^*$ and hence it is a finite 
cyclic group. In particular $o(\alpha) < \infty$ which allows to apply
\Cref{decomposition of FGalpha}:
$$F[G_{\alpha}] \cong \bigoplus_{i = 0}^{o(\alpha)-1} F^{\alpha^{i}}[G].$$
Hence the equivalence between (2) and (3) is a consequence thereof and commensurability of the unit group of $\Z$-orders in $F[G_{\alpha}]$ (see \cite[Lemma 4.6.9]{EricAngel1}). 
Also any order in $F^{\alpha^i}[G]$ is a direct summand of an order in $F[G_{\alpha}]$ 
and hence (2) implies (1).  The main bulk of the proof goes into proving that (1) implies (2). 
More precisely, we will show that $G_{\alpha}$ and $F$ are of the form (i)-(iv). 
In those cases one can directly see that $\U(R[G_{\alpha}])$ is finite 
(e.g. this can be deduced from \cite[Theorem 1.5.6.]{EricAngel1}, handling the case $R=\Z$, together with Dirichlet unit theorem \cite[Theorem 5.2.4.]{EricAngel1}).

Suppose $\U(R^{\alpha}[G]) $ is finite and hence, as noticed at the start of this section, 
$\exp(G_{\alpha})\mid 4 \text{ or } 6$ and $F^{\alpha}[G]$ is a direct sum of (certain) division $F$-algebras. In particular, it contains no non-zero nilpotent elements.

{\it Claim 1: } If $u_k \in G_{\alpha}$, with $k\in G$ such that $o(u_k) = o(k)$ 
then $\langle u_k \rangle$ is normal in $G_{\alpha}$. Moreover, this condition holds 
when $o(u_k)$ is square-free and if $o(u_k)= 2$ then $u_k \in \ZZ (G_{\alpha})$. 

The proof of this claim will be carried out inside $F^{\alpha}[G]$ (and not $F[G_{\alpha}]$). Let $\wt{u_k} := \sum_{i=0}^{o(u_k)-1} u_k^{i}$ and note that $u_k \wt{u_k} =\wt{u_k}$. Hence for any $t \in G$ the element $(1 -u_k) u _t \wt{u_k}\in R^{\alpha}[G]$ is a nilpotent element. By an earlier remark the nilpotent element is zero and thus
$u_t^{-1}u_k u_t \wt{u_k} = \wt{u_k}$.  Consider now $F^{\alpha}[G]$ with its 
canonical $G$-grading. Then $\text{deg}(u_k) = k $ and 
$\text{deg}(u_t^{-1}u_k u_t) = t^{-1}kt.$ Since $o(u_k)= o(k)$, all summands of 
$\wt{u_k}$ have different degrees (in particular $\wt{u_k} \neq 0$) and thus a degree argument shows that 
$u_t^{-1}u_ku_t = u_k^{j}$ for some $j$. Consequently,
$\langle u_k \rangle \triangleleft G_{\alpha}$, as desired.  Now suppose $o(u_k)$ is square-free. 
It is well known\footnote{More generally, if $g_1$ and $g_2$ commute, then $[u_{g_1},u_{g_2}] := u_{g_1}^{-1}u_{g_2}^{-1}u_{g_1}u_{g_2} = \alpha(g_1,g_2)$. Since $u_g^{o(g)}$ is central in $R^{\alpha}[G]$, it follows that $1 = [u_{g_1}^{o(g_1)},u_{g_2}] = \alpha(g_1,g_2)^{o(g_1)}$. In particular $\alpha(k^{i},k)$ is a $o(k^{i})$-root of unity in $R^{*}$ and so also a $o(k)$-root of unity. Since $u_k^{o(k)} = \prod_{i=1}^{o(k)-1}\alpha(k^{i},k) u_1$ and $R$ is commutative we obtain the claim.}
that if $o(\alpha)$ is finite then $u_k^{o(k)}$ is a root of unity in 
$R^*$ of order dividing $o(k)$. Hence if $o(u_k)$ is square-free, it is easily seen that $u_k^{o(k)}= 1$. The last part follows from the rest.

{\it Claim 2: } Either (I) $G_{\alpha}$ is abelian of exponent dividing 4 or 6, (II) $G_{\alpha} \cong Q_8 \times E$ with $E$ an elementary abelian $2$-group.

Note that by Baer-Dedekind's classification theorem \cite[Theorem 1.8.5.]{PolSeh} 
these cases are exactly those Hamiltonian groups with exponent dividing $4$ or $6$. Due to claim $1$ it remains to prove that the generators of order $4$ are also normal. For this recall that \cite[Theorem 11.5.12]{Voight} 
the unit group of a maximal order in a quaternion algebra $\qa{-a}{-b}{\Q}$, 
with $a,b \in \N$, is 
cyclic except for the maximal order in $\qa{-1}{-1}{\Q}$ and $\qa{-1}{-3}{\Q}$. 
In those cases the unit group is $SL(2,3) = Q_8 \rtimes C_3$, resp. 
$Dic_3 := C_3 \rtimes C_4$. Now consider $G_{\alpha} \leq \prod_{e \in
\PCI (F^{\alpha}[G])} G_{\alpha}e$. Since also $\exp(G_{\alpha}e) 
\mid 4,6$ we see that $G_{\alpha}e$ is either cyclic or a subgroup of 
$Q_8$ or $D_6$. Thus one can check that any element of order $4$ in the direct product generates a normal 
therein and in particular the same holds for $G_{\alpha}$. This finishes the proof of the second claim.

We will now study  both cases in more detail. To start, as already used, we 
know that $\U(R)$ is finite and hence that $F = \Q$ or $\Q(\sqrt{-d})$ with $d >0$. 
It remains to restrict the possible values of $d$ and exclude exponent $2$.
 
 {\it Case (I):} Suppose $G_{\alpha}$ is abelian. To start remark 
 $\exp(G_{\alpha})= 2$ is not possible. Indeed, 
 otherwise $1 = (u_g u_h)^2 = \alpha(g,h) u_{g}^2 u_{h}^2 = \alpha(g,h)$ 
 for all $g,h \in G$, in contradiction with the assumption that $\alpha$ is non-trivial. By computations of the same type, $\exp(G_{\alpha}) = 3$ and $F= \Q$ 
 is also impossible for non-trivial $\alpha$. The restriction on $F$ for the 
 other cases will follow from the fact that all simple components have degree 
 at most $2$ over $\Q$. Indeed, suppose $\exp(G_{\alpha}) = 4$ and write 
 $F = \Q(\sqrt{-d})$, with $d$ a square free non-negative integer. Then $F^{\alpha}[G]$ has a simple component 
 $\Q(\sqrt{-d}, \sqrt{-1}) = \Q(\sqrt{d}, \sqrt{-1})$ which is of degree 
 at most $2$ if and only if $d = 0,$ or $1$. If the exponent is $3$ or $6$, 
 then  there is a simple component $\Q(\sqrt{-d}, \zeta_3)$ which would 
 have degree larger than $2$ if $d \neq 0, 3$.  
 
{\it Case (II):} Finally consider the case that $G_{\alpha}\cong Q_8 \times E$ is an Hamiltonian 2-group. Since $F^{\alpha}[G]$ has a simple component $\qa{-1}{-1}{F}$ which needs to be totally definite (in particular $F$ is totally real), we indeed get that $F = \Q$.
\end{proof}

\begin{remark}
The condition on the coefficient ring $R$ can be generalized further. Namely, let $F$ be a global field and $S$ be a finite set of places of $F$ containing the archimedian ones. Denote by $\mc{O}_S$ the ring of $S$-integers of $F$, which is well known to be a Dedekind domain with finite quotients. Therefore, any $\mc{O}_S$-order $R$ is commensurable \cite[Lemma 4.6.9]{EricAngel1} with $\mc{O}_S$. Since $|\mc{U}(\mc{O}_S)|$ is finite only for $F = \Q$ or $\Q(\sqrt{-d})$, with $d>0$, and $S= \{ \infty\}$ (e.g. see \cite[Theorem 3.24.]{Clark}) we see that the conclusion of \Cref{finitetwisted} also holds for such $R$.
\end{remark}

\Cref{finitetwisted} and its proof could be considered as a first contribution about the interplay of torsion units and nilpotent elements between $R^{\alpha}[G]$ and $R[G_{\alpha}]$. A satisfactory answer to the following general question would very useful:

\begin{question}\label{question torsion and nilp interplay}
Is there a concrete connection between the torsion and nilpotent elements of $R^{\alpha}[G]$ and $R[G_{\alpha}]$? In particular, when (in terms of $G, R$ and $\alpha$) does $R^{\alpha}[G]$ not have nonzero nilpotent elements?
\end{question}

Using Theorem~\ref{finitetwisted} one can give an especially precise characterisation when $\U(R^{\alpha}[G])$ is a finite non-abelian group in terms of $G$ and $\alpha$. For this recall that any non-abelian Hamiltonian $2$-group $G$ is isomorphic to $Q_8 \times C_2^n$ for some $n$. In other words, it can be written as a stem\footnote{An extension is called {\it stem} if the base normal group is contained in $G' \cap \mathcal{Z}(G)$. In particular it is central.} extension 
$$1\rightarrow C_2 \rightarrow G \rightarrow C_2^{n+2}\rightarrow 1.$$
On the other hand, any non-trivial cohomology class $[\alpha] \in H^2(G, \Z^*)$ corresponds to a central extension of $G$ by $C_2$.
Therefore, from Theorem~\ref{finitetwisted} (iv) we deduce
\begin{corollary}\label{When is unit group non-abelian finite}
Let $G$ be a finite group, $F$ be a number field, $R$ an order in $F$ and $[\alpha] \in H^2(G, R^*)$ a non-trivial cohomology class. Then $\U(R^{\alpha}[G])$ is a finite non-abelian group if and only if the following conditions are satisfied
\begin{enumerate}
    \item $R=\mathbb{Z}$,
    \item $G$ is an elementary abelian $2$-group of rank at least $2$,
    \item $[\alpha]$ is inflated from a cohomology class $[\gamma]\in H^2(C_2\times C_2,C_2)$ which is determined by $u_x^2=u_y^2=[u_x,u_y]=-1$ where $x,y$ are generators of $C_2\times C_2$ and $-1$ is the generator of $C_2$.
    \end{enumerate}
\end{corollary}
\noindent Notice that the cohomology class $[\gamma]\in H^2(C_2\times C_2,C_2)$ above corresponds to $Q_8$.

To finish this section, we would like to come back on the proof of \Cref{finitetwisted} which unfortunately is quite indirect. Indeed the implication from (1) to (2) goes by classifying all the possibilities for $G_{\alpha}$ and $F$ and then noticing that in all these cases $\U(R[G_{\alpha}])$ is finite. A more natural approach would have been to construct for all $j$ an isogeny between $\U(R^{\alpha}[G])$ and $\U(R^{\alpha^{j}}[G])$. Such a map can however only come from ring (epi)morphism when $\text{gcd}(j, n) = 1$. The statement answers a question of Margolis and Schnabel\footnote{As is apparent from the proof, the main tool is however their result over the complex numbers\cite[Theorem 3.1.]{MaScLast}.} \cite[remark 3.2.]{MaScLast} in case $[\alpha]$ has a cocycle representative of finite order. 

\begin{proposition}\label{proposition on isom of tga}
Let $G$ be a finite group, $F$ a field of characteristic zero and $\alpha \in Z^2(G, F^*)$. If $F^{\alpha}[G] \cong F^{\alpha^{j}}[G]$, isomorphic  as rings, then $\text{gcd}(j, o([\alpha])) = 1$. If $\alpha$ has finite order and $F$ is a number field, then converse also holds.
\end{proposition}
\begin{proof}
First suppose that $F^{\alpha}[G] \cong F^{\alpha^{j}}[G]$. Note that since $G$ is finite, $\langle \im(\alpha) \rangle$ is a finitely generated abelian subgroup of $F^*$. In particular it lies in some countable subfield, say $L \subseteq F$. Moreover, with abuse of notation, $F^{\alpha}[G] \cong F \ot_L L^{\alpha}[G]$ and one can restrict the given isomorphism restricts to a ring isomorphism $L^{\alpha}[G] \cong L^{\alpha^{j}}[G]$. Now, since $L$ is countable it can be embedded in $\C$ and we can view $\alpha$ as having values in $\C^*$.  Consequently, by tensoring with $\C \ot_L -$ we see that $\C^{\alpha}[G] \cong \C^{\alpha^{j}}[G]$ as rings. Therefore we can apply \cite[Theorem 3.1.]{MaScLast} saying that $\text{gcd}(j, o([\alpha])) = 1$. 

Conversely, let $n= o([\alpha])$ and let $j\in \mathbb{Z}_{\geq 0}$ such that $\text{gcd}(j, n) = 1$. Since cohomologous cocycles $\alpha ,\beta \in Z^2(G, F^*)$
admit isomorphic twisted group rings $F^{\alpha}G\cong F^{\beta}G$ we may assume, by Lemma~\ref{lemma prime divisors of alpha}, that $\text{gcd}(j, o(\alpha)) = 1$. Hence we have an isomorphism $\sigma_j: \Q(\zeta_{o(\alpha)}) \rightarrow \Q(\zeta_{o(\alpha)}^{j})$ 
mapping $\zeta_{o(\alpha)}$ to its $j$-th power. With this at hand,
define $\psi : \Q(\zeta_{o(\alpha)})^{\alpha}[G] \rightarrow \Q(\zeta_{o(\alpha)})^{\alpha^{j}}[G]$ 
by $\psi(\sum a_g u_g) = \sum \sigma_j(a_g) v_g$. Note that for all $g,h \in G$:
$$\psi(u_g u_h) = \psi( \alpha (g,h) u_{gh}) = \alpha(g,h)^{j} v_{gh} = v_g v_h =\psi(u_g) \psi(u_h).$$ 
Consequently, $\psi$ is a ring epimorphism and hence isomorphism. Now note that $F$ contains a $\zeta_{o(\alpha)}$-root of unity. Therefore
an extension of scalars with field $F$ now finishes the proof.
\end{proof}

\section{Correlations between \texorpdfstring{$R [\Gamma]$}{} and \texorpdfstring{$R^{\alpha}[G]$}{} - a unit group point of view}\label{section correlations}

Throughout this section {\it we fix an extension} 
\begin{equation}\label{fixed central extension}
1\rightarrow N \rightarrow \Gamma \overset{\lambda}\rightarrow G\rightarrow 1.
\end{equation}
As in (\ref{De algemene extensie}), fix also a section $\mu$ of $\lambda$ and define 
$$\alpha(g,h)= \mu(g)\mu(h)\mu(gh)^{-1}.$$
In particular when $N$ is abelian we will always choose this $\alpha$ as the normalized representative of the cohomology class $[\alpha]$ corresponding (\ref{fixed central extension}). We {\it will always assume} that the underlying field $F$ is such that $\Char(F) \nmid |\G|$ and $R$ is some order in $F$.

We wish to compare $\mc{U}(R [\Gamma])$ and $\mc{U}(R^{\alpha}[G])$ with the aim to pullback results from the smaller group $G$, but in the twisted context, to the larger group $\Gamma$. For this we study the kernel and cokernel of the generalized transgression $\Psi_{\chi}$ from \Cref{ring homom from Tra}. Also in this section we will work more generally with $R^{\beta}[\Gamma]$ for some $[\beta] \in \im(\Inf_N)$. 

\noindent {\it Notation: } Recall that if $f : R \rightarrow S$ is a ring homomorphism, then we denote the induced map on the unit groups by $\wt{f}: \mc{U}(R) \rightarrow \mc{U}(S)$.

\subsection{On the cokernel of the generalized transgression map}\label{cokernel of trans subsection}\addtocontents{toc}{\protect\setcounter{tocdepth}{1}}

In the sequel, finiteness of cokernels of group morphisms will always follow from the following somehow folklore lemma.

\begin{lemma}\label{image porjection in ss finite cokernel}
Let\footnote{The prototypical example of such a ring is the ring of integers in a number field. The condition is the minimal needed to use classical methods with $R$-orders, e.g. that two $R$-orders have commensurable unit groups.} $R$ be a Dedekind domain with field of fractions $F$ such that $R/I$ is finite for all $0 \neq I \triangleleft R$ and let $A$ and $B$ be semimsimple $F$-algebras. Furthermore consider $R$-orders $\mc{O}_A$ in $A$ and $\mc{O}_B$ in $B$. If there exists a $R$-algebra epimorphism $\pi : \mc{O}_A \rightarrow \mc{O}_B$, then $\Coker(\wt{\pi})$ is finite.  
\end{lemma}  
\begin{proof}
By definition $\mc{O}_A$ contains an $F$-basis of $A$ and similarly for $\mc{O}_B$. 
Therefore we can extend $F$-linearly $\pi$ to an $F$-algebra epimorphism 
$\widehat{\pi}: A \rightarrow B$. Formally $\widehat{\pi}= id_F\ot_R \pi$ and we identify 
$A= F\mc{O}_A$ with $F\ot_R \mc{O}_A$. Due to  the semisimplicity of $A$, there exists a central idempotent $e$ in $A$ such that $\restr{\widehat{\pi}}{Ae}: Ae \rightarrow B$ is an isomorphism. In other words, when decomposing $A = Ae \op A(1-e)$ 
we can consider $\widehat{\pi}$ as projection on the first component. To obtain the desired statement, 
consider the $R$-orders $\mc{O}_A \subseteq \mc{O}_Ae \op \mc{O}_A(1-e)$ in $A$. 
Due to the conditions on $R$ there exists a non-zero element $r \in R$ such that $r (\mc{O}_Ae \op \mc{O}_A(1-e)) \subseteq \mc{O}_A$ and see 
\cite[Lemma 4.6.9.]{EricAngel1}, 
$$[\U(\mc{O}_Ae) \times \U(\mc{O}_A(1-e)) : \U(\mc{O}_A)] \leq [\mc{O}_A : r (\mc{O}_Ae \op \mc{O}_A(1-e))]< \infty.$$
Consequently also its epimorphic image $\Ima(\wt{\pi}) \cong \U(\mc{O}_A)e$ is of finite index in 
$\U(\mc{O}_Ae) \cong \U(\mc{O}_B)$ (with upper bound the above number), as desired.
\end{proof}

Note that the proof in fact gives a method to obtain an upper bound on $|\Coker(\wt{\pi})|$ which however depends on the element $r \in R$ obtained which is not explicit.

\begin{question}
What is an explicit, and generic, upper bound on $|\Coker(\wt{\pi})|$?
\end{question}

Using \Cref{ring homom from Tra} a first useful incarnation of \Cref{image porjection in ss finite cokernel} 
is with $\wt{\Psi}_{\chi}$ for some $\chi \in \Lin(\G , R)^{\G/N}$. However, in contrast to the kernel 
and despite it to be finite, in general a concrete description (or even generators) of the cokernel is out 
of reach. Instead, we will focus on comparing the cokernel of $\wt{\Psi}_{\chi,Q}$, defined in (\ref{map psi wrt to a normal sbgrp Q}), for certain types
of `nice' normal subgroups $Q$ of $\G$. The restrictions on $Q$ will be as in \Cref{CE again a CE} and be such that we have the 
following\footnote{In the right column one should be careful with the notation $T(\chi)$. 
More precisely in the right upper corner $T(\chi) = [\chi \circ \alpha] \in H^2(G, R^*)$. 
In the right lower corner $T(\chi) = [\chi \circ \gamma] \in H^2(G/Q, R^*)$ with $\alpha = \Inf (\gamma)$ 
as in \Cref{CE over N}. In particular $\Tra_{\alpha}(\chi) = \Inf_N(\Tra_{\gamma, N}(\chi))$ and so the 
going down arrows exist.} commutative diagram:

\begin{equation}\label{the two psi maps comparing start}
\xymatrix{
\mc{U}(R^{\beta}[\Gamma]) \ar[rr]^{\wt{\Psi}_{\chi}} \ar[d]_{ \wt{\omega}_{Q, \beta} }& & \mc{U}(R^{\beta . T(\chi)}[G]) \ar[d]_{ \wt{\omega}_{Q, \beta. T(\chi)} }\\
\mc{U}(R^{\beta}[\Gamma/Q]) \ar[rr]_{\wt{\Psi}_{\chi,Q} }& & \mc{U}(R^{\beta. T(\chi) }(\chi)[G/Q])  \\
}
\end{equation}
Note that (co)restricting $\wt{\Psi}_{\chi}$ yields a morphism 
$$\wt{\Psi}_{ker} : \U\left( 1 + R^{\beta}[\G] . I_Q \right)\longrightarrow \U\left( 1 + R^{\beta . T(\chi)}[G].I_Q \right) : 1 + x \mapsto 1 + \Psi_{\chi}(x)$$
between $\ker(\wt{\omega}_N)$ and $\ker(\wt{\omega}_{N,\alpha})$ (cf. \Cref{ring homom from Tra} and (\ref{definition Ichi}) for definition $I_Q$). 

\begin{theorem}\label{size and comparission of coker}
Let $\G$ be the extension (\ref{fixed central extension}), $[\beta] \in \im (\Inf_N), \, \chi \in \Lin(N,F)^{\G/N}$ and $R$ a Dedekind domain such that $R/I$ is finite for all $0 \neq I \triangleleft R$ and $\Char(R) \nmid \G$. Then, 
\begin{enumerate}
    \item[(1)] $\Coker(\wt{\Psi}_{\chi, \beta})$ is finite,
    \item[(2)]  for any normal subgroup $Q$ of $\G$ such that $Q \cap N = 1,  \, \alpha(x,y)=1$ for all $x,y \in Q,$ $[\beta] \in \im(\Inf_Q)$ and $\wt{\omega}_{Q}$ and $\wt{\omega}_{Q, \beta . T(\chi)}$ are surjective, then 
    $$|\Coker(\wt{\Psi}_{ker})|\, . \, |\Coker(\wt{\Psi}_{\chi,Q})| =  |\Coker(\wt{\Psi}_{\chi, \beta})|.$$\end{enumerate}
\end{theorem}
If $Q$ has a complement in $\Gamma$ (which will be our setting in the later sections), then the maps $\wt{\omega}_Q$ and $\wt{\omega}_{Q, \beta . T(\chi)}$ are surjective and the inflation condition on $\alpha$ will also be satisfied. The surjectivity of the augmentation is also the case when $R$ is an Artinian ring \cite[Lemma 3.4.]{BarLen}. 

\begin{proof}[Proof of \Cref{size and comparission of coker}.]

The first statement follows by combining \Cref{ring homom from Tra} and \Cref{image porjection in ss finite cokernel}. The second statement will follow from a diagram chasing argument starting from diagram (\ref{the two psi maps comparing start}). For notation simplicity we will write $\wt{\omega}$ for $\wt{\omega}_{Q,\beta}$ and $\wt{\omega}_{\chi}$ for $\wt{\omega}_{Q, \beta . T(\chi)}$.

We complete the rows to an exact sequence by adding their kernels and cokernels with the canonical embedding and projection denoted by $i_1, \pi_1$ (resp. $i_Q, \pi_Q$). Due to the commutativity of (\ref{the two psi maps comparing start}), $\pi_Q \circ \wt{\omega}_{Q\chi}$ induces a map $F_{Cok}$ between the cokernels.  All together we obtain the following diagram where all squares commute:

\begin{displaymath}
\xymatrix{
\ker(\wt{\Psi}_{\chi})\ar@{^{(}->}[r]^{i_1} & \mc{U}(R^{\beta}[\Gamma ]) \ar[rr]^{\wt{\Psi}_{\chi} } \ar@<1ex>[d]^{ \pi\circ \wt{\omega}}& & \mc{U}(R^{\beta . T(\chi)}[G]) \ar[d]^{ \wt{\omega}_{\chi} } \ar@{->>}[r]^{\pi_1} & \Coker(\wt{\Psi}_{\chi})\ar@{-->}[d]^{F_{Cok}}\\
& \frac{\mc{U}(R^{\beta}[\Gamma/Q])}{\ker(\wt{\Psi}_{\chi,Q})} \ar[rr]_{\wt{\Psi}_{\chi,Q}} & & \mc{U}(R^{\beta . T(\chi)}[G/Q]) \ar@{->>}[r]_{\pi_Q}& \Coker(\wt{\Psi}_{\chi,Q})   \\
}
\end{displaymath}
where $\pi$ is the canonical epimorphism from $\mc{U}(R^{\beta}[\Gamma/Q])$ to  $\frac{\mc{U}(R^{\beta}[\Gamma/Q])}{\ker(\wt{\Psi}_{\chi,N})}$.
The Snake lemma applied to the diagram above now yields an exact sequence:
\begin{displaymath}
\xymatrix@C=1em{
\ker(\pi \circ \wt{\omega}) \ar[rr]^{\wt{\Psi}_{ker}}& & \ker(\wt{\omega}_{\chi}) \ar[r] &  \ker(F_{Cok}) \ar[r]^{d} & \Coker(\pi \circ \wt{\omega}) \ar[r] &\Coker(\wt{\omega}_{\chi}) \ar@{->>}[r] & \Coker(F_{Cok})\\
}
\end{displaymath}
with all the morphisms being the straightforward ones and $d$ the connecting morphism.

Now assume that $\wt{\omega}$ and $\wt{\omega}_{\chi}$ are surjective. Then $\Coker(\pi \circ \wt{\omega})$ is trivial and hence, going through the exact sequence above via isomorphism theorems, $\ker(F_{Cok}) \cong \Coker(\wt{\Psi}_{ker})$. Also, $\Coker(F_{Cok})$ is trivial thus all together we obtain the desired statement. 
\end{proof}

\subsection{On the kernel of the generalized transgression map}\label{subsection corrolations}

Unlike $\Coker(\wt{\Psi}_{\chi})$, the kernel of $\wt{\Psi}_{\chi}$ will usually be infinite as already seen with augmentation maps. 
\begin{example}\label{ker of nat hom can be infinite example}
When $N = \G$ and $\chi$ the trivial character, then $\wt{\Psi}_{\chi}$  is simply the usual augmentation map from $\U(\Z[\G])$ to $\U(\Z)= \{ \pm 1 \}$. We now see that $[\U(\Z[\G]) : \ker(\wt{\Psi}_{\chi})] = 2$ and therefore the kernel is infinite when $\U(\Z[\G])$ is infinite which, by a theorem of Higman \cite[Th. 1.5.6.]{EricAngel1}, exactly happens when $\G$ is not an Hamiltonian $2$-group or an abelian group of exponent dividing $4$ or $6$. 
\end{example}

The advantage of the kernel is that thanks to \Cref{ring homom from Tra} and \Cref{decomp voor elke abelse en centrale extensie} one has a significant amount of information about its elements. Combined with  \Cref{finitetwisted} one can describe finiteness of $\ker(\wt{\Psi}_{\chi})$, in case of a central extension. Namely, $\ker(\wt{\Psi}_{\chi})$ is finite if and only if $\U (R(\chi\prime)^{\beta . T(\chi\prime)}[G])$ 
is finite for all $\chi\prime \in \Lin(N,F)$ different of the given $\chi$. Recall that $R(\chi\prime)$ denotes the smallest ring containing $R$ and the values of $\Ima(\chi\prime)$. Below we translate this to easily verifiable necessary conditions on $G,N$ and $\chi$

\begin{theorem}\label{size and comparission of ker}
Suppose that the extension (\ref{fixed central extension}) is central (i.e. $N \subseteq \ZZ (\G)$), $[\beta] \in \im (\Inf_N), \, \chi \in \Hom(N,R^*)$ and $R$ the ring of integers in a number field\footnote{This restriction on the coefficient ring is only due to the restriction in \Cref{Berman-Higman twisted}. With more work, as in \cite[Theorem III.1]{RogTay}, one could probably take any domain of characteristic $0$ such that no prime divisor of $|G|$ is divisible in $R^*$.} $F$. Then,
\begin{enumerate}
    \item $\{ \text{torsion units in } \ker(\wt{\Psi}_{\chi, \beta})\} = \{ \chi(a)^{-1}a \mid a \in N  \}$,
  \item  if $\ker(\wt{\Psi}_{\chi})$ is finite then $\U(R^{\beta}[\G])$ is finite or one of the following hold:
    \begin{itemize}
    \item  $N \cong C_p$ for $p\geq 5$ prime, $\Gamma \cong N \times G, \chi \neq \omega_N$ and $\U(R^{\Res (\beta)}[G])$ is finite,
    \item $N \cong C_9, F = \Q, \chi$ the faithful character and $G_{\beta}$ abelian with $\exp(G_{\beta}) \mid 6$
    \item $N \cong C_8, F = \Q, \chi$ faithful and $G_{\beta}$ abelian with $\exp(G_{\beta}) \mid 4$,
    \item $\exp(N) \mid 4,6$ and if $N$ is non-cyclic then $\lcm(\exp(G), \exp(N)) \mid 4,6$
    \end{itemize}
\end{enumerate}
\end{theorem}

To prove the second part we will need the next theorem which is a generalisation of the classical result of Berman and Higman on torsion units (see \cite[Theorem 2.3.]{KarUnit} or \cite[Theorem III.1]{RogTay}). The proof is similar to the one given by Karpilovsky in the untwisted case. For completeness sake and convenience of the reader we  include a proof. 

\begin{theorem}\label{Berman-Higman twisted}
Let $G$ be a finite group, $R$ the ring of integers in a number field $F$ and $[\beta] \in H^{2}(G, R^*)$. If $x = \sum_{g \in G}a_g u_g$ is a torsion unit in $R^{\beta}[G]$ such that $a_1 \neq 0$, then $x \in \U(R)$. 
\end{theorem}
\begin{proof}The proof is along the same lines as \cite[Theorem 2.3.]{KarUnit}. Let $u=\sum_{g\in G} r_g u_g$ be a torsion unit in $R^{\beta}[G]$, say of order $m$. Assume $r_1 \neq 0$. 
Let $n=|G|$ and consider the left regular representation $\rho: F^{\beta }[G] \rightarrow M_{n}(F)$. Then $\rho(u)$ is of finite order $m$ (note that $\rho$ is injective) and thus it is diagonalisable over the algebraic closure of $F$ in $\C$ (recall that $F$ is a number field).
Its eigenvalues $\zeta_1 , \cdots , \zeta_n$ are roots of unity, each of order a divisor of $m$.  We get that $\zeta_1 + \cdots + \zeta_n =\sum_{g\in G} r_g \mbox{tr} (\rho(u_g)) =nr_1$. Thus $|nr_1|=| \zeta_1 + \cdots  + \zeta_n| \leq n$. Moreover, $|nr_1|=n$ if and only if $\zeta_1 =\zeta_2 = \ldots = \zeta_n$.
If this holds then the diagonalisation of $\rho(u)$ is a multiple of the identity matrix, hence $\rho(u)$ is a central, i.e. $\rho(u)\in F$. 
Because $\rho$ is injective this yields that $u\in F \cap R^{\beta}[G]=R$, as desired.

So, it remains to show that $|nr_1|=n$. Assume, the contrary,  i.e. suppose $|nr_1|<n$. Let $\epsilon$ be a primitive $m$-th root of unity.  Note that $r_1\in \Q (\epsilon)\subseteq \C$. Since also $r_1\in R$ and thus $r_1$ is an algebraic integer, it follows that $r_1$ is in the ring of integers of $\Q (\epsilon)$.

We claim that, for every $\sigma \in \mbox{Gal}(\Q (\epsilon)/\Q)$, we have $|\sigma (nr_1)|<n$.
Indeed, note that complex conjugation and elements of $\mbox{Gal}(\Q (\epsilon)/\Q)$ commute and thus if 
$n= |\sigma (nr_1)|=\sigma (nr_1) \overline{\sigma(nr_1)} = \sigma (nr_1) \sigma ( \overline{nr_1})  $ 
then also $n=\sigma^{-1} (\sigma (nr_1)\sigma (\overline{nr_1}) )= |nr_1|<1$. This proves the claim.

Next, the assumption says that $|r_1|< 1$ and from the claim we get that $|\sigma (r_1)|<1$ for all 
$\sigma \in \mbox{Gal}(\Q (\epsilon)/\Q)$.
It follows that the norm $0\neq |N_{\mbox{Gal}(\Q ((\epsilon)/\Q) }(r_1)| =|\prod_{\sigma \in\mbox{Gal}(\Q (\epsilon)/\Q) } \sigma (r_1)| =
\prod_{\sigma \in\mbox{Gal}(\Q (\epsilon)/\Q) } | \sigma (r_1)|
< 1$. However, this yields a contradiction as $N_{\mbox{Gal}(\Q (\epsilon)/\Q) }(r_1)$ is an algebraic integer (see for example Proposition 4.1.8 in \cite{EricAngel1}) and thus belongs to $\Z$; so it can not be strictly between $0$ and $1$.
\end{proof}

It would be interesting to have a proof of \Cref{Berman-Higman twisted} which makes a reduction to the well-known Berman-Higman theorem for group rings. So, this is for example an instance where understanding \Cref{question torsion and nilp interplay} would help.
We now proceed to the proof of the main theorem
\begin{proof}[Proof of \Cref{size and comparission of ker}.]
The first part follows from a classical trick in untwisted group rings. 
Namely, first note that $\{ \chi(a)^{-1}a \mid a \in N  \}$ is indeed contained in $\ker(\wt{\Psi}_{\chi, \beta})$. Conversely, take a torsion unit $u \in \ker(\wt{\Psi}_{\chi, \beta})$. Because of  \Cref{ring homom from Tra}, $\ker(\wt{\Psi}_{\chi, \beta}) = \mc{U}(1 + R^{\beta}[\Gamma]I_{\chi})$ and thus  $u = 1 +  \sum_{g \in G} y_g u_{\mu(g)}$, with $y_g \in I_{\chi} = \Span_R \{ u_a - \chi(a) u_1 \mid a \in N\}.$ In particular, if there would be no $1 \neq a \in N$ so that $u_a$ is in  $\supp (u)$, 
then the coefficient of $u_1$ in the expression of $u\in R^{\beta}[\Gamma]$ is equal to $ 1$ and hence, by \Cref{Berman-Higman twisted}, $u \in \U(R)$. As $u\in \ker(\wt{\Psi}_{\chi, \beta})$ this implies    $u = 1$. 

Now assume there exists a non-trivial $a \in N \cap \supp(u)$ and consider the element $z := u . u_{a^{-1}} \chi(a)$. Note that, due to the centrality of $N$, the element $z$ again is a torsion unit in $\ker(\wt{\Psi}_{\chi, \beta})$. However now $1\in \supp (z)$ and therefore, by applying \Cref{Berman-Higman twisted}, $z = 1$. In other words, $u = \chi(a)^{—1}a$ as needed.

For the second part, we will {\it assume} that we are not in the trivial case that $\U(R^{\beta}[\G])$ is finite. 
Then, by \Cref{decomp voor elke abelse en centrale extensie}, we see that $\ker(\wt{\Psi}_{\chi})$ is finite if and only if $\U (R(\chi\prime)^{\beta . T(\chi\prime)}[G])$ 
is finite for all $\chi \neq \chi\prime \in \Lin(N,F)$ but $\U (R(\chi)^{\beta . T(\chi)}[G])$ is infinite. The rest of the proof will reinterpret all those conditions simultaneously using \Cref{finitetwisted}. To start, $\U(R(\chi\prime))$ must also be finite for all $\chi \neq \chi\prime $. In particular $\U(R)$ needs to be finite and hence $F= \Q$ or $\Q(\sqrt{-d})$ with $d > 0$. The exact restrictions are recorded by the following claim. \vspace{0,1cm}

 \noindent {\it Claim}: If $\mathcal{U}(RN)$ is finite then: (i) $\exp (N) = 2$ or (ii) $\exp(N) =4 ; F = \Q, \Q(i)$ or (iii) $\exp(N) = 3,6 ;\; F = \Q, \Q(\sqrt{-3})$. \newline
\indent If  $\mathcal{U}(RN)$ is infinite then : (iv) $N \cong C_p$ for a prime $p \geq 5$ or (v) $N \cong C_8, C_9;\; F= \Q $.\vspace{0,1cm}

\noindent Decompose $FN \subseteq \ZZ(F^{\beta}[\G])$ as in (\ref{Perlis-Walker decomposition}), 
i.e. using the theorem of Perlis and Walker:
$FN = \bigoplus_{d\mid |N|} F(\zeta_d)^{\op a_d}$, then there has to be at most one 
$\zeta_d$ for which the the unit group of the ring of integers in $F(\zeta_d)$, 
denoted $R_d$, is infinite. Moreover, for this value we need $a_d = 1$. 
Note that this property is inhered by subgroups of $N$. 
Hence, if $N=N_1\times N_2$,  a direct product of two non-trivial subgroups, then, as  $FN=FN_1 \otimes_{F} FN_2$,  we get that all simple epimorphic images of both group algebras $FN_1$ and $FN_2$ must be such that the unit group of their respective ring of integers is finite; because otherwise we would have at least two distinct $d$ and $d'$ so that  $F(\zeta_d)$ and $F(\zeta_{d'})$ are simple epimorphic images of $FN$ and the unit groups $\U(R_d)$ and $\U(R_{d'})$ are  infinite. 
Hence, if $N$ is decomposable as a direct product then $\U (RN)$ is finite and thus, by \Cref{finitetwisted}, the exponent of $N$ is a divisor of $4$ or $6$. The restrictions on $F$ follow in the same way as at the end of the proof of \Cref{finitetwisted}, i.e. by determining when $\Q(\sqrt{-d}, \zeta_{\exp(N)})$ is at most of degree $2$ over $\Q$.  

On the other hand, if $N$ is indecomposable and its exponent is not a divisor of $4$ or $6$, then $N$ is cyclic and  $N=p^{n}$ for some prime $p$. If $n>1$ then $FN$ has $F(\zeta_p)$ as a summand and $R_p$ is finite only if $p=2, 3$ and $F = \Q, \Q(\sqrt{-d})$ with $d >0$. Thus either $n= 1$ or $p=2,3$. In the latter case, as $R_{p^n}$ is finite only if $p^n = 2,3,4$ we obtain the claim about the form of $N$. The according restrictions on $F$ can be proven with a similar argument as when $\exp (N) \mid 4,6$. \vspace{0,1cm}

\noindent {\it Claim: }  If $N$ is cyclic, then $\chi$ is the unique faithful character. Moreover if $N = C_p$, then $\Gamma$ is a split extension.

If $N$ is cyclic, then there is a bijection between the irreducible characters and the divisors $d$ of $|N|$, say $\chi_d$ corresponds. Since we always assume that $\U(R^{\beta}[\Gamma])$ is infinite we know that there is a unique component of $\bigoplus_{d \mid |N|} R[\chi_d]^{\beta. T(\chi_d)}[G]^{\op a_d}$ with infinite unit group. \Cref{finitetwisted} implies that if $\U(R[\chi_d]^{\beta. T(\chi_d)}[G]^{\op a_d})$ is infinite, then also for a multiple of $d$. Hence indeed the only infinite component is $d=n$, i.e. the unique faithful character. If $|N|$ is odd, then by the previous claim $N = C_9$ or $C_p$ with $p \geq 5$ prime. In both cases $\mathcal{U}(R^{\Res(\beta)}[G]$ is finite and hence $\exp(G_{\beta}) \mid 4,6.$ Therefore $\alpha(g,h)$ needs to devide both $lcm(o(g),o(h))$ and $p \geq 5$ which is only possible if $\alpha(g,h) = 1$ for all $g,h \in G$. Hence the extension is split if $N = C_p$.

We are now able to finish to obtain the desired restrictions in the case that $\mathcal{U}(\Z N)$ is infinite, i.e. $N = C_p , C_9, C_8$ by the first claim. For $C_p$ the desired statement follows by the above claims and the fact that there are only two components in this case (hence the faithful one is the one different from the trivial representation, i.e. $\omega_N$). For $C_8$ and $C_9$ it remains to prove the desired value on $\exp (G_{\beta})$. For this note that $\chi = \omega_N$ yields $\exp (G_{\beta}) \mid 4,6$ by \Cref{finitetwisted}. If $\exp(G) = 4$, then $\Q(\zeta_3)^{\gamma}[G]$ is infinite for any $2$-cocycle by \Cref{finitetwisted}. Similarly for $\Q(i)^{\gamma}[G]$ when $3 \mid \exp (G_{\beta})$. This finishes the $\mathcal{U}(\Z N)$ infinite case.

Finally suppose that $\mathcal{U}(\Z N)$ is finite. In this case $\exp (N) \mid 4,6$ and $F$ is of the form as in the first claim. Also $\exp(G) \mid \exp(G_{\beta T(\chi)}) \mid 4,6$ for various characters $\chi$ of $N$. If $N$ is cyclic, then the only restriction is that $\U(R^{\beta}[G])$ is finite. So suppose that $N$ is not cyclic. If $\exp(N) = 4$ and $3 \mid |G|$, then there are several components having an element of order $12$ which is not possible. Hence in that case $\exp(|G|) | 4$. Similarly if $3 \mid \exp(N)$ and $4 \mid |G|$. Thus we indeed obtain that $\lcm(\exp(G),\exp(N)) \mid 4,6$ if $N$ is non-cyclic. 
\end{proof}

For applications it would be useful to also understand the torsion units in $\ker(\wt{\Psi}_{\chi, \beta})$ in case that the (\ref{fixed central extension}) is abelian (i.e. $N$ is abelian but not necessarily central). The first part of \Cref{size and comparission of ker} could be read as saying that a torsion unit in  $\ker(\wt{\Psi}_{\chi, \beta})$ must be in $\Z[N]$.  Indeed because $N$ is abelian one can then use \Cref{Berman-Higman twisted} to conclude that the torsion units are trivial and hence of the form $\chi(a)^{-1}a$. So we expect the following to be true.

\begin{conjecture}\label{conjecture torsion units in kernel transgression}
If the extension (\ref{fixed central extension}) is abelian, then $\{ \text{torsion units in } \ker(\wt{\Psi}_{\chi, \beta})\} \subseteq \U(\Z[N])$. In particular all torsion units in the kernel are trivial.
\end{conjecture}

To finish this section, we illustrate the concepts of  \Cref{Sectie decompositie} and \Cref{section correlations} on an example.

\begin{example} \label{example order 16 everything worked out}
Consider 
$\Gamma = C_4 \rtimes C_4 := \langle a,b \mid a^4=b^4=1, \, a^b = a^{-1} \rangle$ which can be viewed in the following way as a central extension
\begin{equation}\label{group [16,4] as CE}
1 \rightarrow \langle y_1 : y_1^2=1 \rangle \times \langle y_2 : y_2^2=1 \rangle  \longrightarrow \Gamma \overset{\lambda}\longrightarrow \langle x_1 : x_1^2=1 \rangle \times \langle x_2 : x_2^2=1 \rangle \rightarrow 1
\end{equation}
where in fact $y_1 = a^2$ and $y_2 = b^2$. The epimorphism $\lambda: \Gamma \rightarrow C_2 \times C_2$ is determined by $\lambda(a) = x_1$ and $\lambda(b) = x_2$.  Furthermore, we choose as a section $\mu:\langle x_1,x_2\rangle\rightarrow \Gamma$ defined by 
$$\mu (1)=1, \mu (x_1)=a,\mu(x_2)=b,\mu (x_1x_2)=ab$$
and consider the induced normalized cocycle $\gamma \in Z^2(\langle x_1,x_2\rangle,\langle y_1,y_2\rangle)$ (i.e. $\gamma(g,h) = \mu(g)\mu(h)\mu(gh)^{-1}$). 
Explicitly:
$$\gamma(x_1,x_1)=y_1,\gamma (x_1,x_2)=1,\gamma (x_1,x_1x_2)=y_1$$
$$\gamma(x_2,x_1)=y_1,\gamma (x_2,x_2)=y_2,\gamma (x_2,x_1x_2)=y_1y_2$$
$$\gamma(x_1x_2,x_1)=1,\gamma (x_1x_2,x_2)=y_2,\gamma (x_1x_2,x_1x_2)=y_2.\vspace{0,2 cm}$$
The corresponding class $[\gamma] \in H^2(\Gamma, C_2 \times C_2)$ is determined by
$$[x_1,x_2]= y_1, x_1^2 = y_1 \text{ and } x_2^2 = y_2.$$ 

Now consider the irreducible characters of $C_2\times C_2 = \langle y_1,y_2\rangle$, say $\Irr( \langle y_1,y_2 \rangle) = \{ 1, \chi_1, \chi_2, \chi_3 \}$ with $\chi_1(y_i)=-1, \, \chi_2(y_1) = -1 = - \chi_2(y_2)$ and $\chi_3(y_1) = 1 = - \chi_3(y_2)$.
Then using \Cref{decomp voor elke abelse en centrale extensie} we obtain that 
$$\Q \Gamma \cong \Q[C_2 \times C_2] \oplus \bigoplus_{i=1}^3 \Q^{T_{\gamma}(\chi_{i})}[C_2 \times C_2].$$
It is easily seen that 
$$\Q^{T_{\gamma}(\chi_{1})}[C_2 \times C_2] \cong \qa{-1}{-1}{\Q}, \, \Q^{T_{\gamma}(\chi_{2})}[C_2 \times C_2] \cong \Ma_2(\Q) \text{ and } \Q^{T_{\gamma}(\chi_{3})}[C_2 \times C_2] \cong 2 \Q(i).$$
Therefore
$$\wt{\Psi}_{\chi_2}: \U (\Z \G ) \rightarrow \U(\Z^{T_{\gamma}(\chi_{2})}[C_2 \times C_2])$$
is the only projection with infinite codomain. Moreover, by \Cref{size and comparission of ker} the kernel is finite and is given by 
$$\ker(\wt{\Psi}_{\chi_2}) = \langle - a^2, b^2 \rangle \cong C_2\times C_2.$$
Also $\Coker(\wt{\Psi}_{\chi_2})$ is finite by \Cref{size and comparission of coker} and a precise description will follow from the methods in the upcoming sections.
\end{example}

\section{Deforming bicyclic units via twisting and their contribution}\label{sectie twisted bicyclic en de rest}
In this section we will introduce a generic construction of units in twisted group rings (of finite groups) which generalizes the bicyclic units. They differ by a factor determined by the twisting. We will show in \Cref{sectie bicyclic units} that these units play a role analogous to the one of elementary matrices inside the special linear group, building on a generalisation of the Jespers-Leal theorem obtained in \Cref{Jespers-Leal generalized}. Thereafter we return to the context of the previous section and investigate their contribution to $\Coker(\Psi_{\chi})$.

\subsection{Generalized byclic units versus elementary matrices}\label{subsectie on JL general}\addtocontents{toc}{\protect\setcounter{tocdepth}{2}}
Before specifying to twisted group rings, we wish to write down a general `$S$-order version' of an important `$\Z$-order result' of Jespers and Leal \cite{JesLea2} on bicylic units which was only informally known to some experts. For this let $B$ be a finite dimensional semisimple $\Q$-algebra, $S$ a finite set of places of $\mathbb{Q}$ such that $\{ \infty \} \subseteq S$ and $\Z_S$ the ring of $S$-integers in $\Q$. Further let $A$ be a finite dimensional semisimple $B$-algebra and $R$ an $\Z_S$-order in $B$.

\begin{example}
Let $S = \{\infty, (p_1), \ldots, (p_l) \}$ with $p_i$ a prime number. Then $\Z[\frac{1}{p_1}, \ldots, \frac{1}{p_l}]$ are the $S$-integers in $\Q$. If we now consider some root of unity $\zeta$, then $R= \Z_S[\zeta] = \Z[\frac{1}{p_1}, \ldots, \frac{1}{p_l}, \zeta]$ is a possible $\Z_S$-order in $B = \Q(\zeta)$. For this choice a first main example is $A = \Q(\zeta)^{\alpha}[G]$ for some $[\alpha] \in H^2(G, \Q(\zeta)^*)$ and a finite group $G$. Another frequent example is $A = M_n(D)$ with $D$ a finite dimensional division $\Q$-algebra.
\end{example}

Let $\mathcal{A} = \{ g_1, \ldots, g_n \}$ be a $B$-basis of $A$ and $\{ x_1, \ldots, x_q \}$ a $\Z_S$-basis of $R$.  For every idempotent $f \in A$, there exists a minimal integer $n_f \in \Z$ such that $n_f f \in \Span_R(\mc{A})$. Associated to these, one can now consider the elements
$$b(m_ix_ig,f) = 1 + n_f^2(1-f)m_ix_igf \text{ and } b(f,m_ix_ig) = 1 + n_f^2fm_ix_ig(1-f).$$
 with $m_i \in \Z_S$ and $g\in \mathcal{A}$. As usual, these are units because $(1-f)m_ix_igf$ has square zero.

\begin{definition}\label{Def gen bicyclic}
    Let $\mc{F}$ be a set of idempotents in $A$. Then the elements in
    $$\lbrace b(m_ix_ig,f),b(f,m_ix_ig) \mid f \in \mc{F}, 1 \leq i \leq q, g \in \mc{A}, m_i \in \Z_S \rbrace $$
    are called the {\it generalized bicyclic units} corresponding to $\mc{F}$. The group generated by these is denoted by $\text{GBic}^{\mc{F}}(A,R)$.
\end{definition}
Classically only the generators with $m_i=1$ are called generalized bicyclic units. The values $m_i$ are added in order to still obtain all elements of $R\mathcal{A}$ in the middle between $(1-f)$ and $f$ (when taking the group generated by). Note that $R\mathcal{A}$ is an $\Z_S$-order in $A$.

 We can now formulate the general version of \cite{JesLea2} saying that generalized bicyclic units should be considered as the analogue of elementary matrices. The proof of  \cite{JesLea2} shows that Jespers and Leal handled the case $S = \{ \infty\}$, i.e. of $\Z$-orders, and with $A$ an (untwisted) group ring  . 

\begin{theorem}\label{Jespers-Leal generalized}
Let $e \in \PCI (A)$ such that $Ae \cong \Ma_n(D) $ with $n \geq 2$. Let $\mc{O}$ be a $\Z_S$-order in $D$. If $f$ is an idempotent in $A$ such that $fe \notin \mc{Z}(Ae)$, then there exists a non-zero $y \in \Z$ such that
$$1-e + E_n(y \mc{O}) \subseteq  \text{GBic}^{\{f \}}(A,R)$$
\end{theorem}
 Recall that, a set of matrix units in a simple algebra $Ae$ is a set of elements $\{ E_{i,j} : 1 \leq j \leq n \}$ in $\SL_1(Ae)$, where $n$ is the reduced degree of $Ae$, such that $\sum_{i=1}^n E_{i,i} = 1$ and $E_{i,j} E_{p,q} = \delta_{j,p} E_{i,q}$. Concretely, fixing the isomorphism $Ae \cong \Ma_n(D)$ the element $E_{i,j}$ can be identified with the matrix having $1$ in the $(i,j)$ entry and zero elsewhere. For an ideal $J$ in $\O$ we denote by $E_n(J)$ the subgroup of $SL_n(\O)$ generated by the elements $e + r E_{i,j}$ for $i \neq j$ and $r \in J$. 

\begin{proof}[Proof of \Cref{Jespers-Leal generalized}]
Multiplying generalized bicyclic units one sees that
$$\{1 + n_f^2f \alpha (1-f), 1+n_f^2(1-f)\alpha f \mid \alpha \in R\mathcal{A} \} \subseteq \text{GBic}^{\{f \}}(A,R).$$
Following \cite[Lemma 11.2.4]{EricAngel1}) one can decompose $fe =E_{1,1}+ \cdots + E_{l,l}$ for some $0 < l < n$ and $\{ E_{i,j} \mid 1 \leq i,j \leq n\}$ a set of matrix units. In case  $1 \leq i \leq l$ and $l+1 \leq j \leq n$ one readily sees that by using $E_{i,j} E_{p,q} = \delta_{j,p} E_{i,q}$ the products can be rewritten as following:
$$f \mc{O} E_{i,j}(1-f)e= \mc{O} E_{i,j} \text{ and } (1-f) \mc{O}E_{j,i}fe = \mc{O} E_{j,i}$$
for all $1 \leq i \leq l$ and $l+1 \leq j \leq n$. Since $\O$ is an $S$-order it has a finite $\Z_S$-basis, which we denote $\mathcal{B}_{\mathcal{O}}= \{b_1, \ldots, b_k\}$. Now, for any $i,j$, \cite[Lemma 4.6.9]{EricAngel1} yields a smallest number $n_{ij}$ such that $(1 + b E_{i,j})^{n_{ij}} = 1 + n_{ij} b E_{i,j} \in R\mathcal{A}e$ for all $b \in \mathcal{B}_{\mathcal{O}}$. Consequently, if we also consider the smallest value $n_e \in \Z$ such that $n_e e \in R\mathcal{A},$ then 
$$1 + n_f^2 n_e n_{ij} \mc{O}E_{i,j} \subset \text{GBic}^{\{f \}}(A,R)$$
for all $1 \leq i \leq l$ and $l+1 \leq j \leq n$. Similarly, for these $i$ and $j$ one has that $1 + n_f^2 n_e n_{ji} \mc{O}E_{j,i}$ is a subet of $\text{GBic}^{\{f \}}(A,R)$. The other indices can now be reached by taking the appropriate commutators. For example if $1 \leq i,j \leq l, i\neq j, x \in \Z$ and $\alpha \in \O$ then 
$$1 +x^2 \alpha E_{i,j} = \left( 1+x \alpha E_{i,l+1} \, ,\,  1 + x E_{l+1,j} \right).$$
and using the $E_{j,i}$ one also reaches the $l+1 \leq i,j \leq n$. Appropriate choices of $x$ now yields the desired result.
\end{proof}
 Recall that the literature on the subgroup congruence problem yields that the difference between $E_n(y\mc{O})$ and the congruence subgroup of level $y$ of $\SL_n(\mc{O})$ increases with $y$. Moreover if $n=2$ and $D$ is a division algebra containing an order with finitely many units\footnote{A division algebra contains an order with finite unit group if and only if $D=\Q(\sqrt{-d})$ or $\qa{u}{v}{\Q}$ with $d \geq 0$ and $u,b < 0$, see \cite[Theorem 2.10.]{BJJKT}.}, then $[\SL_n(\mc{O}), E_n(y\mc{O})] = \infty$ starting already from rather small values of $y$. Therefore, for applications an answer to the following would be incredibly valuable.

\begin{question}\label{Question bound on the y in congruence level}
Suppose $n=2$ and $D$ is a division algebra containing an order with finitely many units. What is 
\begin{enumerate}
    \item a tight upper bound for the scalar $y$ of \Cref{Jespers-Leal generalized} in terms of the starting data?
    \item a tight upper bound for the scalar $y$ in terms of the proportions $o(g)/o(ge)$ for $g \in G$?
\end{enumerate}
\end{question}
Note that the proof above shows that one obtains the following value for $y \in \Z$:
\begin{equation}\label{exact value multiple in Jespers-Leal}
y = \left( n_e \,  n_f^2 \, \text{gcm}\{ n_{ij}, n_{ji} : 1 \leq i \leq l, l+1 \leq j \leq n \}\right)^2.
\end{equation}
Moreover if $n=2$, then $x = \sqrt{y}$ suffices as a multiple. A limitation of the value in (\ref{exact value multiple in Jespers-Leal}) is that the decomposition $fe = E_{1,1} + \cdots + E_{l,l}$ and the numbers $n_{ij}$ are not explicit from the starting data in \Cref{Jespers-Leal generalized}. A value for $n_{ij}$ can be deduced by following the steps in the proofs of \cite[Lemma 4.6.6 \& 4.6.9]{EricAngel1}, however the value would require to know too much of the isomorphism type of $G$ and hence to weak for practical use. Nevertheless, according to \cite[Lemma 4.6.9]{EricAngel1} there exists a $0 \neq r \in \Z_S$ such that $\Ma_n(r\mc{O}) \subseteq R\mathcal{A}e$ and that its additive index yields a `uniform looking' upper-bound:
\begin{equation}\label{upperbound multiple in Jespers-Leal}
y = \left( n_e \,  n_f^2 \, [R\mathcal{A}e : \Ma_n(r\O) ] \right)^2.
\end{equation}
Unfortunately in practice the additive index above is hard to compute.

\subsection{Deforming bicyclic units - case of twisted group rings}\label{sectie bicyclic units}
From now on we consider the setting where $S=\{\infty \}$, i.e. $\Z$-orders, and $A = \Q(\zeta_n)^{\alpha}[G]$ where $\zeta_n$ is some primitive $n$-th root and $[\alpha] \in H^2(G, \Z[\zeta_n]^{*}) \subseteq H^2(G,\Q(\zeta_n)^{*})$ arbitrary. The importance of this case for the study of $\Q[\G]$ for $\G$ some central extension is highlighted by \Cref{decomp voor elke abelse en centrale extensie} and \Cref{ring homom from Tra}.

\subsubsection*{Idempotents from trivial units}

Let $\alpha \in Z^2(G, \Z[\zeta_n]^{*})$ be a $2$-cocycle of finite order with $o(\alpha) \mid n$. Then, for $0\leq j<n$ we can partition $G$ into the following sets,
\begin{equation}\label{partition $G$ depending on self-twist}
\mc{G}_{j}^{\alpha} = \{ g \in G \mid u_g^{o(g)}= \zeta_n^j \}.
\end{equation}

The observation is now that if $g \in \mc{G}_{0}^{\alpha} = \{ g \in G \mid o(u_g) = o(g) \}$, then $\widehat{u_g}= \frac{1}{o(g)} \wt{u_g}$ with 
\begin{equation}\label{definition tilde}
\wt{u_g} := u_1 + u_g + \ldots + u_g^{o(u_g)-1}
\end{equation}
is again a non-trivial idempotent in $\Q(\zeta_n)^{\alpha}[G]$ if $g \neq 1$. Consequently, $\wt{u_g}(o(g) - \wt{u_g}) = 0$. If $g \in \mc{G}_{j}^{\alpha}$ for a non-zero $j$, then\footnote{Going in the sum till the usual upper bound $o(g)-1$ will not yield an idempotent.} $\wt{u_g} = 0$. In that case, one could instead take the hat $\widehat{( \cdot ) }$ of $u_g^{n/gcd(j,n)}$ which would be an idempotent  but in practice this construction will not be useful. Hence:

\begin{question}\label{question idemp in case of self-extensions}
Is there a generic way to produce a non-trivial idempotent from $g \in \mc{G}_{j}^{\alpha}$ for a non-zero $j$?
\end{question}

\begin{remark}\label{G-sets det by cohom class}
It is well known (see e.g. \cite[Theorem 2.3.1]{KarProj}) that for a cyclic group $C_n = \langle g \rangle$ the second cohomology group over a commutative ring $R$ is isomorphic to $R^*/(R^*)^n$. Due to this, the value of $u_g^{o(g)}$ is not uniquely determined by the cohomology class $[\alpha ]$. For example, the class $[\alpha ] \in H^2(C_n, \C^*)$ determined by the value $\lambda = u_g^{o(g)}$ is trivial for any $0 \neq \lambda \in \C$. Therefore the sets $\mc{G}_{j}^{\alpha}$ are unfortunately dependant of the chosen cocycle. However, one could define in an independent way the sets $\mc{G}_0$ and $\mc{G}_{\neq 0}$ where the latter would be all $g \in G$ such that $u_g^{o(g)}$ is not in the class of $1$ in $R^*/(R^*)^n$. In particular for $R = \Z$ the situation simplifies and the value $u_g^{o(g)} \in \{ \pm 1\}$ is uniquely determined by the class $[\alpha]$. 
\end{remark}

Despite the absence of an answer to \Cref{question idemp in case of self-extensions}, one can prove that under natural conditions the set $\mc{G}_0^{\alpha}$ yields enough idempotents to apply \Cref{Jespers-Leal generalized}. More precisely that for $\mathcal{F} = \{ \widehat{u_g} = \frac{1}{o(g)}\wt{u_g} \mid g \in \mathcal{G}_0^{\alpha} \}$ the group $\text{GBic}^{\mc{F}}(\Q(\zeta_n)^{\alpha}[G],\Z[\zeta_n])$ contains sufficiently many elementary matrices. 

\begin{corollary}\label{Jespers-LEal for twisted bicyclic}
Let $G$ be a finite group and $\alpha \in Z^2(G,\Z[\zeta_n]^{*})$ of finite order such that $G_{\alpha}$ has no fixed-point free non-abelian images\footnote{A finite group is called fixed point free if it has an irreducible complex representation $\rho$ such that $1$ is not an eigenvalue of $\rho(g)$ for all $1 \neq g \in G$. Such groups are exactly the Frobenius complements \cite[Proposition 11.4.6.]{EricAngel1} and hence those that are subgroups of $D^*$ for some finite dimensional division algebra $D$.}. Then for any $e \in \PCI (\Q(\zeta_n)^{\alpha}[G])$ such that $\Q(\zeta_n)^{\alpha}[G]e \cong M_m(D) $ with $m \geq 2$ we have that $$1-e+E_m(y \mc{O}) \subseteq \text{GBic}^{\mc{F}}(\Q(\zeta_n)^{\alpha}[G],\Z[\zeta_n])$$ 
with $\mc{O}$ an order in $D$ and some $y \in \Z$. In particular, if $\Q(\zeta_n)^{\alpha}[G]$ has no exceptional components then $\text{GBic}^{\mc{F}}(\Q(\zeta_n)^{\alpha}[G],\Z[\zeta_n])$ is of finite index in $\SL_1(\Z[\zeta_n]^{\alpha}[G])$.
\end{corollary}

Recall that for a subring $R$ of a semisimple algebra $A = \prod_{i=1}^q \Ma_{n_i}(D_i)$, $h_i$ the projection on the $i$-th component and $E_i$ a splitting field of $K_i:= \mc{Z}(D_i)$ one defines 
\begin{equation}\label{def reduced norm}
nr (r) = (\text{Rnr}_{\Ma_{n_1}(D_1)/K_1}(h_1(r)), \ldots, \text{Rnr}_{\Ma_{n_q}(D_q)/K_q}(h_q(r)))
\end{equation}
with $\text{Rnr}_{\Ma_{n_i}(D_i)/K_i}(h_i(r)) := \det (1_{E_i} \ot_{K_i} h_i(r))$ the {\it reduced norm over $K_i$} and
\begin{equation}\label{def of SL1}
    \SL_1(R) = \ker(nr) = \{ r \in \U(R) \mid \forall i \, : \, \text{Rnr}_{\Ma_{n_i}(D_i)/K_i}(h_i(r)) = 1  \}
\end{equation}
the (multiplicative) group of reduced norm $1$ elements. 
Also recall that a simple quotient of $\Q(\zeta_n)^{\alpha}[G]$ is called an {\it exceptional component} if it is either (I) a non-commutative division algebra different from a totally definite quaternion algebra or (II) of the form $\Ma_2(\Q(\sqrt{-d}))$ or $\Ma_2(\qa{-a}{-b}{\Q})$ with $a,b >0$ and $d \in \N$. The division algebras appearing in these matrix algebras are exactly those having an order with finite unit group \cite[Theorem 2.10.]{BJJKT}.

With the conditions assumed in \Cref{Jespers-LEal for twisted bicyclic}, the proof of \cite[Theorem 11.3.2]{EricAngel1} also works in this setting and therefore we will omit the details. In short, that $G_{\alpha}$ has no fixed point free images guarantees that for every $e \in \PCI (\Q(\zeta_n)^{\alpha}[G])$ there exists $g\in \mc{G}_0^{\alpha}$ so that $\widehat{u_g}e$ is a non-central idempotent. If $e_C \in \PCI ( \C^{\alpha}[G]e)$ is such that $(\C^{\alpha}[G]e)e_C$ is non-commutative, then the element $g$ is chosen among those for which $\rho(g)$ has eigenvalue one, where $\rho : \C\ot_{\Q(\zeta_n)} \Q(\zeta_n)[G_{\alpha}]	\twoheadrightarrow (\C \ot_{\Q(\zeta_n)} \Q(\zeta_n)^{\alpha}[G]e)e_C$ is the complex representation factorising through the projection of $\Q(\zeta_n)[G_{\alpha}]$ on $\Q(\zeta_n)^{\alpha}[G]$ in \Cref{decomposition of FGalpha}. The last part follows from the known results on the subgroup congruence problem, e.g. see \cite[Theorem 11.2.3.]{EricAngel1}. \vspace{0,1cm}

{\bf Convention:} For an integral twisted group ring $\mathbb{Z}^{\alpha}G$, for any element of $g\in G$ of odd order, the restriction of $[\alpha]$ to $\langle g \rangle$ is trivial. Therefore we {\it may and will always from now on choose a normalized representative $\alpha$ such that $u_g^{o(g)}=1$ in $\mathbb{Z}^{\alpha}G$ for all $g$ of odd order. }

\subsubsection*{Deforming a bicyclic unit : construction}\label{ConJRep}
We will use the $G$-module structure of the twisted group ring $R^{\alpha}[G]$ described in \cite[\S 3.2]{ginosar2012semi} $$g(u_x)=u_gu_xu_g^{-1}=\alpha (g,x)\alpha(gxg^{-1},g)^{-1}u_{gxg^{-1}}.$$

For every $x\in G$, the restriction of the conjugation representation from $G$ to the
centralizer $C_{G}(x)$ admits a $1$-dimensional invariant subspace spanned by $u_x$. By
the above this ordinary $1$-dimensional representation $\chi_x:C_{G}(x)\rightarrow R^*$ is given by
$$\chi_x (g)=[u_g,u_x]=\alpha (g,x)\alpha (x,g)^{-1}.$$

For $R=\mathbb{Z}$, for any $x\in G$ we can partition $C_G(x)$ into two disjoint sets 
$$C_x^{+}=\{g\in C_G(x)\mid \chi_x (g)=1 \},\quad C_x^{-}=\{g\in C_G(x) \mid \chi_x (g)=-1 \}.$$
We remark that if $C_x^{-}=\emptyset$ then $x$ (and its conjugacy class) is {\it called $\alpha$-regular} and the set of all $\alpha$-regular conjugacy classes forms a basis for the center of $\mathbb{Z}^{\alpha}G$ (see 
\cite{Oystaeyen})
 An extreme case for that is when $\alpha $ is trivial. In this case $C_x^{-}=\emptyset$ for any $x\in G$.

\begin{example}
For $g,h\in G$ and $\alpha\in Z^2(G,R^*)$,
a straightforward calculation yields that in $R^{\alpha}G$ if $[u_g,u_h]=\zeta \in R^*$ is a root of unity of order $k$, then $k$ is a divisor of the greatest common divisor of the orders of $g$ and $h$ in $G$. Hence, for $\alpha \in Z^2(G,\mathbb{Z}^*)$, all elements of $G$ of odd order are $\alpha$-regular.
\end{example}

Using this we now introduce a new type of units.

\begin{definition}\label{algemene def twisted bicyclic}
With the above notation, the elements of the set
$$\lbrace u_1 + y \wt{u_g}; \, u_1 + \wt{u_g} y   \mid g \in \mc{G}_0^{\alpha}, y \in \Z^{\alpha}[G], \, \supp (y) \subseteq C_g^{-} \rbrace$$
 are called  {\it H-units.} The subgroup of $\U( \Z^{\alpha}[G])$ generated by these is denoted $\mc{H}_{\alpha}(G)$.
\end{definition}

At first it might seem surprising that the above elements are invertible in 
$\Z^{\alpha}[G]$. This is because the square of $y \wt{u_g}$ and $\wt{u_g} y $ are zero. Hence they are unipotent units with inverse $u_1 -  y \wt{u_g}$, resp. $u_1 - \wt{u_g} y$. One way to see that those terms are indeed square zero, is by noticing that $y \wt{u_g} = (o(g)- \wt{u_g}) \frac{y}{o(g)} \wt{u_g}$. This follows from the fact that $u_g^ju_h\wt{u_g} = (-1)^j u_h \wt{u_g}$ for 
$h \in C_g^{-}$  and hence when $o(g)$ is even\footnote{If $o(g)$ is odd, the last sum in (\ref{rewriting a twisted bicyclic}) equals $o(g) - 1$ and hence in general the 
$H$-unit will not be an integral combination. However, as remarked in \Cref{ConJRep}, elements of odd order are $\alpha$-regular 
(i.e. $C_g^{-} = \emptyset$) and thus one can only take $y = 0$. In particular we don't need to specify the order of the elements 
from which we built the $H$-units in order to be a subset of $\Z^{\alpha}[G]$.} and $g \in \mc{G}_0^{\alpha}$
\begin{equation}\label{rewriting a twisted bicyclic}
\begin{array}{rcl}
(o(g) - \wt{u_g})u_h \wt{u_g} & = & u_h \left( o(g) - \sum_{j=0}^{o(g)-1} (-1)^j u_g^j \right) \wt{u_g} \\
& = & u_h \left( o(g) - \sum_{j=0}^{o(g)-1} (-1)^j \right) \wt{u_g} \\ 
& =&  o(g) u_h \wt{u_g}.
\end{array}
\end{equation}
Analogously $\wt{u_g}u_h(o(g) - \wt{u_g}) = o(g) \wt{u_g} u_h$. 
Thus the $H$-units are indeed elements of $\U (\Z^{\alpha}[G])$. 
Furthermore, for $\mathcal{F} = \{ \widehat{u_g} = \frac{1}{o(g)}\wt{u_g} \mid g \in \mathcal{G}_0^{\alpha} \}$, 
every $H$-unit has an appropriate power that is in $\text{GBic}^{\mc{F}}(\Q^{\alpha}[G],\Z)$. 
However, as will be pointed out in \Cref{Gbic not equal to twisted bic}, due to the factor $\frac{y}{o(g)}$ in general
\begin{equation} \label{Gbic in twisted bic}
\mc{H}_{\alpha}(G) \nsubseteq \text{GBic}^{\mc{F}}(\Q^{\alpha}[G],\Z) .
\end{equation}
A concrete example will be given in \Cref{prop description unit twisted group ring}. That $\langle \text{GBic}^{\mc{F}}(\Q^{\alpha}[G],\Z), \mc{H}_{\alpha}(G) \rangle$ is not necessarily obtained by bicylic units, even not up to commensurability, is truly its raison d'être and hence will be crucial for the applications.
\begin{example*}
If $\alpha$ is a trivial cocycle, then $\mc{G}_{j}^{\alpha} = \emptyset = C_g^{-}$ for any $j \neq 0$, $g \in G$ and $\mathcal{G}_0^{\alpha} = G$. Therefore, in that case $\Q^{\alpha}[G]= \Q[G]$ and $\mc{H}_{\alpha}(G) = \{ u_1 \}$. In particular $\langle \text{GBic}^{\mc{F}}(\Q^{\alpha}[G],\Z), \mc{H}_{\alpha}(G) \rangle = \text{Bic}(G)$, i.e. we recover exactly the classical bicyclic units.
\end{example*}

\begin{remark}\label{remarks on h-unit for twisted grp group}
\begin{enumerate}
\item The elements $u_1 + (o(g)- \wt{u_g}) \frac{u_h}{o(g)}\wt{u_g}$ and $u_1 + \wt{u_g}\frac{u_h}{o(g)}(o(g)- \wt{u_g})$ are in fact instances of a common more general construction of units. Indeed, consider $\alpha \in Z^2(G, \Z[\zeta_n]^*)$ with $o(\alpha) \mid n$ and let $\xi$ be a $o(g)$-root of unity inside $\Z[\zeta_n]$. Then $u_1 + u_h \, \wt{\xi u_g}$ is an invertible element of $\Z[\zeta_n]^{\alpha}G$ with inverse $u_1 - u_h \, \wt{\xi u_g}$. Note that the choice $\xi = -1$ yields $u_1 + u_h\, \wt{(-u_g)} = u_1 + \wt{u_g} u_h = u_1 + \wt{u_g}\frac{u_h}{o(g)}(o(g)- \wt{u_g})$, as desired. As we will see in \Cref{subsection H-units in ss alg}, an essential ingredient is that $\wt{[u_g, u_h]} =0$.

\item The $H$ in the name $H$-unit refers to the cohomology group $H^2(G, \Z^{*})$. Following (\ref{rewriting a twisted bicyclic}), the elements in \Cref{algemene def twisted bicyclic} look like roots of bicyclic units. However these elements and those in the previous point fit in a general new generic construction of units which will be developed in \Cref{section h-units}. There the role of the $H^2$, i.e. splitting into extensions, will become more apparent. In loc. cit. the units will in fact be called primitive $H$-units (as they are the smallest of their kind). 

\item Using the twisting as above one could instead have considered a deformed version of the classical bicyclic units, i.e. for any $x \in \Z^{\alpha}[G]$ all the elements of the form $1 + (o(g)- \wt{u_g}) (x + \frac{y}{o(g)})\wt{u_g}$ are units. Such elements include both $H$-units and $\text{GBic}^{\mc{F}}(\Q^{\alpha}[G],\Z)$. However, as hopefully \Cref{section h-units} will convince, it seems to be better to not think in terms of bicyclic units.
\end{enumerate}
\end{remark}

Thanks to (\ref{Gbic in twisted bic}), one can
apply \Cref{Jespers-LEal for twisted bicyclic} to show that $\langle \text{GBic}^{\mc{F}}(\Q^{\alpha}[G],\Z), \mc{H}_{\alpha}(G) \rangle$ contains sufficiently many elementary matrices. It follows from the proof of \Cref{Jespers-LEal for twisted bicyclic} that the value $y \in \Z$ is the one yielded by \Cref{Jespers-Leal generalized}. However by adding the $H$-units, depending on the order of the twisted elements and the number of $\alpha$-irregular elements, one can aim to decrease the value of $y$. 

\begin{question}
What is a formula for the smallest $y\in \N$ such that $1 - e + E_m(y\mc{O}) \leq \langle \text{GBic}^{\mc{F}}(\Q^{\alpha}[G],\Z), \mc{H}_{\alpha}(G) \rangle$?
\end{question}

\subsection{On $H$-units in cokernel of the transgression map}
We now return to the setting of \Cref{background transgression}. More precisely we consider the central extension (\ref{De start CE setting prelim}) corresponding to $[\alpha] \in H^2(G,A)$. We also fix a normalized representant $\alpha \in Z^2(G,A)$ of the form $\alpha(g,h)= \mu(g)\mu(h)\mu(gh)^{-1}$ for a section $\mu$ of $\lambda: \Gamma \twoheadrightarrow G$. 

The $H$-units from the previous section will now allow to construct elements in $\Coker(\wt{\Psi_{\chi}})$. For this take elements $g \in \mc{G}_0^{\Tra_\alpha(\chi)}$ and $h \in C_g^{-}$. Note that the existence of such an element $h$ indirectly assumes that $g$ is $\Tra_\alpha(\chi)$-irregular. In particular $\Tra_\alpha(\chi)$ must be a non-trivial cocycle and hence $\chi  \neq \omega_A$. Furthermore we will also consider a central subgroup $Q$ of $\G$ such that $A \cap Q = 1$. Now with every such triple $(g,h,Q)$ we consider the set
\begin{equation}\label{minimal bicyclic in coker}
\mathcal{H}^{min}_{g,h,Q} := \lbrace 1 +  z u_h \wt{u_g} \in \Z^{\Tra_\alpha(\chi)}[G] \mid z \in \{ 1- u_q \, : \,  q \in Q \text{ or } u_q = 0 \} \rbrace.
\end{equation}

\begin{proposition}\label{non-trivial elements in cokernel}
Let $\chi \in \Hom(A,\Z^*),  g \in \mc{G}_0^{\Tra_\alpha(\chi)}, h \in C_g^{-}$ and $Q$ a central subgroup of $\G$ such that $A \cap Q = 1$. Suppose $\U(\Z G)$ is finite, then the following hold
\begin{enumerate}
\item if $q \in Q \setminus \langle \mu(g) \rangle$ or $u_q=0$, then $1 + (1-u_q) u_h \wt{u_g} \notin \im \left( \restr{\wt{\Psi_{\chi}}}{\U (\Z [\G])} \right)$
\item if $Q \cap\langle \mu(g) \rangle = 1 = A \cap\langle \mu(g) \rangle$, then $\langle \, \mathcal{H}^{min}_{g,h,Q} \, \rangle \cong C_{2}^{|Q|}$ as subgroup of $\Coker\left( \restr{\wt{\Psi_{\chi}}}{\U (\Z [\G])} \right)$.
\end{enumerate} 
\end{proposition}
Keeping in mind that $\lbrace 1-q \mid 1 \neq q \in Q \rbrace$ is a $\Z$-basis of $\ker (\omega_Q)$, a direct computation yields that 
$$\langle \, \mathcal{H}^{min}_{g,h,Q} \, \rangle = \lbrace 1 + y u_h \wt{u_g} \in \Z^{\Tra_\alpha(\chi)}[G] \mid y \in \ker (\omega_Q) \cup \{1 \} \rbrace.$$
In particular, when $Q \cap \langle u_g \rangle = 1$ then $\langle \, \mathcal{H}^{min}_{g,h,Q} \rangle$ can be viewed as subgroup of $\Coker(\wt{\Psi}_{ker})$. Note also that $qh \in C_g^{-}$ for all $q \in Q$ and hence $\langle \, \mathcal{H}^{min}_{g,h,Q} \, \rangle \leq \mc{H}_{\alpha}(G).$
\begin{proof}[Proof of \Cref{non-trivial elements in cokernel}.]
Due to the centrality of $z = 1-u_q$ it follows that $1 + z u_h \wt{u_g}$ is trivial if and only if $u_q \in \langle u_g\rangle$ (in particular for $z =1$ it is never trivial). On turn this is equivalent to $q \in Q \setminus \langle \mu(g) \rangle$ (or $q$ trivial) thanks to the assumption $o(g) = o(u_g)$. From now we assume that this is the case and let $y$ be a sum of different such elements ($1-u_q$ with $q\in Q\setminus \langle \mu(g) \rangle$ or $1$). Then the unit $ 1 + (o(g)- \wt{u_g}) y \frac{u_h}{o(g)} \wt{u_g}$ is also non-trivial. We claim that more generally all such element are not in the image. 

An element in $\Psi_{\chi}^{-1}(1 + y u_h \wt{u_g})$ is of the form $\tau := 1 + y \mu(h) (1+ \mu(g) + \cdots + \mu(g)^{o(g)-1}) + x$ for some $x\in \ker(\Psi_{\chi})$ and due to the non-triviality $\tau \neq 1$. We need to prove that $\tau \notin\U(\Z \Gamma) $ for any $x$. Assume otherwise and start by decomposing $x$ according to \Cref{ring homom from Tra}:
$$x = \sum_{a \in \ker(\chi)} x_a (a - 1) + \sum_{a \in A\setminus \ker(\chi)} x_a (a+1)$$
with $x_a = t_a + \sum_{1 \neq w\in G}  t_{w,a} \mu(w) \in \Z[\Gamma]$.
  
Therefore, since $\mu(g)^{o(g)} \in A$,
$$\omega_A(\tau) = 1+ y h \wt{g} + 2 \sum_{a\in A\setminus \ker(\chi)} x_a.$$
Note that $y h \wt{g}\neq 0$ and in fact all its coefficients are $\pm 1$ due to all the assumptions. Furthermore, $\omega_A(\tau)$ is a torsion unit since $\U(\Z G)$ is finite. Since the coefficient of the identity element of $G$ is $1 + 2 \sum_{a\in A\setminus \ker(\chi)} t_a$, which is non-zero, Theorem~\ref{Berman-Higman twisted} yields that $\omega_A(\tau) = \pm 1$. However $\omega_A(\tau) = 1+ y h \wt{g} \mod 2 \neq 1$. This yields the desired contradiction and finishes the proof of the claim and in particular of $(1)$. 

For the second part, assume that $Q \cap \langle \mu(g) \rangle = 1 = A \cap  \langle \mu(g) \rangle$. Note that $g \in \mc{G}_0^{\Tra_\alpha(\chi)}$ implies that $\mu(g)^{o(g)} \in \ker (\chi)$. The stronger condition $A \cap  \langle \mu(g) \rangle =1$, equivalently $\mu(g)^{o(g)} =1$,  is assumed to have that $\Psi_{\chi}(\wt{\mu(g)}) = \wt{u_g}$. Denote $b_z :=  1 + z u_h \wt{u_g}.$ Note that $b_{z_1}b_{z_2} = 1 + (z_1 + z_2) u_h \wt{u_g}$. Moreover, all the $b_z$ commute and in fact $\langle \, b_{1-q} \mid 1 \neq q \in Q \rangle$ is isomorphic to the additive group $\ker (\omega_Q)$. By the general claim every $b_{z_1}\ldots b_{z_l}$, for different $z_i \in \{1, 1-u_q \}$, is not in the image. Thus it remains to prove that the square of every $b_z$ is attained. For this rewrite 
$$
\begin{array}{rcl}
b_{1-q}^2 & =& 1 + 2 (u_h - u_{qh}) \wt{u_g}\\
&= & \left( 1 + 2 u_h\wt{u_g} \right) \left( 1+ 2 u_{qh}\wt{u_g} \right)^{-1} \\
& = & \wt{\Psi_{\chi}}(u_1) \, \wt{\Psi_{\chi}}(u_2^{-1})
\end{array}
$$
where $u_1 = 1 + (1 - [\mu(g),\mu(h)]) \mu(h)\wt{\mu(g)}$ and $u_2 = 1 + (1 - [\mu(g),\mu(qh)]) \mu(qh)\wt{\mu(g)}$. Note that  $[\mu(g),\mu(h)], [\mu(g),\mu(qh)] \in A \setminus \ker(\chi)$ because $h,qh \in C_g^{-}$. That $u_1$ and $u_2$ are units can be seen by rewriting them as bicyclic units, e.g. $u_1 = 1 + (1 - \mu(g)) \mu(h) \wt{\mu(g)}$.
\end{proof}

We expect that $H$-units contribute to $\Coker(\Psi_{\chi})$ in full generality, in the way they do in \Cref{non-trivial elements in cokernel} . 

\begin{question}\label{min twists in cokernel question}
Is the conclusion of \Cref{non-trivial elements in cokernel} also valid without the condition that $\U(\Z G)$ is finite?
\end{question}
%

Most likely, in general other type of elements can be contained in $\Coker(\wt{\Psi_{\chi}})$. However in \Cref{section elem ab description} we will see that for $\G$ an extension of $C_2$ with an elementary abelian $2$-group $G  = C_2^n$ that these elements will generate the full cokernel. Consequently, in that case $\Coker(\wt{\Psi}_{ker}) = \ker(F_{Cok}) \cong C_2^{n-2}$. It would be interesting to when this happens. 

\begin{remark}\label{Remark h-unit and opp in coker}
    With exactly the same proof, the statement of \Cref{non-trivial elements in cokernel} also holds for the $H$-units $1 +\wt{u_g} u_h z$. Hence \Cref{min twists in cokernel question} can also be formulated for these elements. In fact it is also an interesting question to understand the image of $\langle b_1 := 1 + u_h \wt{u_g} , b_2 := 1 +\wt{u_g} u_h\rangle$. In case $o(u_h)=o(u_g)=o(g)= 2$, then a direct computation shows that $b_1b_2^{-1} \in \im \left( \restr{\wt{\Psi_{\chi}}}{\U (\Z [\G])} \right).$ Hence $\langle b_1, b_2 \rangle \cong C_2$ as subgroup of $\Coker\left( \restr{\wt{\Psi_{\chi}}}{\U (\Z [\G])} \right)$. It is however very likely that this is an order $2$ phenomena.
\end{remark}

\section{Method to describe \texorpdfstring{$\mc{U}(\Z^{\alpha}[G\times C_2^{n}])$}{}}\label{subsection reduing elem ab 2-subgroups}
Let $G$ be a finite group and $$C_2^n = \langle x_1 \rangle \times \cdots \times \langle x_n \rangle$$ an elementary abelian $2$-group. In this section we will consider a, potentially trivial, $2$-cohomology class $[\alpha]$ in the image of the inflation map 
\begin{equation}\label{inflation map for times C2n}
\Inf: H^2(G,\Z^{*}) \rightarrow H^2(G \times C_2^n, \Z^{*}).
\end{equation}
In particular, by \Cref{cor imabe inf yields central}, the subgroup $C_2^n=\langle u_{x_i}\rangle$ is central in $\U(\Z^{\alpha}[G\times C_2^{n}])$. For such cocycles we will now present a method to describe $\mc{U}(\Z^{\alpha}[G\times C_2^n])$ whenever $\mc{U}(\Z^{\alpha}[G])$ is known.\vspace{0,1cm}

{\it\noindent Convention:} $C_2^{i}$ denotes the concrete subgroup $\langle x_1, \ldots, x_i \rangle$. In particular when we write $\Z^{\alpha}[G \times C_2^{i}]$ we truly mean the subring of $\Z^{\alpha}[G \times C_2^{n}]$ generated by $G \times C_2^{i}$.\vspace{0,2cm}

To start, consider for every $1 \leq i \leq n$ the natural projection $$\psi_i: \Z^{\alpha}[G\times C_2^{i}] 
\rightarrow \Z^{\alpha}[G\times C_2^{i-1}]$$ induced by mapping $x_i$ to $1$, which is a ring 
morphism\footnote{This is a ring-morphism due to the centrality of $x_i$. In particular for arbitrary 
$2$-cocycles this method does not work.}. Note that, due to the centrality of $x_i$, 
$\Z^{\alpha}[G\times C_2^{i}]= \left(\Z^{\alpha}[G\times C_2^{i-1}]\right)[\langle x_i \rangle]$ and 
therefore $\psi_i$ is globally defined by $\psi_i (u+v x_i) = u+v$ for 
$u,v \in \Z^{\alpha}[G\times C_2^{i-1}]$. By definition $\psi_i$ induces an epimorphism 
$$
\wt{\psi}_i: \mc{U}(\Z^{\alpha}[G\times C_2^{i}]) \rightarrow \mc{U}(\Z^{\alpha}[G \times C_2^{i-1}]).
$$
Accordingly, we obtain the following splitting.

\begin{lemma}\label{lem splitting wrt kernels K_i}
Let $K_i = \ker (\wt{\psi}_i)$. Then, $\mc{U}(\Z^{\alpha}[G\times C_2^{i}]) = K_i \rtimes \mc{U}(\Z^{\alpha}[G \times C_2^{i-1}]).$ Consequently,
$$\mc{U}(\Z^{\alpha}[G\times C_2^{n}])\cong K_n \rtimes \Big( K_{n-1} \rtimes \big( \cdots K_2 \rtimes (K_1 \rtimes \mc{U}(\Z^{\alpha}G)) \big) \Big).$$
\end{lemma}
\begin{proof}
Since $\wt{\psi}_i$ is surjective we obtain an extension 
$$ 1 \rightarrow K_i \hookrightarrow \mc{U}(\Z^{\alpha}[G\times C_2^{i}]) \xrightarrow{\wt{\psi}_i} \mc{U}(\Z^{\alpha}[G \times C_2^{i-1}]) \rightarrow 1.$$
Clearly, this extension is split by the identity map. The second part now follows by iteration.
\end{proof}

In view of \Cref{lem splitting wrt kernels K_i} we will now concentrate on describing all the kernels $K_i$. To this end we define the groups
$$U_i = \{ 1+ 2^{i} u \mid u \in \Z^{\alpha}G \} \cap  \mc{U}(\Z^{\alpha}G).$$
Our goal is to show that $K_i$ can be built from isomorphic copies of the groups $U_1, \ldots, U_i$ and the unit group $\mc{U}(\Z^{\alpha}G)$. Note that, by \Cref{Berman-Higman twisted}, $U_i$ is torsion-free when $i \geq 2$ and the only torsion elements in $U_1$ are $\pm 1$. 

As a tool we need to consider more generally the groups
$$U_{k,j} = \{ 1+ 2^{k} u \mid u \in \Z^{\alpha}[G\times C_2^{j}] \} \cap  \mc{U}(\Z^{\alpha}[G\times C_2^j]).$$
Note that $U_{k,0} = U_k$ and $U_{k_1,j_1} \leq U_{k_2,j_2}$ for $k_2 \leq k_1$ and $j_1 \leq j_2$. Next to these, the projections 
$$\varphi_i: \Z^{\alpha}[G \times C_2^{i}] \rightarrow \Z^{\alpha}[G \times C_2^{i-1}] \mbox{ with } x_i \mapsto -1$$
will also be instrumental. The induced epimorphism on the unit groups is denoted by $\wt{\varphi}_i$.
\begin{lemma}\label{lem connection between the U groups}
For all $1 \leq k,j$, with notations as above, we have that
\begin{itemize}
    \item $K_j \cong \wt{\varphi}_{j}(K_j) = U_{1,j-1},$
    \item $U_{k,j} \cong U_{k+1, j-1} \rtimes U_{k,j-1}$.
\end{itemize}
\end{lemma}
\begin{proof}
By definition,
$$K_j= \{ 1 + u(1-x_j) \mid u \in \Z^{\alpha}[G\times C_2^{j-1}] \} \cap \mc{U}(\Z^{\alpha}[G\times C_2^{j}]).$$
Using the centrality of $x_j$, one gets that $\Z^{\alpha}[G\times C_2^{j}] = \left( \Z^{\alpha}[G\times C_2^{j-1}] \right)\langle x_j \rangle$ and thus one immediately obtains that $\wt{\varphi}_{j}$ is injective on $K_j$. Consequently, $K_j \cong \wt{\varphi}_{j}(K_j)$ and more explicitly
$$\wt{\varphi}_j(K_j) = \{ 1+2u \mid u \in \Z^{\alpha}[G\times C_2^{j-1}] \} \cap \mc{U}(\Z^{\alpha}[G\times C_2^{j-1}]) = U_{1,j-1}.$$
Note that $\wt{\varphi}_j(K_j)$ contains all the units above because $\Z^{\alpha}[G\times C_2^{j-1}]$ is truly meant as the subring of $\Z^{\alpha}[G\times C_2^{j}]$. 

For the second statement, remark that $\wt{\psi}_j$ is the identity map on $U_{k,i}$ for all $0 \leq i < j$ and all $k$. Also, $\wt{\psi}_{j}(U_{k,j}) = \wt{\psi}_{j}(U_{k,j-1})$ where $U_{k,j-1}$ is the subgroup $U_{k,j} \cap \U (\Z^{\alpha}[G \times \langle x_1, \ldots, x_{j-1} \rangle])$. Combined we obtain the internal splitting
\begin{equation}\label{splitting Ukj}
U_{k,j} = \ker (\restr{\wt{\psi}_{j}}{U_{k,j}}) \rtimes U_{k,j-1}.
\end{equation}
Next, since $\ker (\restr{\wt{\psi}_{j}}{U_{k,j}}) = \{ 1 + 2^{k}u(1-x_j) \mid u \in \Z^{\alpha}[G \times C_2^{j-1}] \} \cap U_{k,j}$ the map $\wt{\varphi}_{j}$ is injective on it and so $\wt{\varphi}_{j}\left( \ker (\restr{\wt{\psi}_{j}}{U_{k,j}}) \right) \cong U_{k+1,j-1}$, finishing the proof.
\end{proof}

An iterative process of the previous lemma decreases the second index of $U_{k,j}$, with the cost of producing extra complements. Hence it now readily follows that we can reduce to the groups $U_i$. For the applications to classical group rings in \Cref{section h-units} we however need an explicit internal splitting of $K_i$. As indicated in the proof of \Cref{lem connection between the U groups} such a splitting is available for $U_{k,j}$. Inspired by this we define for every tuple $\vec{j}= (j_{1}, \ldots, j_{i-1}) \in \{0,1 \}^{i-1}$ the following:
\begin{equation}\label{the pricse pieces of Ki}
K_{\vec{j}} := \{ 1 + u (1-x_1)^{j_{1}} \cdots (1-x_{i-1})^{j_{i-1}}(1-x_i) \mid u \in \Z^{\alpha}[G] \} \cap \U(\Z^{\alpha}[G \times C_2^{i}]).
\end{equation}
Using that $gx_{l}(1-x_l) = -g (1-x_l)$, it is easily verified that $K_{\vec{j}}$ is a normal subgroup of $\U (\Z^{\alpha}[G \times \langle x_l : j_l \neq 0 \rangle])$. In general however it won't be normal in $K_i$. 

\begin{theorem}\label{th reducing off an elemantary abelian subgroup}
Let $G$ be a finite group, $[\alpha] \in \Ima(\Inf)$ as in (\ref{inflation map for times C2n}). Then, for all $i \geq 2$, 
\begin{equation}\label{decomp in precise pice without saying order}
K_i :=  \underset{\vec{j} \in \{0,1\}^{i-1}}{\overrightarrow{\mathlarger{\mathlarger{\rtimes}}}} K_{\vec{j}}
\end{equation}
where the internal semi-direct product is with respect to a specific ordering on $\{0,1\}^{i-1}$, and $K_{\vec{j}} \cong U_{1+ \sum_{t=1}^{i-1} j_t}$ for $\vec{j}= (j_{1}, \ldots, j_{i-1}) \in \{0,1 \}^{i-1}$.
Furthermore,
$$\begin{array}{rl}
\mc{U}(\Z^{\alpha}[G\times C_2^{n}])& = K_n \rtimes \Big( K_{n-1} \rtimes \big( \cdots K_2 \rtimes (K_1 \rtimes \mc{U}(\Z^{\alpha}G)) \big) \Big) \\
& = \Big( \big( (N_n \rtimes N_{n-1}) \rtimes \cdots \big) \rtimes N_1 \Big) \rtimes  \big( \mc{U}(\Z^{\alpha}[G]) \times \langle x_1, \ldots, x_n \rangle \big)
\end{array}$$
for some torsion-free normal subgroup $N_i$ of $K_i$ such that $K_i \cong N_i \times \langle x_i \rangle$.
\end{theorem}
\begin{remark*} The ordering in (\ref{decomp in precise pice without saying order}) is deducible from the proof. In particular, in terms of the $U_i$ the decomposition of $K_i$ can be defined recursively as follows: $K_1 \cong U_1$. For $i \geq 2$,  $K_i \cong \mc{U}_i \rtimes K_{i-1}$ with $\mc{U}_2 := U_2$ and $\mc{U}_i := \mc{U}_{i-1}[1] \rtimes \mc{U}_{i-1}$ where $\mc{U}_{i-1}[1] $ looks like $\mc{U}_{i-1}$ but with all indices increased by one. For example, $K_2 \cong U_2 \rtimes K_1, K_3 \cong (U_3 \rtimes U_2) \rtimes K_2, K_4 \cong \left( (U_4 \rtimes U_3) \rtimes (U_3 \rtimes U_2) \right) \rtimes K_3$ and
$$K_5 \cong (\Big((U_5 \rtimes U_4)\rtimes (U_4 \rtimes U_3)\Big)\rtimes \Big( (U_4 \rtimes U_3)\rtimes (U_3 \rtimes U_2)\Big))\rtimes K_4.$$
\end{remark*}
\begin{remark*}
For certain class of nilpotent groups $G$, Jespers, Leal and del R\'io give in \cite[Proposition 5]{JesLealRio} a description of $\U (\Z G)$ in terms of some subnormal series. The $2$-groups in their class are of the form $H \times C_2^n$ for some $n$ and $H \leq G$. In that case their subnormal series coincides with the one in \Cref{th reducing off an elemantary abelian subgroup}.
\end{remark*}

\begin{proof}[Proof of \Cref{th reducing off an elemantary abelian subgroup}.]
The description of $K_i$ follows from an iterative use of \Cref{lem connection between the U groups} or rather the explicit form in (\ref{splitting Ukj}). In fact we will proof more generally that 
$$U_{k,i} = \underset{\vec{j} \in \{0,1\}^{i}}{\overrightarrow{\mathlarger{\mathlarger{\rtimes}}}} U_{k, \vec{j}}$$
where $U_{k, \vec{j}} = \{ 1 + 2^k u (1-x_1)^{j_{1}} \cdots (1-x_{i})^{j_{i}} \mid u \in \Z^{\alpha}[G] \} \cap \U(\Z^{\alpha}[G \times C_2^{i}])$. To start one uses
 $$K_i \cong \wt{\varphi}_i(K_i) \cong U_{1,i-1} \cong U_{2,i-1} \rtimes U_{1,i-2} \cong U_{2,i-1} \rtimes K_{i-1}.$$
Explicitly the copy of $U_{2,i-1}$ in $K_i$ is given by $\{ 1 + u (1-x_{i-1})(1- x_i) \in K_i \mid u \in \Z[G \times \langle x_1, \ldots, x_{i-2} \rangle ] \}$ and the copy of $K_{i-1}$ is $\{ 1 + v (1-x_i) \in K_i \mid v \in \Z[G \times \langle x_1, \ldots, x_{i-2} \rangle ]\}$. In terms of $U_{1,i-1}$ the $1-x_i$ is identified with a factor $2$ via $\wt{\varphi}$. Now, if $i = 2$ the procedure finishes here and yields the decomposition $(\ref{decomp in precise pice without saying order})$. Next, for an arbitrary $i$ one case use now induction on the copy of $K_{i-1}$ inside $K_i$. This yields all the terms $K_{\vec{j}}$ such that $j(i-1) = 0$ (i.e. no $x_{i-1}$ in the support). The other terms will appear by applying induction to $U_{2,i-1}$ (in the recursive process this corresponds to applying the the second part of \Cref{lem connection between the U groups} to $U_{2,i-1}$). 

The first equality of the second part was proven in \Cref{lem splitting wrt kernels K_i}. Furthermore, it is easy to see that the brackets can be rearranged to $\U(\Z^{\alpha}[G\times C_2^n]) = \Big( \big( (K_n \rtimes K_{n-1}) \rtimes \cdots \big) \rtimes K_1 \Big) \rtimes  \mc{U}(\Z^{\alpha}G)$. Hence due to the centrality of the subgroup $C_2^n$, it remains to prove the existence of such $N_i$ that is normal in $K_{i-1}$. For this notice that $(\ref{decomp in precise pice without saying order})$ yields only one subgroup isomorphic to $U_1$, namely $K_{\vec{0}} = \{ 1 + (1-x_i) u \mid u \in \Z^{\alpha}[G] \} \cap 
\mc{U}(\Z^{\alpha}[G \times \langle x_i \rangle]).$  All the other $K_{\vec{j}}$ will be isomorphic to $U_l$ with $2 \leq l \leq i$ and hence are torsion-free. Now, under the isomorphism $K_{\vec{0}} \cong  U_1$, the subgroup $\langle x_i \rangle$ corresponds to $\{ \pm 1 \}$ which is the only torsion in $U_1$. Therefore $K_{\vec{0}} \cong F_i \times \langle x_i \rangle$ for some torsion-free subgroup $F_i$. Taking $F_i$ together with all the $K_{\vec{j}}$ with $\vec{j}\neq \vec{0}$ we obtain the desired torsion-free subgroup $N_i$ of $K_i$ such that $K_i \cong N_i \times \langle x_i \rangle$. 
\end{proof}

The decomposition obtained allows to transfer properties of $\mc{U}(\Z^{\alpha} G)$ to $\mc{U}(\Z^{\alpha}[G \times C_2^n])$. Earlier interesting results in that line can be found in \cite{BKS, Her, Low, HofKim}.

\begin{corollary}\label{Properties that get inherited with times elem ab}
Let $G$ be a finite group, $[\alpha] \in \Ima(\Inf)$ as in (\ref{inflation map for times C2n}). If $\mc{U}(\Z^{\alpha} G)$ satisfies one of the following properties:
\begin{enumerate}
    \item $\pm G$ has a normal (torsion-free) complement in $\mc{U}(\Z^{\alpha} [G])$
    \item it is commensurable with a direct product of free-by-free groups
    \item has a non-trivial amalgam decomposition
    \item $G$ satisfies (HSP)\footnote{(HSP) stands for the Higman subgroup property, i.e. any finite subgroup in $V(\Z^{\alpha} G)$ is isomorphic to a subgroup in $G$. See also \cite{MarSIP}.}
\end{enumerate}
then the same holds for $G\times C_2^n$ and $\mc{U}(\Z^{\alpha}[G \times C_2^n])$.
\end{corollary}

In the case of untwisted\footnote{The classical Zassenhaus conjectures are stated for group rings, however also literally makes sense for twisted group rings. For example (ZC3) would be that all finite $H \leq U(\Z^{\alpha}[G])$ are conjugated over $\Q^{\alpha}[G]$ to a subgroup of $G$.}  group rings it was proven in \cite{HofKim} that the first Zassenhaus conjecture (ZC1) is preserved. It would be interesting to prove so for twisted group rings. Especially since the counterexample  to (ZC1) of Eisele and Margolis \cite{MR3866907}, positive instances of the Zassenhaus conjectures are of special interest. For a recent survey see \cite{MarRioSurvey}.

\begin{proof}[Proof of \Cref{Properties that get inherited with times elem ab}]
The first part directly follows from the last decomposition in \Cref{th reducing off an elemantary abelian subgroup}, since $ \big( (N_n \rtimes N_{n-1}) \rtimes \cdots \big) \rtimes N_1$ is torsion-free.

Next, for the second statement, let $A$ be a finite dimensional semisimple $\Q$-algebra and $\mc{O}$ an order in $A$. The desired statement is a direct consequence of the fact that `$\U( \mc{O})$ commensurable with a direct product of free-by-free groups' is fully determined by the type of the simple components of the algebra $A$, as follows from \cite[Theorem 2.1]{JPdRRZ}. For instance, \cite[Lemma 3.1. \& Proposition 3.3.]{JPdRRZ} (or \cite{EricAngel2}) yields that\footnote{To conclude this it is used that free groups are exactly the groups with cohomological dimension one \cite{Sta,SwanVCD}. Due to this free-by-free groups have virtual cohomological dimension at most two which is the content of \cite[Section 3]{JPdRRZ}.} the property holds if and only if the simple quotients of $A$ are either a field, a totally definite quaternion algebra or a matrix algebra $\Ma_2(D)$ with $D$ in a list of certain quadratic imaginary extensions of $\Q$ or certain quaternion algebras over totally real number fields. However if $[\alpha]$ is in the image of (\ref{inflation map for times C2n}), then $$\Q^{\alpha}[G\times C_2^n] \cong \Q^{\alpha}[G] \otimes_{\Q} \Q[C_2^n] \cong  \underbrace{\Q^{\alpha}[G] \oplus \cdots \oplus \Q^{\alpha}[G]}_{2^n-times}.$$
Thus the isomorphism types of simple quotients of $\Q^{\alpha}[G\times C_2^n]$ is the same as for $\Q^{\alpha}[G]$.

Concerning the third statement, note that, due to \Cref{th reducing off an elemantary abelian subgroup}, $\mc{U}(\Z^{\alpha}[G])$ is an epimorphic image of $\mc{U}(\Z^{\alpha}[G \times C_2^n])$ and hence its amalgam decomposition can be lifted along the epimorphism. 

For the fourth statement, first notice that by induction we may assume that $n=1$. Hence, with the above 
notation, $\mc{U}(\Z^{\alpha}[G \times C_2])=K_1 \rtimes \mc{U}(\Z^{\alpha}[G])$ with $K_1 \cong N_1 \times \langle x_1 \rangle$ where $N_1$ is a torsion-free normal subgroup.  For (SIP), assume now that $H$ is a finite subgroup of $\mc{U}(\Z^{\alpha}[G \times C_2])$. As $N_1$ is torsion-free, $H \cong H/(H\cap N_1)$ and so we may consider $H$ as a subgroup of $\mc{U}(\Z^{\alpha}[G \times C_2]) / N_1 \cong \U(\Z^{\alpha}[G]) \times C_2.$ Consequently, there exist finite index sets $I$ and $J$ and $t_i, k_j \in \U(\Z^{\alpha}[G])$ such that $H/(H\cap N_1)= \{ t_i, x_1 k_j \mid i \in I, \, j \in J \}$. Now consider the larger group $T = \langle t_i, k_j, x_1 \mid i \in I, \, j \in J \rangle$. Using that $x_1$ is a central element of order $2$, it is directly checked that $T$ is a (finitely generated) torsion group.  Since it is also linear, $T$ is in fact finite.  By assumption $\langle t_i, k_j \mid i \in I, \, j \in J \rangle$ is isomorphic to a subgroup of $G$ and hence $T$ to one of $G \times C_2$. Consequently, also the smaller group $H \cong H/(H\cap N_1)$ also as needed.  
\end{proof}

\begin{remark} \label{remark that preserved with direct prod with elem ab}
The proof of (2) in \Cref{Properties that get inherited with times elem ab} in fact shows that any group theoretical property $\mc{P}$ that can be read off the Wedderburn-Artin components (i.e. the simple quotients) is inherited. Other examples of such properties are: Kazhdan's property $(T)$, property FAb and HFA (see \cite{BJJKT2,BJJKT}). In general, good candidates for such properties are the ones that are constant on commensurability classes.
\end{remark}
Following \Cref{Properties that get inherited with times elem ab} the property to have a (torsion-free) normal complement 
is preserved, however a concrete description thereof seems difficult. For example, in \cite{Li} it was proven that no 
normal complement of the trivial units in $\U(\Z[D_8\times C_2 \times C_2])$ is generated by bicyclic units, although 
a normal complement in $\U(\Z[D_8\times C_2])$ is $\Bic(D_8 \times C_2)$ \cite{Jespers1995}. For $D_8\times C_2 \times C_2$ it is unknown whether biyclic units nevertheless form a subgroup of finite index. Recently B\"achle, Maheshwary and Margolis 
have proven \cite[Theorem A]{BMM} that $\rank B^{ab} = \rank \ZZ (\U (\Z G))$ where $B$ is the group generated by bicyclic and Bass units in $\Gamma$, $B^{ab}$ denotes the abelianisation of $B$ and the rank is as finitely generated abelian group. As a consequence, they deduced that if the bicyclic units are of finite index\footnote{This is equivalent to say that 
the bicyclic units together with the Bass units are of finite index in $\U (\Z \G)$ due to a Theorem of Bass and Milnor, see \cite[Theorem 11.1.2. \& Prop. 9.5.11.]{EricAngel1}. } in $\SL_1 (\Z \G)$ then $\rank \U(\Z \G)^{ab} = \rank \ZZ (\U (\Z \G))$. 
One can read this result as a new method to detect whether the bicyclic units are of infinite index in $\SL_1 (\Z \G)$, which happens if $\rank \U (\Z \G)^{ab} \gneq \rank \ZZ (\U (\Z \G))$. This motivates the following question.
\begin{question}\label{question on abelianisation}
With notations as above, what is the connection between $\rank \U (\Z^{\alpha}G)^{ab} $ and $\rank \U (\Z^{\alpha}[G\times C_2^n])^{ab}$?
\end{question}

\section{Full description for elementary abelian 2-groups}\label{section elem ab description}

Let $G$ be an elementary abelian $2$-group of rank $n+2$,
\begin{equation}\label{eq:genG}
G=C_2\times C_2\times \ldots \times C_2=\langle g \rangle \times
\langle h \rangle \times \langle x_1 \rangle \times\ldots 
\times \langle x_n \rangle,
\end{equation}
and consider the cohomology class $[\alpha]\in H^2(G,C_2)$ determined by the values 
\begin{equation}\label{eq:alpharelations}
[u_g,u_h]=-1,\quad u_g ^2= u_h ^2= u_{x_i}^2 = [u_g,u_{x_i}]=[u_h,u_{x_i}]=[u_{x_i},u_{x_j}]=1,
\end{equation}
for any $1\leq i \neq j \leq n$. 
We can think of this $[\alpha ]$ as the cohomology class determining the group $\Gamma \cong D_8\times
C_2^n$ and  the twisted group ring $\Z ^{\alpha}G$.
 Recall our standing convention of choosing a normalized $2$-cocycle representant of $[\alpha]$.

To start, we handle the case that $n=0$, i.e. we describe $\mc{U} (\Z^{\alpha}[\langle g,h\rangle ])$.

\subsection{The starting case \texorpdfstring{$C_2 \times C_2$}{}}\addtocontents{toc}{\protect\setcounter{tocdepth}{1}} \label{subsection start case elemen ab}

Denote by $\wt{D}$ the subalgebra of $M_2(\mathbb{Z})$ consisting of the elements
$$\left\{\left( \begin{array}{cc}
a & b  \\
c & d
\end{array}\right) \in \Ma_2(\Z) \mid a\equiv_2 d, \, b \equiv_2 c \right\}.$$
It easily can be  verified that
\begin{equation}\label{eq:dihedralalg}
\wt{D}=\left\{\left( \begin{array}{cc}
m+n & k+r  \\
k-r & m-n
\end{array}\right) \in \Ma_2(\Z) \mid m,n,k,r\in \mathbb{Z}\right\}.
\end{equation}

\noindent  Consider now the $\mathbb{Z}$-linear map $\phi:\mathbb{Z}^{\alpha}[C_2 \times C_2] \rightarrow \wt{D}$ defined by 
\begin{equation}\label{def phi from section 8}
 u_1\mapsto \left( \begin{array}{cc}
1 & 0  \\
0 & 1
\end{array}\right),
\quad u_g\mapsto \left( \begin{array}{cc}
1 & 0  \\
0 & -1
\end{array}\right),
\quad u_h\mapsto \left( \begin{array}{cc}
0 & 1  \\
1 & 0
\end{array}\right),
\quad u_{gh}\mapsto \left( \begin{array}{cc}
0 & 1  \\
-1 & 0
\end{array}\right)
\end{equation}
which can easily be seen to be a ring morphism.

\begin{proposition} \label{Prop ring iso type twisted group ring}
With notations as above, $\Z^{\alpha}[C_2 \times C_2] \cong \wt{D}$ as rings. 
\end{proposition}
\begin{proof}
The map $\phi$ is surjective because 
$$A=\phi \left(mu_1+nu_g+ku_h+ru_{gh} \right),$$
for an arbitrary element  $A=\left( \begin{array}{cc}
m+n & k+r  \\
k-r & m-n
\end{array}
\right)
$  of  $\wt{D}$ (see~\eqref{eq:dihedralalg}). Injectivity follows from a linear independence argument.
\end{proof}

Next we will obtain a presentation of the unit group $\mathcal{U}(\mathbb{Z}^{\alpha}[C_2 \times C_2])$. 
Already in this small example the role of $H$-units is decisive, as they will yield a normal complement for the trivial units. 
As expanded at the end of this section, see \Cref{Gbic not equal to twisted bic}, the generalized bicyclic units together 
with the trivial units are not sufficient to generate the full unit group. Furthermore, without them we would not be 
able to handle later on, even not up to commensurability, the elementary abelian $2$-group case.

\begin{proposition}\label{prop description unit twisted group ring}
Let $\alpha$ be the $2$-cocycle determined by the values $(1,1,-1)$ on $(u_g^2, u_h^2, [u_g,u_h])$. Then $$\mathcal{U}(\Z^{\alpha}[ C_2 \times C_2 ]) = F_2 \rtimes D_8$$
where $D_8 = \langle u_g, u_h \rangle$ and $F_2 = \langle v,w \rangle$ is a free group of rank $2$ generated by the $H$-units $v= u_1 + u_h - u_{gh}$ and $w = u_1 + u_h + u_{gh}$. The action of $D_8$ on $F_2$ is defined by $u_h^{-1}vu_h=w, u_h^{-1}wu_h=v$ and $u_g^{-1}xu_g = x^{-1}$ for $x= v,w$. Furthermore,
$$\SL_1(\Z^{\alpha}[ C_2 \times C_2 ]) = \mc{H}_{\alpha}(C_2 \times C_2) = \langle v,w, u_{gh} \rangle.$$
In particular, it is a subgroup of index $2$ in $\mathcal{U}(\Z^{\alpha}[ C_2 \times C_2 ])$. Also, $\phi(\mathcal{U}(\Z^{\alpha}[ C_2 \times C_2 ]))$ has index $3$ in $\GL_2(\Z).$
\end{proposition}

Recall that $\SL_1(\Z^{\alpha}[ C_2 \times C_2 ])$ denotes the group of reduced norm $1$ elements, see (\ref{def of SL1}).

\begin{proof}
The elements $v$ and $w$ are examples of $H$-units, see \Cref{algemene def twisted bicyclic} and 
(\ref{rewriting a twisted bicyclic}). Moreover, their respective image under $\phi$ is equal to 
$$\phi(v)=\phi(1+u_h-u_{gh})=
\left( \begin{array}{cc}
1 & 0  \\
0 & 1
\end{array}
\right)
+
\left( \begin{array}{cc}
0 & 1  \\
1 & 0
\end{array}
\right)
-
\left( \begin{array}{cc}
0 & 1  \\
-1 & 0
\end{array}
\right)
=
\left( \begin{array}{cc}
1 & 0  \\
2 & 1
\end{array}
\right)
$$

and
$$\phi(w)=\phi(1+u_h+u_{gh})=
\left( \begin{array}{cc}
1 & 0  \\
0 & 1
\end{array}
\right)
+
\left( \begin{array}{cc}
0 & 1  \\
1 & 0
\end{array}
\right)
+
\left( \begin{array}{cc}
0 & 1  \\
-1 & 0
\end{array}
\right)
=
\left( \begin{array}{cc}
1 & 2  \\
0 & 1
\end{array}
\right).$$
Thus the images $\{ \phi(v),\phi(w) \}$ generate a well-known free group of rank $2$, sometimes called the Sanov subgroup, which is of index $24$ in $\GL_2(\Z)$ 
(see e.g. \cite[Example 2]{SANOVINDEX}). 
Therefore, $\{ v,w \}$ generates a 
free group of rank $2$ in $\mc{U}\left( \Z^{\alpha}[ C_2 \times C_2 ]  \right)$, 
say of index $r$. Denote this subgroup by $H$ and denote by $T\left( \wt{D} \right)=\pm
\{u_1,u_g,u_h,u_{gh}\}$ the trivial units of $\tilde{D}$. It is
easy to verify that $T\left( \wt{D} \right)\cong D_8$, the
dihedral group of order $8$, and moreover for any $x\neq y\in
\phi^{-1}(T\left( \tilde{D} \right))$, the cosets $xH\neq yH$.
Consequently, $[U\left( \mathbb{Z}^{\alpha}[C_2 \times C_2] \right):H]\geq 8$. As can be seen with the matrices $\phi(u_h)$ and $1 + E_{12}$, the subgroup $\phi\left( \mc{U}( \mathbb{Z}^{\alpha}[C_2 \times C_2] )\right) = \mc{U}(\wt{D})$ is not normal in $GL_2(\mathbb{Z})$ and therefore its index is at
least $3$. Since
$$24=[GL_2(\mathbb{Z}):H]=[GL_2(\mathbb{Z}):\mc{U}(\wt{D})] . [\phi\big(\mc{U}\left( \mathbb{Z}^{\alpha}[C_2 \times C_2] \right) \big):H]\geq 3\cdot8,$$
we conclude that $[\mc{U}\left( \mathbb{Z}^{\alpha}[C_2 \times C_2] \right):H]=8$. A
simple calculation shows that
$$u_h^{-1}vu_h=w,\quad u_g^{-1}vu_g=v^{-1},\quad u_{gh}^{-1}vu_{gh}=w^{-1} \, \text{ and } \, u_g^{-1} w u_g = w^{-1}.$$
Therefore $H$ is normal in $\mathcal{U}\left( \mathbb{Z}^{\alpha}[C_2 \times C_2] \right)$,
and consequently
$$\mathcal{U}\left( \mathbb{Z}^{\alpha}[C_2 \times C_2]\right)=H\rtimes \phi^{-1}(T\left( \tilde{D} \right))\cong F_2\rtimes D_8.$$

The second part of the statement can be checked explicitly. 
To start, note that $\mc{G}_0^{\alpha} = \{ u_h, u_g, u_1 \}, C_g^{-} = \{ u_h, u_{gh}\}$ and $C_h^{-} =\{ u_g, u_{gh}\}$. 
So in total one obtains $8$ non-trivial $H$-units that generate $\mc{H}_{\alpha}(C_2 \times C_2)$. 
However they all can  be expressed in terms of $v,w,u_{gh}$. Indeed $1 + u_{gh} \wt{u_g} = 1 - u_h u_g \wt{u_g} = w^{-1}$ 
and $1 + u_g \wt{u_h} = 1 + u_g + u_{g}u_{h} = u_{gh} v$. The others are similar and, in particular, this shows that 
$\mc{H}_{\alpha}(C_2 \times C_2) = \langle v,w, u_{gh} \rangle$.

Finally, recall that $u_{gh}^2=-1$ and thus $[D_8:\langle u_{gh} \rangle] =2.$ In particular, from the description of $\mathcal{U}(\Z^{\alpha}[ C_2 \times C_2 ])$ obtained earlier, we now see that $\mc{H}_{\alpha}(C_2 \times C_2)$ is a subgroup of index $2$ in $\mathcal{U}(\Z^{\alpha}[ C_2 \times C_2 ])$.
Besides, because $H$-units are unipotent, $\mc{H}_{\alpha}(C_2 \times C_2)$ is a subgroup of $\SL_1(\Z^{\alpha}[C_2 \times C_2])$. 
Since the latter does not contain $u_h$, as seen from the matrix representation, it is a proper subgroup of $\mathcal{U}(\Z^{\alpha}[ C_2 \times C_2 ])$ and hence indeed is equal to $\mc{H}_{\alpha}(C_2 \times C_2)$. That $\phi(\mathcal{U}(\Z^{\alpha}[ C_2 \times C_2 ]))$ has index $3$ in $\GL_2(\Z)$ follows from the fact that $\langle \phi(v), \phi(w) \rangle$ has index $24$ in $\GL_2(\Z)$ and $\phi(\mathcal{U}(\Z^{\alpha}[ C_2 \times C_2 ])) = \langle \phi(v), \phi(w) \rangle \rtimes D_8.$
\end{proof}

The above proof shows the importance of having understanding on subgroup of small index in $\SL_2(\Z).$

\subsection{The general case}

Now we consider the general case of $G = \langle g,h \rangle \times C_2^n$. More precisely, we will follow the method outlined in \Cref{subsection reduing elem ab 2-subgroups} and therefore describe the isomorphism type of the groups $U_1, \ldots, U_n$. Combined with \Cref{th reducing off an elemantary abelian subgroup} this would reduce a full description of $\mc{U}(\Z^{\alpha} [\langle g,h \rangle \times C_2^n])$ to describing the actions in the semi-direct products. {\it We use freely the objects and notations introduced in \Cref{subsection reduing elem ab 2-subgroups}.}

\begin{proposition}\label{the groups U_i for elemen ab}
The groups $U_i$ satisfy the following:
\begin{itemize}
    \item $[\U( \Z^{\alpha}[\langle g,h \rangle] ): U_1] =8$,
    \item $[U_i : U_{i+1}] = 8$ for all $i \geq 1$,
    \item $U_1 \cong F_3 \times C_2$,
    \item If $i\geq 2$, then $U_i$ is a free group of rank $n_i = 1 + 8^{i-1}$.
\end{itemize}
\end{proposition}
\begin{proof}
Recall that, by definition, $U_i = \{ 1+2^{i} u \mid u \in \Z^{\alpha}[\langle g,h \rangle] \} \cap \mc{U}(\Z^{\alpha}[\langle g,h \rangle])$. In particular, elements of $U_1$ are of the form $1+2u$ for $u \in \Z^{\alpha}[\langle g,h \rangle]$ and hence the only trivial units of $\Z^{\alpha}[\langle g,h \rangle]$ contained in $U_1$
are $\pm 1$. In fact, using the notations of the proof of \Cref{prop description unit twisted group ring}, one has that $A=\langle v^2, w^2, wv^{-1},-1\rangle \subseteq  U_1$. Indeed, recall that $v=1+u_h-u_{gh}$ and $w=1+u_h+u_{gh}$, hence a simple calculation yields that
$$v^2=1+2u_h-2u_{gh}=1+2(u_h-u_{gh}), \quad w^2=1+2u_h+2u_{gh}=1+2(u_h+u_{gh}).$$
Additionally, $v^{-1}=1-u_h+u_{gh}$ and hence $$wv^{-1}=-1-2u_g+2u_{gh}=1+2(-1-u_g+u_{gh}).$$
We conclude that indeed $A \subseteq U_1$. However $v$ cannot be written in the form $1+2u$ (which already can be seen $\text{mod } 2$), thus $A \leq U_1 \lneq \langle v,w \rangle \times \langle -1 \rangle$. Note that $[\langle v,w \rangle : \langle v^2,w^2, vw^{-1}\rangle ]= 2$ as the product of any two generators of $\langle v, w \rangle$ is in $\langle v^2,w^2, vw^{-1}\rangle$. Thus $A = U_1$. Now recall the Nielsen-Schreier formula that says that a subgroup of index $t$ in a free group of rank $e$ is a free group of rank $1 + t(e-1)$. Consequently, since $\langle v, w \rangle \cong F_2$ by \Cref{prop description unit twisted group ring}, we obtain all together that 
\begin{equation}\label{generators U_1}
U_1=\langle v^2, w^2,wv^{-1}\rangle \times \langle -1 \rangle\cong F_3\times C_2.
\end{equation}
Now comparing it to the generators of $\U (\Z^{\alpha}[\langle g,h \rangle])$ we see that $[\U( \Z^{\alpha}[\langle g,h \rangle] ): U_1] =8$.

Next, for the statements concerning the $U_i$, we use \Cref{Prop ring iso type twisted group ring}
and that $\phi(\Z^{\alpha}[\langle g,h \rangle])$ is an order contained in the maximal order $\Ma_2(\Z)$ and thus, by a well-known fact,  that an element is invertible in the former if and only if it is in the latter. Hence, we obtain that
\begin{equation}\label{Ui as matrix}
U_i \cong \phi(U_i) = \{ \left( \begin{array}{cc}
1 + 2^{i}a_{11} & 2^{i}a_{12}  \\
2^{i}a_{21} & 1+2^{i}a_{22}
\end{array}\right) \mid a_{11} \equiv_2 a_{22}, \, a_{12}\equiv_2 a_{21} \} \cap \GL_2(\Z).
\end{equation} 
Further remark that $U_1 \leq \SL_1(\Z^{\alpha}[C_2 \times C_2])$ and hence $\phi(U_{i}) \leq \phi(U_1) \leq \SL_2(\Z).$

Consider now the principal congruence subgroup $\Gamma(2^{i})$ of level $2^{i}$ in $\SL_2(\Z)$:
$$\Gamma(2^{i}) = \ker (\pi_{2^i}): \SL_2(\Z) \rightarrow \SL_2(\Z/2^{i}\Z).$$
More concretely, $\Gamma(2^{i}) = \big( 1 + 2^{i}\Ma_2(\Z) \big) \cap \
\SL_2(\Z)$. Note that $\Gamma(2^{i+1}) \leq \phi(U_i) \leq \Gamma(2^{i})$. The second part of the statement will be a consequence of the following:
$$\textit{{\it Claim: }} [\Gamma(2^{i}) : \phi(U_i)] = 2 \text{ and } [\phi(U_i) : \Gamma(2^{i+1})] = 4 \text{ for all } i\geq 1$$
To start, recall  the following well-known formula (e.g. see\footnote{In this reference the formula for $[\SL_2(\Z) : \ker\left(\SL_2(\Z) \rightarrow \PSL_2(\Z/n\Z)\right)]$ is given. If $n>2$ this differs by a factor $2^{-1}$ with difference coming from the central matrix $-1$.} \cite[p 146]{NewmanBook}):
\begin{equation}\label{formula index congruence}
[\SL_2(\Z) : \ker\left(\SL_2(\Z) \rightarrow \SL_2(\Z/n\Z)\right)] = n^3 \prod\limits_{p \, \mid \, n} (1 - \frac{1}{p^2}),
\end{equation}
where the product runs over all the prime divisors $p$ of $n$. 
%
%
%
%
%
Consequently, $[\Gamma (2^{i}) : \Gamma (2^{i+1})] = \frac{[\SL_2(\Z) : \Gamma (2^{i+1})]}{[\SL_2(\Z) : \Gamma(2^{i})]} = 8$ for all $ i \geq 1$. 
Further note that $\left( \begin{array}{cc} 1 & 2^{i} \\ 0 & 1 \end{array}\right) \in \Gamma(2^{i}) \setminus \phi(U_i)$ for all $i$. Therefore, if we prove that $A.B \in \phi(U_i)$ 
for any pair of matrices 
$A = \left(\begin{array}{cc} 1+2^{i} a_{11}& 2^{i}a_{12} \\ 2^{i}a_{21} & 1+2^{i}a_{22} \end{array}\right),B=\left(\begin{array}{cc} 1+2^{i}b_{11}& 2^{i}b_{12} \\ 2^{i}b_{21} & 1+2^{i}b_{22} \end{array}\right) \in \Gamma(2^{i}) \setminus \phi(U_i)$, 
then $[\Gamma(2^{i}) : \phi(U_i)] = 2$, which is the first part of the claim\footnote{A concrete matrix in $\phi(U_i)$ but not in 
$\Gamma(2^{i+1})$ is $\left(\begin{array}{cc} 1 + 2^{i} & 2^{i} \\ -2^{i} & 1 - 2^{i} \end{array} \right)$.}. Subsequently, the second part follows from this and our value of the index $[\Gamma(2^{i}): \Gamma(2^{i+1})].$ A direct computation shows that $AB \in \phi(U_i)$ if and only if $a_{11} + a_{22}\equiv_2 b_{11}+ b_{22}$ 
and $a_{12}+ a_{21} \equiv_2 b_{12} + b_{21}$. However it also is easy  to prove that $A$ having determinant $1$ already 
implies that $a_{11} \equiv_2 a_{22}$ (analogously for $B$). Hence if $A \in \Gamma(2^{i}) \setminus \phi(U_i)$ then 
$a_{12} +a_{21}$ is odd and the same for $b_{12} + b_{21}$. All together the conditions for $AB \in \phi(U_i)$ are always 
satisfied, as needed. The second part of the statement now follows from the claim as follows:
$$[U_i : U_{i+1}] = [\phi(U_i) : \phi(U_{i+1})] = [\phi(U_i) : \Gamma (2^{i+1})] \, . \, [\Gamma (2^{i+1}) : \phi(U_{i+1})] =  8.$$ 

Finally we prove the last part of the result. By the above, $U_2$ is a torsion-free subgroup of 
index $8$ in $U_1 =\langle v^2, w^2,wv^{-1}\rangle \times \langle -1\rangle$. 
Note that an element $-1 . (1 + 2 u)$ with $u \in \Z^{\alpha}[\langle g,h \rangle]$ can not 
be of the form $1 + 4 v$ with $v \in \Z^{\alpha}[\langle g,h \rangle]$, 
as seen by working modulo $4$. Thus $U_2$ is a subgroup of index 
$4$ in $F_3 = \langle v^2, w^2,wv^{-1}\rangle.$ 
Thus the Nielsen-Schreier formula yields that $U_2 \cong F_{1+4\cdot2}$. Since $U_{i}\leq U_{i-1}$ 
also $U_i$ is a free group whose rank is computed with a recursive use of 
Nielsen-Schreier's formula. 
\end{proof}
\begin{remark}\label{Gbic not equal to twisted bic}
Consider $\mathcal{F} = \{ \widehat{u_g} = \frac{1}{o(g)}\wt{u_g} \mid g \in \mathcal{G}_0^{\alpha} \}$. The computations in the previous proof show
that $v^2, w^2, wv^{-1}$ are in the group $\text{GBic}^{\mc{F}}(\Q^{\alpha}[\langle g,h \rangle],\Z)$ generated by the (generalized) bicyclic units (see \Cref{Def gen bicyclic}). Actually these two groups are equal (as can for example be checked by expressing the generating bicyclic units in terms of $v^2, w^2, wv^{-1}$). 
Hence $\U(\Z^{\alpha}[C_2 \times C_2])$ is an 
example where $$\text{GBic}^{\mc{F}}(\Q^{\alpha}[G],\Z) \lneq\mc{H}_{\alpha}(G).$$ That the difference is still of finite index  heavily relies on 
the fact that the generators of $\mc{H}_{\alpha}(G)$ are all formed from elements of order $2$ (which allowed $wv^{-1}$ to be a generalized bicyclic unit). In larger examples however the $H$-units can be an infinite index overgroup, as will be shown in \Cref{subsection H-units infnite index ext}.
\end{remark}

\section{Description of \texorpdfstring{$\mc{U}(\Z [D_8 \times C_2^n])$}{}}\label{section desciprtion D8 times elem abelian 2-group}

Let $G$ be the group $C_2^{n+2}= \langle g,h, x_1, \ldots, x_n \rangle$ as in~\eqref{eq:genG} and  
$$
\Gamma\cong D_8 \times C_2\times \ldots \times C_2=\langle a,b \rangle  \times \langle y_1 \rangle \times\ldots 
\times \langle y_n \rangle.
$$
where $D_8 = \langle a,b \mid a^4=1= b^2, a^b = a^{-1} \rangle$. We consider the 
epimorphism $\lambda: D_8 \times C_2^n \rightarrow G$ determined by 
$\lambda(a) = gh, \lambda(b)  = g, \lambda(y_i) = x_i$. With the associated canonical 
choice of section $\mu: G \rightarrow D_8 \times C_2^n$, one directly checks that 
the associated $2$-cocycle $\alpha(s,t) = \mu(s)\mu(t)\mu(st)^{-1}$ is exactly the one determined by~\eqref{eq:alpharelations}. In other words, the central extension $\Gamma$ associated with the $[\alpha]\in H^2(G,C_2)$ from \Cref{section elem ab description} is isomorphic to $D_8 \times C_2^n$. 

Taking now $\chi \in \Hom (\langle a^2 \rangle, \Z^{*})$ with $\chi(a^2)=-1$ we see that $\Psi_{\chi}: \Z[\Gamma] \rightarrow \Z^{\Tra(\chi)}[G]$ is nothing else than the ring epimorphism. 
\begin{equation}\label{concrete map in concrete D8 case}
\psi :\Z [D_8 \times C_2^n] \rightarrow \Z ^{\alpha}[G]
\end{equation}
defined by 
$$a\mapsto u_{gh},\quad b\mapsto u_g, \quad x_i\mapsto u_{y_i}.$$

In this section we will pullback the description of $\mc{U}(\Z^{\alpha}[G])$, obtained in section \Cref{section elem ab description}, by using the methods from \Cref{subsection corrolations} in order to obtain a description of $\mc{U}(\Z [D_8 \times C_2^n])$. First we invoke \Cref{size and comparission of ker} and \Cref{size and comparission of coker}  to give a precise description of the (co)kernel of the (co)restricted morphism $\wt{\psi}_{res} : \mc{U} (\Z[D_8]) \rightarrow \mc{U}(\Z^{\Res(\alpha)}[\langle g,h \rangle]): a\mapsto u_{gh}, b\mapsto u_g$.

\begin{lemma}\label{kernel and cokernel for D_8 times elem ab}
With  notations as above we have that
\begin{itemize}
    \item[(i)] $\ker(\wt{\psi}) = \langle -a^2 \rangle \cong C_2,$
    \item[(ii)] $|\Coker(\wt{\psi}_{res})| = 2.$
\end{itemize}
\end{lemma}
\begin{proof}
The description of the kernel follows from \Cref{size and comparission of ker} which says that the kernel is finite and $\ker(\wt{\psi}) = \{1, - a^2 \}$, as desired. 

In order to describe $\Coker(\wt{\psi}_{res})$, recall that by \Cref{prop description unit twisted group ring}
$$\mathcal{U}(\Z^{\Res(\alpha)} [C_2 \times C_2]) = \langle v,w \rangle \rtimes D_8$$
with $v = u_1 + u_h\wt{u_g}$ and $w = u_1 + \wt{u_g}u_h.$ As was claimed (without providing details)
in \Cref{Remark h-unit and opp in coker}, in this case the product of any two generators of 
$\mc{U}(\Z^{\Res(\alpha)}[\langle g,h \rangle])$ belongs to $\Ima(\wt{\psi}_{res})$. 
Furthermore \Cref{non-trivial elements in cokernel} asserts that 
$v$ is a non-trivial element in $\Coker(\wt{\psi}_{res})$, 
hence $\wt{\psi}_{res}$ is not surjective and thus $|\Coker(\wt{\psi}_{res})|= 2$. 

To prove the above stated claim on the product of generators, note first that the trivial units are attained. So we need to consider now the torsion-free part. For this consider the bicyclic units $b_1 = 1 + (1 - b) ab (1+b)$ and $b_2 = 1 + (1+b) ab (1-b)$. Recalling (\ref{rewriting a twisted bicyclic}) we see that $\psi(b_1) = v^2$ and $\psi(b_2)= w^2$. A direct computation also gives that
$$wv^{-1} = -u_1-2u_g-2u_{h}u_{g} = -1 (u_1 + (1+u_h)u_g(1-u_h))$$
Therefore if we consider the bicyclic unit $b_3 = 1 - (1-ab) b (1+ab)$, we also notice that $\psi(-b_3) = wv^{-1}$. Thus the claim follows and it finishes the proof.
\end{proof}

\Cref{kernel and cokernel for D_8 times elem ab} combined with \Cref{the groups U_i for elemen ab} and \Cref{th reducing off an elemantary abelian subgroup} yield  the description we were looking for.

\begin{proposition}\label{full decomp for D8 times elem ab}
The following hold:
\begin{enumerate}
    \item $\mc{U}(\Z[D_8]) \cong F_3 \rtimes (\pm D_8)$ with $F_3$ generated by bicyclic units,
    \item $U_1 \cong F_{9} \times C_2$,
    \item $U_i \cong F_{n_i}$ with $n_i = 1 + 8^{i}$ for $i \geq 2$, 
    \item $\pm D_8 \times C_2^n$ has a torsion-free normal complement in $\mc{U}(\Z[D_8 \times C_2^n])$.
\end{enumerate}
\end{proposition}
This result for example yields that $\U (\Z[D_8 \times C_2]) \cong \left( F_{9} \rtimes F_3 \right) \rtimes (\pm D_8 \times C_2).$  Such a decomposition was already obtained in \cite[Theorem 5]{JesParment}. 
\begin{proof}
To start note that $\psi(\pm D_8) = \langle g,h \rangle \cong D_8$. Also it is not hard to see 
that the elements $b_1,b_2,b_3$ (used in the proof of 
\Cref{kernel and cokernel for D_8 times elem ab}) are in fact generators for 
$\langle \text{Bic}(D_8) \rangle$, the group generated by the bicyclic units of $D_8$ 
(e.g. see \cite[part (3) of proof Example 1.5.4 ]{EricAngel1}). Hence we already know that 
$$\langle \text{Bic}(D_8), a^2 \rangle \cong \langle \text{Bic}(D_8) \rangle \times 
\langle a^2 \rangle  \cong \pm 1 .  \langle v^2,w^2,wv^{-1} \rangle = 
U_1(\mathbb{Z}^{\alpha}[C_2\times C_2])$$
with $U_i(\mathbb{Z}^{\alpha}[C_2\times C_2]):=  \{ 1 + 2^{i} u \mid u \in \Z^{\alpha}[\langle g,h\rangle] \} \cap \mc{U}(\Z^{\alpha}[\langle g,h\rangle])$ 
and where the second isomorphism is given by $\wt{\psi}_{res}$. 
The latter is isomorphic to $F_3 \times C_2$ by \Cref{the groups U_i for elemen ab}. 

Now, via \Cref{the groups U_i for elemen ab} and \Cref{kernel and cokernel for D_8 times elem ab} 
one infers that $\ker(\wt{\psi}) \subset \langle \text{Bic}(D_8), \pm D_8 \rangle$ 
and that $\wt{\psi}( \langle \text{Bic}(D_8), \pm D_8 \rangle)$ is of index two in 
$\U(\Z^{\alpha}[C_2\times C_2])$, thus equals $\im(\wt{\psi})$. Consequently, the first part 
follows and we have given a new proof of the well known fact that 
$\U(\Z[D_8]) = \langle \text{Bic}(D_8) \rangle \rtimes (\pm D_8)$. 
In particular all elements in $\langle \text{Bic}(D_8) \rangle$ are of the form 
$1 + (1-a^2) u$ with $u \in \Z[D_8]$. 
This in particular implies that for all $i$
\begin{equation}\label{U_i for D_8}
U_i(\Z D_8) = \{1+ 2^{i}(1-a^2)u \mid u \in \Z[D_8]\}\cap \mc{U}(\Z[D_8])
\end{equation}

Next, due to the above, we see that for each $i$ the torsion-free elements of $U_i(\Z D_8)$ are in 
$\langle \text{Bic}(D_8) \rangle$.  In particular, for each $i \geq 2$ this means that 
$U_i(\Z D_8) \leq \langle \text{Bic}(D_8) \rangle$ and $\wt{\psi}$ is injective on 
them by \Cref{kernel and cokernel for D_8 times elem ab}. Using the description (\ref{U_i for D_8}) we moreover have that $\wt{\psi}(U_i(\Z D_8)) = U_{i+1}(\Z^{\alpha}[C_2\times C_2]).$ 
In summary, $\wt{\psi}$ induces an isomorphism
\begin{equation}\label{change of the Ui}
U_i(\Z D_8) \cong U_{i+1}(\Z^{\alpha}[C_2\times C_2]).
\end{equation}
for all $i \geq 2$. If $i=1$ then by \Cref{th reducing off an elemantary abelian subgroup} the group $U_1(\Z D_8) \cong \{ \pm 1\} \times N_1 $ with $N_1$ torsion-free. By the above $N_1 \leq \langle \text{Bic}(D_8) \rangle$ and concretely  $N_1 = \{1 + 2 (1-a^2) u \mid u \in \Z[D_8] \}$.
So via $\wt{\psi}$ we have the isomorphism $N_1 \cong U_{2}(\Z^{\alpha}[C_2\times C_2])$. Therefore the second and third statement now follows from \Cref{the groups U_i for elemen ab}. The last statement follows from the first and \Cref{Properties that get inherited with times elem ab}.
\end{proof}

As mentioned after \Cref{remark that preserved with direct prod with elem ab}, it is known that the group generated by the bicyclic units form a normal complement in $\U (\Z[D_8 \times C_2^n])$ if $n \leq 1$, but for $n = 2$ they do not  form a normal complement as shown in \cite{Li}. In fact, it is expected that if $n \geq 2$ the bicyclic units are even of infinite index.  It would be especially instructive to describe a minimal set of generators (in a generic way) of the torsion-free normal complement mentioned \Cref{full decomp for D8 times elem ab}. 

To finish this section, we would like to record the exact size and generators of $\Coker(\wt{\psi})$. 

\begin{lemma}\label{full cokernel for D8 times elem ab}
With notations as above we have that
\begin{enumerate}
    \item $|\Coker(\wt{\psi})| = [ \mc{U}(\Z^{\alpha}[G]): \Ima(\wt{\psi})] = 2^{n+3}-4n - 6,$
    \item $\Coker(\wt{\psi})$ is generated by $H$-units.
\end{enumerate}
\end{lemma}
\begin{proof}
By \Cref{th reducing off an elemantary abelian subgroup}, each of  the unit groups 
$\U(\Z[\Gamma])$ and $\U(\Z^{\alpha}[G])$ is determined by respective groups $K_{\vec{j}}.$ 
To distinguish, we respectively write $K_{\vec{j}}^{(i)}(\Gamma)$ and $K_{\vec{j}}^{(i)}(G)$ 
for $\vec{j} \in \{ 0,1 \}^{i-1}.$ The methods in \Cref{subsection reduing elem ab 2-subgroups} 
imply that
$$|\Coker(\wt{\psi})| = |\Coker(\wt{\psi}_{res})| + \sum_{i=1}^{n} \sum_{\vec{j} \in \{ 0,1 \}^{i-1}}[K_{\vec{j}}^{(i)}(G): \wt{\psi}(K_{\vec{j}}^{(i)}(\Gamma))].$$
Following \Cref{kernel and cokernel for D_8 times elem ab}, $|\Coker(\wt{\psi}_{res})|=2$. 
Next, to compute $[K_{\vec{j}}^{(i)}(G): \wt{\psi}(K_{\vec{j}}^{(i)}(\Gamma))]$ we use (see \Cref{th reducing off an elemantary abelian subgroup}) the 
isomorphism $K_{\vec{j}}^{(i)}(G) \cong U_{1 + \sum_{t=1}^{i-1} j(t)}$  which we will denote 
by $\theta$ (where we abuse notations as it in fact depends on $i$ and $\vec{j}$). 
For notational 
simplicity denote $\sum \vec{j} :=  \sum_{t=1}^{i-1} j(t)$. Note that $\theta(\wt{\psi}
(K_{\vec{j}}^{(i)}(\Gamma))) \cong \wt{\psi}_{res}(\theta(K_{\vec{j}}^{(i)}(\Gamma) )).$ Therefore, if 
$\vec{j} \neq \vec{0}$, using (\ref{change of the Ui}) and \Cref{the groups U_i for elemen ab} 
yields
$$[K_{\vec{j}}^{(i)}(G): \wt{\psi}(K_{\vec{j}}^{(i)}(\Gamma))] = [\theta\left( K_{\vec{j}}^{(i)}(G) \right): \theta\left( \wt{\psi}(K_{\vec{j}}^{(i)}(\Gamma)) \right)] = [U_{1 + \sum \vec{j}} \, : \, U_{2 + \sum \vec{j}}]= 8.$$
If $\vec{j} = \vec{0}$, then both $K_{\vec{j}}^{(i)}(G)$ and $K_{\vec{j}}^{(i)}(\Gamma)$ 
contain a copy of $C_2$ (namely $\langle u_{y_i} \rangle$, resp. $\langle x_i\rangle$) and
$\wt{\psi}$ induces an isomorphism on these. Thus in this the same 
passage through the above isomorphisms yields that $[K_{\vec{j}}^{(i)}(G): \wt{\psi}(K_{\vec{j}}^{(i)}(\Gamma))] = [\langle v,w, vw^{-1} \rangle : U_2] = 4$. So all together
$$|\Coker(\wt{\psi})| = 2 + 4n + 8 \sum_{i=1}^{n} (2^{i-1}-1) = 2 + 4n - 8n + 8(2^n-1) = 2^{n+3} - 4n -6.$$
For the second part of the statement, using the elements in $\mc{H}_{\alpha}(C_2\times C_2)$ 
we will give explicit generators of the quotient $K_{\vec{j}}^{(i)}(G) / \wt{\psi}(K_{\vec{j}}^{(i)}(\Gamma))$. First recall that by \Cref{the groups U_i for elemen ab} $U_{1 + \sum \vec{j}} / U_{2 + \sum \vec{j}}$ is of order $8$ if $\vec{j} \neq \vec{0}$ and $4$ else. Now consider the set
\begin{equation}\label{generators of the factors}
\{ v^{2^{1+ \Sigma \vec{j}}}, (vw^{-1})^{2^{\Sigma \vec{j}}}, v^{2^{\Sigma \vec{j}}}w^{-2^{\Sigma \vec{j}}} \}.
\end{equation}
Note that the second and third generators are equal if $\vec{j} = \vec{0}$. Computing their 
image under the isomorphism $\phi$ from the proof of \Cref{the groups U_i for elemen ab}, and 
also of all their double products, one deduces that they generate an elementary abelian $2$-group 
of order $8$ if $\vec{j} \neq \vec{0}$ and of order $4$ else. In particular, for all $\vec{j}$ 
their images generate $\phi(U_{1 + \sum \vec{j}}) / \phi(U_{2 + \sum \vec{j}})$ and so the set 
itself generates $U_{1 + \sum \vec{j}} / U_{2 + \sum \vec{j}}$. It now remains to construct 
three $H$-units in $K_{\vec{j}}^{(i)}(G)$ which map under $\theta$ to the elements in (\ref{generators of the factors}). 
For notation ease we denote $(1-u_x)^{\vec{j}} = (1-u_{x_{1}})^{j(1)}\ldots (1-u_{x_{i-1}})^{j(i-1)}$. We claim that the desired units are:
$$\begin{array}{l}
z^{(i)}_{\vec{j},1} = 1 + (1-u_{x_{i}})(1-u_x)^{\vec{j}} \, u_h \wt{u_g},\\
\\
z^{(i)}_{\vec{j},2} = 1 - (1-u_{x_{i}})(1-u_x)^{\vec{j}} \, \wt{u_h} u_g,\\
\\
z^{(i)}_{\vec{j},3} = \left( 1 + (1-u_x)^{\vec{j}}\,  u_h \wt{u_g} \right) \left( 1 - (1-u_x)^{\vec{j}} \, u_h \wt{(u_{x_i}u_{g})}\right).
\end{array}$$
Writing out the third element results in 
$$z^{(i)}_{\vec{j},3} = 1 + (1-u_x)^{\vec{j}} u_h u_g(1-u_{x_i}) - 2^{\Sigma \vec{j}} (1-u_x)^{\vec{j}} (1-u_g)(1-u_{x_i}).$$
Thus we see that $\{ z^{(i)}_{\vec{j},1}, z^{(i)}_{\vec{j},2}, z^{(i)}_{\vec{j},3} \} \subset K_{\vec{j}}^{(i)}(G)$ and clearly 
$$\theta (z^{(i)}_{\vec{j},1})= v^{2^{1+ \Sigma \vec{j}}}, \, \,\theta(z^{(i)}_{\vec{j},2})=(vw^{-1})^{2^{\Sigma \vec{j}}},\, \,\theta(z^{(i)}_{\vec{j},3}) = v^{2^{\Sigma \vec{j}}}w^{-2^{\Sigma \vec{j}}}.$$
This finishes the proof of the second statement.
\end{proof}

\section{A new generic construction of units in integral group rings}\label{section h-units}

In this section we introduce a new generic construction of units in $\Z G$ which in fact are elements of 
$\SL_1(\Z G)$ (see (\ref{def of SL1})). These elements originated as the pullback of the products of the 
$H$-units in twisted group rings from \Cref{algemene def twisted bicyclic} along the transgression morphisms 
of \Cref{background transgression}. However, they can also be defined directly as will be shown in \Cref{General H-units in ZG}. Furthermore, as explained in \Cref{subsection H-units in ss alg}, those units of \Cref{algemene def twisted bicyclic} 
can be constructed in general finite dimensional semisimple $F$-algebras, with $F$ a number field. Finally, 
in \Cref{subsection H-units infnite index ext}, we will give an infinite family of groups where the newly 
constructed units contain the bicyclic units as subgroup of infinite index. In particular, these elements 
are indeed a new step towards the problem of describing generators of $\U(\Z G)$, up to commensurability, 
generically in $\Z G$. Interestingly, these units will also yield the first generic construction of free 
groups of rank larger than $2$, see \Cref{prop H-unit is H-unit}.

\noindent {\it Notation.} Recall that $\tilde{H} = \sum\limits_{h \in H} h$ for any finite subgroup $H$ in an algebra $A$ and $\widehat{H} = \frac{1}{|H|}\tilde{H}$. If $A$ is semisimple we will denote by $\PCI(A)$ the set of primitive central idempotents of $A$ 

\subsection{Restricted construction of units for orders in general semisimple algebras}\addtocontents{toc}{\protect\setcounter{tocdepth}{2}}\label{subsection H-units in ss alg}

Let $A$ be a finite dimensional semisimple $F$-algebra, with $F$ a number field, and let $\O$ be a $\Z$-order 
in $A$. Inspired by \Cref{algemene def twisted bicyclic} and \Cref{remarks on h-unit for twisted grp group} 
we define the following.

\begin{definition}\label{primitive h-units}
Let $x,t \in \U(\O)$ be torsion units such that $[t,x] \in N_{\U(\O)}(\langle t \rangle)$ is of finite order and $\wt{[t,x]} = 0$. Then the elements 
$$u_{x^{\pm 1}, \wt{t}} := 1 + x^{\pm 1}\wt{t} \mbox{ and } u_{\wt{t}, x^{\pm 1}} := 1 + \wt{t}x^{\pm 1}$$
will be called {\it primitive $H$-units.}
\end{definition}

As the name suggests, the elements in \Cref{primitive h-units} are indeed invertible elements in $\U(\O )$, 
as proven below. The reason for adding, compared to \Cref{algemene def twisted bicyclic}, primitive in the name 
will be clarified later on in \Cref{remark on the name}.

\begin{theorem}\label{H-units are units}
Let $x,t\in \U(\O)$ be as in \cref{primitive h-units}. Then $u_{x^{\pm 1}, \wt{t}}$ and $u_{\wt{t}, x^{\pm 1}}$ are unipotent units in $\U(\O )$. In particular they are of infinite order.
\end{theorem}

We will prove that $\left(x^{\pm}\wt{t}\right)^2 = 0 $, hence the inverse of $u_{x^{\pm 1}, \wt{t}}$ is $1 - x^{\pm 1}\wt{t}$. Similarly for $u_{\wt{t}, x^{\pm 1}}$.

\begin{example}
The archetypical example is a tuple $(x,t)$ of torsion units such that $[t,x] \in F^*$ is a root of unity. In that case $[t,x]$ is both central and $\wt{[t,x]} = 0$.  If $A = R^{\alpha}[G]$ is some twisted group ring and $x,t \in G$, then the linear independence of the basis elements yields that $u_{x^{\pm 1}, \wt{t}}$ and $u_{\wt{t}, x^{\pm 1}}$ are non-trivial (i.e. $\neq 1$) if $\tilde{x} \neq 0$ in $A$. However, in general, there is no transparent characterisation of being non-trivial. 
\end{example}

In order to prove  \Cref{H-units are units} we need the  following lemma which was shared to us by \'Angel del R\'io and whose proof depends on useful identities from \cite{BrGDR}.

\begin{lemma}\label{label about dividing order}
Let $G = \langle g,a \rangle$ be a finite meta-cyclic group with $\langle g\rangle $ a normal subgroup such that $o(g) = o(ga)$. Then $o(a) \mid o(g).$
\end{lemma}
\begin{proof}
Denote $m=o(g)$, $n=[G: \langle g\rangle]=\frac{|G|}{m}$ and $s=[G : \langle a \rangle]=[\langle g \rangle : \langle g\rangle \cap \langle a\rangle]$.
Then $s$ divides $m$ and $o(a)=\frac{|G|}{s}=\frac{mn}{s}$.

For a prime $p$ denote by $v_p(k)$ the $p$-adic valuation of $k\in \mathbb{Z}_{\geq 0}$. Suppose that $o(a)=\frac{nm}{s}$ does not divide $m=o(g)$, or equivalently $n$ does not divide $s$, i.e. $v_p(n)>v_p(s)$ for some prime $p$.


Let $i$ be the smallest non-negative integer such that $a^{-1}ga=g^i$ and $1\leq t\leq m$ such that $a^n = g^t$. 
Note that $g = a^{-n} g a^n = g^{i^n}$ and so $i^n \equiv 1 \mod m$. Next, by using formula \cite[eq. (2.2)]{BrGDR} 
(notice that in loc. cit. the role of $a$ and $g$ are interchanged)  we have $(ag)^n=a^ng^{S(i|n)}=g^{t+S(i|n)}$ 
with $S(i|n) := \sum_{j=0}^{n-1}i^{j}$. Then, by the assumption,  $m=o(g)=o(g^{-1}(ga)g)= o(ag)=n \frac{m}{gcd(m,t+S(i|n))}$ and 
hence $n=gcd(m,t+S(i|n))$.  In particular $n$ divides both $t+S(i|n)$ and $m$. 
Also, $\frac{nm}{s} = o(a)=n. o(g^t)=\frac{nm}{\gcd(m,t)}$. Therefore $s=\gcd(m,t)$. All together we deduce that
$$v_p(s)= \min \{ v_p(t), v_p(m) \} < v_p(n) = \min \{ v_p(m), v_p(t + S(i|n)) \}.$$
This entails that $v_p(t) = v_p(s) < v_p(n)$. Consequently also $p\mid m$ and $v_p(t) < v_p(t+S(i|n)).$  In particular $v_p(t) = v_p(t + S(i|n) - t) = v_p(S(i|n))$, hence $v_p(S(i|n)) < v_p(n)$.




Now let $q$ be the multiplicative order of $i$ modulo $p,$ i.e. the smallest integer $q$ with $p\mid i^q-1.$
Since $i^n \equiv 1 \mod m$ and $p$ divides $m$, it follows that $q$ divides $n.$
Moreover, $q$ divides $p-1.$ Hence $v_p(n)=v_p(\frac{n}{q}) \geq 1.$
From \cite[Lemma 8.2]{BrGDR}  it follows that  in case $p$ is odd or $p=2$ and $i^q \equiv 1 \mod 4$, we have $v_p(S(i^q|\frac{n}{q})) = v_p(\frac{n}{q})$. Otherwise $p=2, v_p(i^q+1)\geq 2$ and
$v_2(S(i^q|\frac{n}{q})) = v_2(\frac{n}{q}) + v_2(i^q+1)-1 > v_2(\frac{n}{q}).$
So all together, $v_p(S(i^q|\frac{n}{q})) \geq v_p(\frac{n}{q}).$

Moreover, by \cite[Lemma 8.1]{BrGDR}, $S(i|n)=S(i|q)S(i^q|\frac{n}{q})$.
Thus, $v_p(\frac{n}{q})=v_p(n)>v_p(S(i|n)) = v_p(S(i|q)) + v_p(S(i^q|\frac{n}{q})) \geq v_p(\frac{n}{q}),$ which is clearly a contradiction. The contradiction came from the assumption that $v_p(n)>v_p(s)$ for some prime $p$. Hence we conclude that $v_p(n)\leq v_p(s)$ for all primes $p$, that is, $n$ divides $s$ and hence $o(a) \mid o(g).$
\end{proof}

We can now proceed to the proof of the theorem.

\begin{proof}[Proof of \Cref{H-units are units}.]
Assume  $x,t \in \U(\O)$ are torsion units such that $[t,x] \in N_{\U(\O)}(\langle t \rangle)$ is of finite order and $\wt{[t,x]} = 0$.
For simplicity of notation, denote $[t,x] = a.$ Since $x^{-1}tx=ta$ we have that $o(t)= o(ta)$. Thus $G=\langle a, t \rangle$ satisfies the conditions of \Cref{label about dividing order} and hence $o(a) \mid o(t)$.

Next, recall the notation $S(s\mid n) := \sum_{j=0}^{n-1}s^{j}$ for any numbers $s,n \in \N$. Since $a \in N_{\U(\O)}(\langle t \rangle)$ there is some $i \neq 0$ such that $ta = at^{i}$. In other words, $tx = xa t^{i}$ and $x^{-1}t = a t^{i}x^{-1}.$ Via induction one directly obtains that
\begin{equation}\label{swapping g and h rule}
 x^{-1} t^k = a^k t^{iS(i|k)} x^{-1} \mbox{ and } t^{k}x = x a^{k} t^{iS(i|k)}.
\end{equation}
These formulas to swap $t$ and $x$ enable us to compute $\wt{t}x^{\pm 1} \wt{t}.$ We claim that: 
$$\wt{t}x^{-1}\wt{t} = \frac{o(t)}{o(a)} \wt{a}\, \wt{t}\, x^{-1} \mbox{ and } \wt{t}x\wt{t} = \frac{o(t)}{o(a)} x\, \wt{t}\, \wt{a}.$$ 
In order to prove this claim, denote $\frac{o(t)}{o(a)} = m \in \N$ and rewrite $\wt{t} = \sum_{\ell = 0}^{m-1} t^{\ell \cdot o(a)} \left( \sum_{j=0}^{o(a)-1} t^{j} \right).$ Then, using the rule in (\ref{swapping g and h rule}) and that $a \wt{t} = \wt{t}a$, because $a$ normalizes $\langle t \rangle$, we obtain that 
$$
\begin{array}{rcl}
\wt{t} x \wt{t} & = & x\, \Big( \sum\limits_{\ell = 0}^{m-1} a^{l\, o(a)} t^{iS(i\mid l\,o(a))} \big( \sum\limits_{j=0}^{o(a)-1} a^{j}t^{iS(i\mid j)} \big) \Big)\, \wt{t} \\
& = & x \left(  \sum\limits_{\ell = 0}^{m-1} t^{iS(i\mid l\,o(a))} \big(   \sum\limits_{j=0}^{o(a)-1} a^{j}\wt{t} \, \big) \right) \\
& = & x \, \wt{t} \, (\sum\limits_{\ell = 0}^{m-1} \wt{a}).
\end{array}
$$

In the second and last equality we have used that $\wt{t}t^{j} = \wt{t}$ for every $j$. 
The expression for $\wt{t}x^{-1} \wt{t}$ is computed analogously. Now, since by assumption $\wt{a}=0$, 
we obtain that $\left(x^{\pm 1}\wt{t}\right)^2 = 0 = (\wt{t}x^{\pm 1})^2.$ Consequently, 
$u_{x^{\pm 1},\wt{t}}^n = 1 + n x\wt{t}$ for all $n\in \Z$. In particular it is invertible with 
inverse $1 - x^{\pm 1}\wt{t}$ and has infinite order. Similarly for $u_{\wt{t},x^{\pm 1}}.$
\end{proof}

\subsection{General construction and properties of H-units} \label{General H-units in ZG}

In this section we start with constructing a class of units in $\SL_1(\Z G)$ from any triple $(g,h,Q)$ satisfying 
the properties mentioned in the definition below. A first new feature of these, that we will obtain in \Cref{prop H-unit is H-unit}, 
is that they produce free groups of large rank. Furthermore, as shown later in \Cref{subsection H-units infnite index ext}, 
the group consisting of these  units can contain  the  bicyclic units as subgroup of  infinite index and hence also up to commensurability they are new. 

\begin{definition}\label{def h-units}
Let $G$ be a finite group and $(g,h,Q)$ a triple satisfying 
\begin{itemize}
\item $Q$ is a normal subgroup  in $\langle g,h,Q \rangle$,
\item $[g,h]Q \in \mathcal{Z}(\langle g,h,Q\rangle/Q)$ and $o(hQ) = 2$.
\end{itemize}
Denote $\wt{(\pm g)}_Q := \sum\limits_{i=0}^{o(gQ)-1} (\pm g)^{i}$ and for any tuple $(x_1,x_2,y_1,y_2)\in \Z_{\geq 0}^4$ define the element
\begin{equation}\label{eq H-unit element}
v_{(x_1,x_2,y_1,y_2)} := 1 + \frac{1}{2|Q|}\wt{Q}(1-[g,h])\Big(  h\,  [ x_1\, \wt{(-g)_Q} + x_2\, \wt{g}_Q ] +  y_1 \, \wt{(-g)}_Q + y_2 \,\wt{g}_Q \Big).
\end{equation}
A quadruple $(x_1,x_2,y_1,y_2)\in \Z_{\geq 0}^4$ will be called {\it admissible for $(g,h,Q)$} if
\begin{equation}\label{Admissable quadruple cond}
y_1 + y_2 + o(gQ)\, y_1y_2 = o(gQ)\, x_1 x_2 \mbox{ and } x_2 \pm x_1 \equiv 0 \equiv y_2 \pm y_1 \mod 2 |Q|.
\end{equation}
The elements $v_{(x_1,x_2,y_1,y_2)}$ for admissible $(x_1,x_2,y_1,y_2)$ will be called {\it an H-unit}. \vspace{0,2cm}
\end{definition}

The first condition to be admissible will exactly correspond to being invertible (more precisely,  
to belong to $\SL_1(\Z G)$) and the second condition yields that the element is in $\Z G$.

For a fixed triple $(g,h,Q)$ as in Definition \ref{def h-units}, the set of elements of the form (\ref{eq H-unit element}) for admissible quadruples (\ref{Admissable quadruple cond}) {\it will be denoted $\mathcal{H}(g,h,Q)$} and 
\begin{equation}\label{def H-units group}
\mathcal{H}(G) = \langle \mathcal{H}(g,h,Q) \mid (g,h,Q) \mbox{ as in  \Cref{def h-units} } \rangle
\end{equation}
is the group generated by all $H$-units. 

Among others, the next result says that $\mathcal{H}(g,h,Q)$  is not only  a set, but even a subgroup of the group of reduced norm $1$ elements.\vspace{0,2cm}

\begin{theorem}\label{prop H-unit is H-unit}
Let $(g,h,Q)$ be a triple as in \Cref{def h-units}. Then,
\begin{enumerate}
\item[(1)] $\mathcal{H}(g,h,Q)$ is a finitely generated subgroup of $\SL_1(\Z G)$ and $v_{(x_1,x_2,y_1,y_2)}^{-1} = v_{(-x_1,-x_2,y_2,y_1)},$
\item[(2)] $\mathcal{H}(g,h,Q) \neq 1$ if and only if $[g,h] \notin \langle g \rangle  Q$.
\end{enumerate}
Moreover, for $\mathcal{H}(g,h,Q) \neq 1$, 
\begin{enumerate}
\item[(3)] if\footnote{Since we assume $\mathcal{H}(g,h,Q)$ to be non-trivial, by part (2), $o(gQ)|Q| =2$ if and only if $Q=1$ and $o(g)=2$.} $o(gQ)|Q| =2$, then $\mathcal{H}(g,h,Q) \cong F_3 \times C_2$, and
\item[(4)] if $o(gQ)|Q| > 2$, then $\mathcal{H}(g,h,Q) \cong F_{n}$ with $n = 1 + \frac{(o(gQ)|Q|)^3}{6}\prod_{p} (1 - \frac{1}{p^2})$, where the product runs over the prime divisors $p$ of $o(gQ)|Q|$.
\end{enumerate}
\end{theorem}

As $\mathcal{H}(g,h,Q)$ is finitely generated, one only needs to construct $v_{(x_1,x_2,y_1,y_2)} $ for a finite 
number of admissible $(x_1,x_2,y_1,y_2)$. We have not tried to give a precise upperbound, but in principle this could be done. 
In fact the generators should somehow correspond to the `minimal solutions' of the equations in (\ref{Admissable quadruple cond}).\vspace{0,2cm}

For any $a,c \in G$, with $c\notin N_G(\langle a \rangle)$, consider the corresponding bicyclic unit  $b(a,c) = 1 + (1-a)c\wt{a}$. 
It was proven by Marciniak and Sehgal \cite{MarSeh2} that $\langle b(a,c), b(a,c)^* \rangle \cong F_2$, where $b(a,c)^*$ 
is the image of $b(a,c)$ under the canonical involution of $\Q G$. For  nilpotent groups $G$  a similar statement holds for their 
torsion variant \cite{JanJesTem,MarSeh3}: the set consisting of the unipotent units belonging to $\Bic (G)$ and that form a free product with  $b(a,c)\, a$ is profinitely dense in $\Bic (G)$ \cite{JTT}.
However, \Cref{prop H-unit is H-unit} yields the first generic construction of large free groups (different 
from taking artificially a copy of $F_n$ inside $F_2$). See \cite{GonRioSurvey} for a survey on constructing free subgroups 
of $\U (\Z G).$

\begin{proof}[Proof of \Cref{prop H-unit is H-unit}]
To start, note that $(-x_1,-x_2,y_2,y_1)$ is admissible if $(x_1,x_2,y_1,y_2)$ is. Moreover, it directly follows from the $x_2 \pm x_1 \equiv 0 \equiv y_2 \pm y_1 \mod 2 |Q|$ congruences that both elements in the statement are in $\Z G$. To prove the statements of the result,  we will ``locate'' precisely the $H$-units inside $\Q[\langle g,h,Q\rangle ]$. 

First we consider when $\mathcal{H}(g,h,Q)$ is trivial, i.e. when all elements  $v_{(x_1,x_2,y_1,y_2)}$ equal $1$. 
Or equivalently, all $\frac{1}{2|Q|}\wt{Q}(1-[g,h])\Big(  h\,  [ x_1\, \wt{(-g)_Q} + x_2\, \wt{g}_Q ] +  y_1 \, \wt{(-g)}_Q + y_2 \,\wt{g}_Q \Big)=0$. 
Since $\Z \langle g,h, Q \rangle \frac{\wt{Q}}{|Q|} \cong \Z (\langle g,h,Q\rangle /Q)$, it easily follows that we may 
assume that $Q=1$; and thus $o(h)=2$ and $[g,h]\in \mathcal{Z}(\langle g,h \rangle)$. Hence, we need to verify when all 
$(1-[g,h])\Big(  h\,  [ x_1\, \wt{(-g)_Q} + x_2\, \wt{g}_Q ] +  y_1 \, \wt{(-g)}_Q + y_2 \,\wt{g}_Q \Big)=0$. 
Of course, if $[g,h]=1$ then the latter always holds.
Hence, for the remaining of the proof of part (2) we may also assume that $1\neq [g,h]$. Now, as $o(h)=2$ and $1 \neq [g,h]\in \mathcal{Z}(\langle g,h \rangle)$ one has that $o([g,h])=2$. Furthermore, $o(g)$ is even and $g^{2}$ is central in $\langle g,h\rangle$.  

Now suppose that $[g,h] \in \langle g \rangle$. Write $[g,h]=g^k$ and thus $h^{-1}gh=g^{k+1}$. As $o(h^{-1}gh)=o(g)$ it follows that $k+1$ is odd and thus
$[g,h]\in \langle g^2 \rangle$ is central. Therefore the required triviality follows from $(1-[g,h])(\wt{g}_Q)=0 = (1-[g,h])\wt{(-g)_Q}$. So we have shown that if $[g,h]\in \langle g\rangle $ then $\mathcal{H}(g,h,Q)=1$.
Conversely, assume $\mathcal{H}(g,h,Q)=1$, in particular $v(1,-1,1,-1)=1$ and, because $o(g)$ is even, we thus get  $(1-[g,h]) (h+1) (\wt{(-g)_Q} +\wt{g_Q})=
(1-[g,h]) (h+1) \wt{(g^{2})_Q}=0$. In other words, $(h+1)\wt{g^{2}_Q}= [g,h](h+1)\wt{g^2_Q}$. Clearly, a support argument in the group ring $\Z \langle g,h\rangle$ yields that $[g,h]\in \langle g^2 \rangle$ or $[g,h]\in h\langle g^2 \rangle$. However, the latter is impossible as elements in $h\langle g^2\rangle$ are not central in $\langle g,h\rangle$. Hence, we have shown that if $\mathcal{H}(g,h,Q)=1$ then $[g,h]\in \langle g^2\rangle$, and thus part (2) of the result follows.

From now on we may assume that $\mathcal{H}(g,h,Q) \neq 1$. In particular, according to the  above $[g,h] \notin \langle g \rangle  Q$. Furthermore, as mentioned above, we thus also have that  $o([g,h]Q) = 2$.

Consider the central idempotent $e := \widehat{\langle Q, -[g,h] \rangle}$  in $\Q [\langle g,h,Q \rangle]$ and the associated decomposition $\Q[\langle g,h,Q \rangle] = \Q[\langle g,h,Q \rangle](1-e) \oplus \Q[\langle g,h,Q \rangle]e$. 
Notice that $v_{(x_1,x_2,y_1,y_2)} (1-e) = 1-e$, i.e. the projection on the first part is trivial. Thus we need to prove 
the desired statements (1), (3) and (4) within the second component. 
For this, put $\overline{g} = gQ$ and $\overline{h} = hQ$. As mentioned above, $o(\overline{g})$ is even and note that $\Q[\langle g,h,Q \rangle]e \cong \Q^{\alpha}[\langle \overline{g} \rangle \times \langle \overline{h} \rangle]$
 with $[\alpha] \in H^2(\langle \overline{g} \rangle \times \langle \overline{h} \rangle, \Z^{*})$ determined by $u_{\overline{g}}^{o(\overline{g})} = 1 = u_{\overline{h}}^{o(\overline{h})}$ and $[u_{\overline{g}}, u_{\overline{h}}] = -1$. Under that isomorphism $v_{(x_1,x_2,y_1,y_2)}$ corresponds to the element 
 $$ 1 + u_{\overline{h}}\,  [ x_1\, \wt{(-u_{\overline{g}})} + x_2\, \wt{u_{\overline{g}}} ] +  y_1 \, \wt{(-u_{\overline{g}})} + y_2 \,\wt{u_{\overline{g}}}.$$
A direct verification yields that $\wt{(-u_{\overline{g}})} . \wt{u_{\overline{g}}} = 0$ and 
$\wt{\pm u_{\overline{g}}}^2 = o(\overline{g}) \wt{\pm u_{\overline{g}}}$. 
Also, $[ x_1\, \wt{(-u_{\overline{g}})} + x_2\, \wt{u_{\overline{g}}} ] \, u_{\overline{h}}  
= u_{\overline{h}}  \, [ x_2\, \wt{(-u_{\overline{g}})} + x_1\, \wt{u_{\overline{g}}} ].$ Using all this, 
one could straightforwardly compute the image of the product  $v_{(x_1,x_2,y_1,y_2)} \; v_{(-x_1,-x_2,y_2,y_1)}$ 
and deduce that it is equal to $1$ if and only if $y_1 + y_2 + o(\overline{g})\, y_1y_2 = o(\overline{g})\, x_1 x_2$. 
However, by decomposing $\Q^{\alpha}[\langle \overline{g} \rangle \times \langle \overline{h} \rangle]$ 
further one can also give the following conceptual explanation, which moreover also will yield  the remainder of the result.

Remark that $\overline{g}^2$ is central and hence, by \Cref{decomp voor elke abelse en centrale extensie},
$$\Q^{\alpha}[\langle \overline{g} \rangle \times \langle \overline{h} \rangle] \cong 
\bigoplus\limits_{\chi \in \Lin( \langle \overline{g}^2 \rangle, \Q)} 
\Q(\chi)^{\Tra_{\alpha}(\chi)}[C_2\times C_2] \cong \bigoplus\limits_{d \, \mid \, o(\overline{g}^2)} 
\Q(\zeta_d)^{\alpha_d}[C_2 \times C_2]. $$
Explicitly, denoting $C_2 \times C_2 = \langle a ,b \rangle$, $[\alpha_d] \in 
H^2(\langle a ,b \rangle, \langle \pm \zeta_d \rangle)$ is determined by $u_a^{2} = \zeta_d, u_b^2 = 1$ and $u_a u_b = - u_b u_a$. 
The projections onto the direct summands are given by the transgression maps from \Cref{ring homom from Tra}, 
hence $u_{\overline{g}} \mapsto u_a, u_{\overline{h}} \mapsto u_b$ and $u_{\overline{g}^2} \mapsto \zeta_d\, u_1$.
Next, notice that $\wt{(\pm u_{\overline{g}})} = (1 \pm u_{\overline{g}})\, . \,  \wt{( u_{\overline{g}}^2)}$. 
Due to this, $\wt{(\pm u_{\overline{g}})}$ maps to $0$ in every direct summand except for the one indexed by  the trivial representation of $\langle \overline{g}^2 \rangle$. 
In other words, there is only a single component of $\Q[\langle g,h, Q\rangle]$ where $v_{(x_1,x_2,y_1,y_2)}$ has a 
non-trivial projection and this component is isomorphic to $\Q^{\alpha_1}[\langle a,b \rangle]$ with $u_a^2 = 1 = u_b^2$ 
and $[u_a,u_b]=-1$. In that component the projection of  $v_{(x_1,x_2,y_1,y_2)}$ is
$$
1 +  o(\overline{g}^2) \big( u_b [ x_1 (1- u_a) + x_2 (1+ u_a) ] +  y_1 (1- u_a) + y_2 (1+ u_a) \big).
$$

Restricting to $\Z [\langle g,h, Q\rangle]$, we can now compose with the isomorphism $\phi$ from \Cref{subsection start case elemen ab}, defined in (\ref{def phi from section 8}). This composition is defined by 
\begin{align*}
&  u_a\mapsto \left( \begin{array}{cc}
1 & 0  \\
0 & -1
\end{array}\right),
\quad u_b\mapsto \left( \begin{array}{cc}
0 & 1  \\
1 & 0
\end{array}\right),
\quad u_{ab}\mapsto \left( \begin{array}{cc}
0 & 1  \\
-1 & 0
\end{array}\right).
\end{align*}
By doing so we have obtained a ring morphism $\Phi: \Z[\langle g,h,Q\rangle] \rightarrow \wt{D}$ which on $\mathcal{H}(g,h,Q)$ is defined as
\begin{equation}\label{embedding h-unit in matrices}
\Phi :  v_{(x_1,x_2,y_1,y_2)} \mapsto \left(\begin{array}{cc}
1 + 2 \, o(\overline{g}^2) \, y_2 & 2 o(\overline{g}^2) \,  x_1  \\
2 o(\overline{g}^2)  \, x_2 & 1 + 2 \,o(\overline{g}^2)\, y_1 
\end{array} \right) = \left(\begin{array}{cc}
1 + o(\overline{g}) \, y_2 & o(\overline{g})  \, x_1  \\
 o(\overline{g})  \,x_2 & 1 + o(\overline{g})\, y_1
\end{array} \right).
\end{equation}
Moreover, the elements of $\mathcal{H}(g,h,Q)$ being trivial on all the other components, $\Phi$ is injective on 
$\mathcal{H}(g,h,Q)$. In particular, see (\ref{def reduced norm}), $nr(v_{(x_1,x_2,y_1,y_2)})=1$ if 
$\det(\Phi(v_{(x_1,x_2,y_1,y_2)})) =1$. 
The latter holds as $\det(\Phi(v_{(x_1,x_2,y_1,y_2)})) = 1 + o(\overline{g}) (y_1 +y_2)
+ o(\overline{g})^2 y_1y_2 - o(\overline{g})^2 x_1x_2=1$,  by the definition of an admissible quadruple. 
In fact, since $\Phi$ is injective on $\mathcal{H}(g,h,Q)$, the first condition of being admissible is equivalent 
to $\det(\Phi(v_{(x_1,x_2,y_1,y_2)})) = 1$. 
Now, we also see that the inverse is 
 $$\Phi(v_{(x_1,x_2,y_1,y_2)})^{-1} = \left(\begin{array}{cc}
1 + o(\overline{g}) \, y_1 & - o(\overline{g})  \, x_1  \\
 -o(\overline{g})  \,x_2 & 1 + o(\overline{g})\, y_2
\end{array} \right).$$ 
Hence indeed $v_{(x_1,x_2,y_1,y_2)}^{-1} = v_{(-x_1,-x_2,y_2,y_1)}$. In conclusion, $\mathcal{H}(g,h,Q)$ is a subgroup of $\SL_1(\Z G)$. 

All the above in fact yields more. Namely that 
$$\Phi : \mathcal{H}(g,h,Q) \hookrightarrow \Gamma(o(\overline{g})),$$
where $\Gamma(n)$ is the principal congruence subgroup of level $n$ of $\SL_2(\Z)$. However, $\Phi$ is not onto due to the congruences $x_2 \pm x_1 \equiv 0 \equiv y_2 \pm y_1 \mod 2|Q|$.

 {\bf Claim:} $x_2 \pm x_1 \equiv 0 \equiv y_2 \pm y_1 \mod 2|Q|$ if and only if $x_i = |Q| l_i, y_i = |Q| t_i$ 
 for $l_i,t_i \in \N$ such that $l_1 \equiv l_2 \mod 2$ and $t_1 \equiv t_2 \mod 2$.

Indeed, if $x_2 \pm x_1 \equiv 0\mod 2|Q|$, then $2 x_1 \equiv 0$ and $2 x_2 \equiv 0$ $\mod 2|Q|$.
Hence $x_2 \equiv x_1 \equiv 0 \mod |Q|$ and consequently $x_i = |Q| l_i$, 
for some $l_i \in \N$.
Now, as $2|Q|$ divides $x_2 \pm x_1$, one must have that $ 2 \mid l_1 \pm l_2$,
as desired. Conversely, $x_i$ of that form clearly satisfy  $x_2 \pm x_1 \mod 2|Q|$. 
The proof for the $y_i$ is exactly the same, hence the claim follows.

Now denote
\begin{equation}\label{definition Vm}
V_{m} = \left\{ \left(\begin{array}{cc}
1 + m \, l_2 & m \, t_1  \\
 m \,t_2 & 1 + m\, l_1
\end{array} \right) \in \SL_2(\Z) \mid l_1 \equiv l_2 \text{ and } t_1 \equiv t_2 \mod 2 \right\}.
\end{equation}
Notice that the groups $U_i \cong \phi(U_i)$ from the proof of \Cref{the groups U_i for elemen ab} are equal to $V_{2^{i}}$. By the claim above 
\begin{equation}\label{H-units as Vm}
\Phi(\mathcal{H}(g,h,Q)) = V_{o(\overline{g})|Q|}.
\end{equation}
Since, $\Gamma(2m) \leq V_m \leq \Gamma(m)$ for every $m \geq 1$, we have obtained that $\Phi(\mathcal{H}(g,h,Q))$ is a finite index subgroup in $\SL_2(\Z)$. In particular it is finitely generated, finishing the proof of (1).

Finally, to obtain (3) and (4) we recall the well-known fact that $\SL_2(\Z) \cong C_4 \star_{\langle -1 \rangle} C_6$ with $C_4 = \langle \left(\begin{array}{cc}
0  & -1  \\
1 & 0 \\
\end{array} \right) \rangle$ and $C_6 = \langle \left(\begin{array}{cc}
0  & -1  \\
1 & -1 \\
\end{array} \right) \rangle .$ Hence, by Kurosh's theorem, $\Gamma(m)$ is free if and only if it is torsion-free. 
Moreover, all periodic subgroups are conjugated to a subgroup of  $C_4$ or $C_6$. As $\Gamma(m)$ is normal in $\SL_2(\Z)$, to verify when $\Gamma (m)$ is free  it is enough to verify explicitly 
which powers of the two matrices are in $\Gamma(m)$. By doing so we see that $\Gamma(m)$ is torsion-free if 
$m \neq 2$ and $-1$ is the only torsion element in $\Gamma(2)$. As a consequence, the same conclusion is valid 
for $V_m$ instead of $\Gamma(m)$. Also recall, e.g. see the proof of \Cref{the groups U_i for elemen ab}, 
that $V_2 = U_1 \cong F_3 \times \langle -1 \rangle$. All this applied to $\Phi(\mathcal{H}(g,h,Q)) = V_{o(\overline{g})|Q|}$ 
yields that $\mathcal{H}(g,h,Q)$ is a finitely generated free group, except if $|Q| = 1$ and $o(\overline{g}) =2$. 
In this case $\mathcal{H}(g,h,Q) \cong F_3 \times C_2.$ 

When $m := o(\overline{g})|Q| >2$ the rank of the free group $V_m$ can be computed. In order to do so, recall that  $o(\overline{g})$ is even.
In particular $\mathcal{H}(g,h,Q) \cong V_m$ with $m$ even. Next, with the same method as in the proof of \Cref{the groups U_i for elemen ab}, it can be shown that $[\Gamma(m) : V_m] = 2$ when $m$ is even. Also, because $V_ m \leq F_2 \leq \Gamma(2) = F_2 \times C_2$, we need to compute $[F_2 : V_m] = [\Gamma(2):V_m]/2$. This can readily be done using (\ref{formula index congruence}):
$$
[\Gamma(2):V_m] = [\Gamma(2):\Gamma(m)]\; [\Gamma(m): V_m]  = 2 \frac{[\SL_2(\Z): \Gamma(m)]}{[\SL_2(\Z): \Gamma(2)]}  = 2 \frac{m^3}{6} \prod\limits_{p \mid m} ( 1 - \frac{1}{p^2}).
$$
Finally, by Nielsen-Schreier's formula, we obtain that $V_m \cong F_n$ with $n = 1 +\frac{m^3}{6} \prod\limits_{p \mid m} ( 1 - \frac{1}{p^2})$. 
\end{proof}

In order to make further use of it in the text, we explicitly state the following fact that has been noticed in the previous proof.

\begin{remark}\label{remark on triviality H-unit}
Let $(g,h,Q)$ be a triple as in \Cref{def h-units}. Note that $[g,h] \notin Q$ yields that $o([g,h]Q) = 2$. 
Consequently, $\frac{1}{2}\widehat{Q}(1-[g,h]) = \widehat{Q}\; \widehat{(-[g,h])} = \widehat{\langle Q, -[g,h] \rangle}.$ 
\end{remark}

Interestingly, a further inspection of the proof of \Cref{prop H-unit is H-unit} yields the following matrix description of $\mathcal{H}(g,h,Q)$.

\begin{theorem}\label{H-unit contain principal congruence in one comp}
Let $(g,h,Q)$ be a triple such that $\mathcal{H}(g,h,Q)\neq 1$ and let $H = \langle g,h,Q \rangle$. Then there exists a unique primitive central idempotent $e$ of $\Q H$ such that\footnote{Confusing at first, $1 - e + \Q He$ however consists simply of the elements of $\Q H$ projecting to the identity   in all simple components of $\Q H$ except the simple component $(\Q H)e$ where any element of $\Q He$ is taken.}
$$ \mathcal{H}(g,h,Q) = \SL_1(\Z H) \cap  \left( 1 - e + \Q He \right) .$$ 
Moreover, $\mathcal{H}(g,h,Q) = 1 - e + V_{m}$ for $m = o(gQ)|Q|$ and $V_m$ defined in (\ref{definition Vm}). Also, considering any maximal order $\O$ in $\Q He$ and denoting by $\Gamma(m)$ the principal congruence subgroup of level $m = o(gQ)|Q|$ in $\SL_1(\mathcal{O})$, one has that 
$$\Gamma(2m) \leq V_m \leq \Gamma(m) \text{ and } [\Gamma(m): V_m] = 2.$$
In particular, $\mathcal{H}(g,h,Q)$ is a finite index normal subgroup of $1 -e + \U(\Z He)$.
\end{theorem}
\begin{proof}
Note that the proof of \Cref{prop H-unit is H-unit} in fact works locally, in the sense that it is carried out in $\Q[\langle g,h,Q \rangle]$ 
rather than $\Q[G]$. Furthermore, the morphism $\Phi$ defined in (\ref{embedding h-unit in matrices}) in fact coincides with the projection 
onto its simple component $\Q^{\alpha_1}[\langle a,b \rangle]$ (notations as in the proof). More precisely, this projection factors 
through $\Q^{\alpha}[\langle \ov{g} \rangle \times \langle \ov{h}\rangle]$. Projecting thereon is given by first multiplying with the central idempotent $e$ 
and then the  subsequent 
projection onto $\Q^{\alpha_1}[\langle a,b \rangle]$ is given by multiplying with  $e_{T}$ (we use the subscript $T$ because 
the component arises from the trivial character). 
The element $e_T \, e$ is a primitive central idempotent of $\Q[\langle g,h,Q \rangle]$ and $\Phi$ is a concrete realization of the composition of 
multiplying with $e_T \, e$ followed by applying $\phi$ (defined in (\ref{def phi from section 8})). In particular 
$\Phi$ is injective on $\Q H \cap( 1 - e_T\, e + \Q H e_T \, e)$. 

Next, as $ge_T\, e= a$ and $he_T\, e =b$, one has that $o(ge_T\, e) =  2 = o(he_T\, e)$. 
Using this a direct verification yields that any element in 
$\Q H \cap (1 - e_T\, e + \Q H e_T\, e)$ can be rewritten in the form (\ref{eq H-unit element}) for some $x_1,x_2,y_1,y_2 \in \Q$. However, 
as explained in the proof of \Cref{prop H-unit is H-unit}, via the injectivity of $\Phi$, such an element is in $\SL_1(\Z G)$ if and only if 
$(x_1,x_2,y_1,y_2)$ is admissible for $(g,h,Q)$. Hence $\mathcal{H}(g,h,Q) = \SL_1(\Z H) \cap (1 - e_T\, e + \Q He_T\, e)$ and thus $e_T\, e$ is the 
desired primitive central idempotent of $\Q H$. The moreover part of the statement of the result was explicitly obtained in the proof of 
\Cref{prop H-unit is H-unit}, see (\ref{H-units as Vm}). The final finite index assertion holds because $\Q He_Te \cong \Ma_2(\Q)$  hence 
the center of the unit group of an order therein is finite and therefore $\SL_1(\Z Ge)$ is of finite index in $\U( \Z He)$. 

It remains to prove that $\mathcal{H}(g,h,Q)$ is normal in $1 -e + \U(\Z He)$ or in other words that $V_m$ is normal in 
$\U (\Z He) \cong \U(\Z^{\alpha_1}[\langle a,b \rangle]).$ A presentation of the latter group was obtained in 
\Cref{prop description unit twisted group ring}. The image under $\phi$ of the generators are the following matrices:
$$\left\{ \left(\begin{array}{cc}
    1 &2  \\
     0& 1 
\end{array} \right), \left(\begin{array}{cc}
    1 &0  \\
    2& 1 
\end{array} \right), \left(\begin{array}{cc}
    1 &0  \\
    0& -1 
\end{array} \right), \left(\begin{array}{cc}
    0 & 1  \\
    1& 0 
\end{array} \right)
\right\}.$$
A direct verification yields that these matrices normalize $V_m$ for every $m$, finishing the proof.
\end{proof}

Note also that, unlike primitive $H$-units, the elements $v_{(x_1,x_2,y_1,y_2)}$ are not necessarily unipotent. 
For instance, using the map $\Phi$ in (\ref{embedding h-unit in matrices}) we see that 
\begin{equation}\label{unitpotent h-unit condition}
-1 + v_{(x_1,x_2,y_1,y_2)}  \text{ is nilpotent } \Leftrightarrow \left\lbrace \begin{array}{l}
y_1 = - y _2\\
y_1y_2 = x_1 x_2
\end{array} \right.
\end{equation}

Finally, it would be interesting to investigate how different  $H$-units interact. 

\begin{question}\label{question grp gen by two triples}
Let $(g_i,h_i,Q_i),$ with $i = 1,2$, be two different triples as in \Cref{def h-units}. 
What is the structure of the  group $\langle \mathcal{H}(g_1,h_1,Q_1), \mathcal{H}(g_2,h_2,Q_2) \rangle$ ?
\end{question}

A first interesting contribution to \Cref{question grp gen by two triples} would be to determine when it is the direct product of 
$\mathcal{H}(g_1,h_1,Q_1)$ and $\mathcal{H}(g_2,h_2,Q_2)$.

\begin{remark}\label{remark on the name}
The proof of \Cref{prop H-unit is H-unit} has shown that locally, i.e. in $\Q[\langle g,h,Q\rangle]$, $v_{(x_1,x_2,y_1,y_2)}$ injects via $\Phi$ to an element of $\SL_1(\Z^{\alpha}[C_2 \times C_2])$. By \Cref{prop description unit twisted group ring}, $\Phi(v_{(x_1,x_2,y_1,y_2)})$
is a product of elements as in \Cref{primitive h-units}. In particular the elements in the latter are somehow the smallest, hence the name primitive. Besides, the map $\Phi$ is directly related with decomposing $\langle g,h,Q\rangle$ as a non-split extension of $\langle Q,[g,h], g^2 \rangle$ by $C_2 \times C_2.$ Thus the non-triviality of that second cohomology group was explicitly necessary for the existence of $H$-units. This clarifies our choice of the name for those units. 
\end{remark}

\subsection{H-units extend bicyclic units with infinite index}\label{subsection H-units infnite index ext}

Recall that a simple quotient of $\Q [G]$ is called an {\it exceptional component of type II} if it is of the 
form\footnote{The division algebras appearing in these matrix algebras are exactly those having an order with finite unit 
group \cite[Theorem 2.10.]{BJJKT}.} $\Ma_2(\Q(\sqrt{-d}))$ or $\Ma_2(\qa{-a}{-b}{\Q})$ with $a$ and $b$ strict positive 
integers and $d \in \N$ (i.e. $\qa{-a}{-b}{\Q}$ is a totally definite quaternion algebra) and it is called an {\it exceptional component of type I} if it is a division algebra which is not a totally definite 
quaternion algebra. The terminology `exceptional' refers to the fact that in their absence the bicyclic units are of finite index 
in $\SL_1(\Z G)$, e.g. see \Cref{Jespers-LEal for twisted bicyclic}. Therefore, as $\SL_1(\O)$ for $\O$ an order in a division 
algebra has no unipotent units, the focus of current research is on the type II exceptional components

As demonstrated by \cite[Appendix A]{BJJKT} the most recurrent\footnote{$\Q G$ has an exceptional $2\times 2$ component exactly 
when it maps onto one of the $52$ groups in \cite{EKVG}. Now the table in \cite[Appendix A]{BJJKT} says that only $16$ of 
them have no simple component of the type $\Ma_2(\Q)$. Among nilpotent groups there are only $5$ such groups. } 
component of that type is $\Ma_2(\Q)$ and this is a simple quotient of $\Q G$ if and only if $G$ surjects onto $D_8$ or $S_3$.  
It follows that, when $3 \nmid |G|$ and $\Ma_2(\Q)$ is a simple component of $\Q G$, one has a triple $(g,h,Q)$ satisfying the following:
\begin{equation}\label{Conditions D8 triple}
g,h \in N_G(Q) \text{ and } \langle g,h,Q \rangle /Q \cong D_8.
\end{equation}
In fact the stronger properties $Q \triangleleft G$ and $G/Q = \langle gQ, hQ \rangle = \langle ghQ, hQ \rangle \cong D_8$ are satisfied. We will call a triple satisfying (\ref{Conditions D8 triple}) {\it a $D_8$-triple.} \vspace{0.1cm}

\noindent {\it Convention:} For a $D_8$-triple $(g,h,Q)$ we will always assume that $o(gQ) = o(hQ) = 2$ and $D_8 \cong C_4 \rtimes C_2 = \langle ghQ\rangle \rtimes \langle hQ \rangle.$\vspace{0.1cm}

Using $H$-units built on $D_8$-triples we can describe generically a subgroup of finite index for the following class of groups. 

\begin{theorem}\label{H-units for 2-grps give finite index}
Let $G$ be a $2$-group such that the only exceptional components of $\Q G$ are of the form $\Ma_2(\Q)$, then $\langle \Bic (G) , \mathcal{H}(G) \rangle $ is of finite index in $\SL_1(\Z G)$. Consequently,  $\langle \mathcal{B} (G) , \mathcal{H}(G) \rangle$ is of finite index in $\U (\Z G)$.
\end{theorem}
In the above $\mathcal{B}(G)$ denotes the subgroup of $\U(\Z G)$ generated by the bicyclic and Bass units (see \cite[Section 1.2.]{EricAngel1} for definitions).

\begin{proof}
Denote by $\mathcal{E}_{exc}$ the set of primitive central idempotents $e$ such that $\Q G e \cong \Ma_2(\Q)$. 
For $e \in \mathcal{E}_{exc}$ consider the associated projection $\pi_e :  G \rightarrow Ge$ 
(which is the restriction of the natural projection from $\U(\Q G)$ to $\U(\Q G)e$). Since $G$ is a $2$-group one has that 
$\pi_e(G) = Ge \cong D_8$. Consider $Q = \ker( \pi_e)$ and take $g,h \in G$ such that 
$Ge \cong \langle \pi_e(g) \rangle \rtimes \langle \pi_e(h) \rangle$, with $o(gQ) = 4, o(hQ) = 2$ and 
$[gQ,hQ] = (gQ)^2$. In this way we obtain a $D_8$-triple $(gh, h, Q)$ such that $1 \neq \mathcal{H}(gh,h,Q)$ (by \Cref{prop H-unit is H-unit}) 
and $G = \langle gh,h,Q \rangle$ . This allows to apply \Cref{H-unit contain principal congruence in one comp} to conclude that 
$H_e := \mathcal{H}(gh,h,Q) \leq \SL_1(\Z G)$ is of finite index in $1 - e + \U(\Z Ge)$.

Now consider $\mathcal{E} := \PCI (\Q G) \setminus \mathcal{E}_{exc}$ and take $e^{\prime} \in \mathcal{E}$, i.e. $\Q G e^{\prime} \ncong \Ma_2(\Q)$. 
By assumption, if $\Q G e^{\prime}$ is a division algebra it needs to be a totally definite quaternion algebra, and hence $\SL_1(\Z Ge^{\prime})$ is 
finite \cite{Kleinert}. It now remains to consider the case that $\Q G e^{\prime}$ is not a division algebra, say $\Ma_n(D)$, and let $\O$ 
be an order in the division algebra $D$. In this case classical arguments can be used. Namely, for such a component \Cref{Jespers-LEal for twisted bicyclic} gives a $y\in \Z$ and 
subgroup $1 - e^{\prime} + E_n(y\mathcal{O}) \leq \Bic (G)$. Because $\Q G e^{\prime}$ is non-exceptional, the solution to the subgroup 
congruence problem in higher rank (e.g. see \cite[Theorem 11.2.3]{EricAngel1}) yields that $E_{e^{\prime}} := E_n(y\mathcal{O})$ is 
of finite index in $\GL_n(\O)$. All together we obtained a subgroup $H = \prod_e H_e \times \prod_{e^{\prime}} E_{e^{\prime}}$ of $\SL_1(\Z G)$ 
which is of finite index in $\prod_{e \in \PCI (\Q G))} \SL_1(\Z G e)$, hence also in $\SL_1(\Z G)$. The second part now classically follows 
by a Bass-Milnor Theorem \cite[Theorem 11.1.2]{EricAngel1} which says that the Bass units map to a subgroup of finite index in the 
Whitehead group $K_1(\Z G) := \GL (\Z G)^{ab}$ of $\Z G$ and hence \cite[Prop. 9.5.11]{EricAngel1} jointly with $H$ it is of finite 
index in $\U( \Z G)$. In particular also $\langle \mathcal{B} (G) , \mathcal{H}(G) \rangle$ is of finite index in $\U (\Z G)$.
\end{proof}

\begin{remark}\label{extend Jespers-Leal remark}
The first part of the proof of \Cref{H-units for 2-grps give finite index} shows that when $G$ does not map onto $S_3$ 
one can extend the Jespers-Leal theorem \cite{JesLea2} by including also the exceptional component $\Ma_2(\Q)$. 
Concretely, let $e \in \PCI(\Q G)$ such that $\Q Ge \cong \Ma_2(\Q)$ and $3 \nmid |Ge|$. Then $\SL_1(\Z G)$ 
contains a subgroup $W_e$, consisting of $H$-units, that is of finite index in $ 1-e + \U(\Z Ge)  \subset 1-e + \Q Ge$. 
Moreover $W_e$ contains a principal congruence subgroup of level $2 o(gQ) |Q| = 2\cdot 2 \cdot \frac{|G|}{8} = \frac{|G|}{2}$ of $\SL_1(\Z Ge)$. More precisely, $W_e \cong V_{|G|/4}.$ 

The condition $3 \nmid |Ge|$ stems from the fact that the proof of \Cref{prop H-unit is H-unit} requires \Cref{prop description unit twisted group ring}. On turn the latter originates from a splitting of $D_8$ and needs a precise understanding on subgroups of small index in $\SL_2(\Z).$ Using \cite{MPV}, the necessary tools seem to exist to extend the results to $S_3$. In particular we expect that the above and \Cref{H-units for 2-grps give finite index} extends to non $2$-groups.
\end{remark}

By a result of Jespers and del R\'io \cite{JesRioCrelle} natural examples for finite groups $G$ as in Theorem \ref{H-units for 2-grps give finite index} 
can be found when $\U( \Z G)$ is virtually a direct product of free products of abelian groups. Such groups have been classified and  
the $2$-groups of this type  are isomorphic to $K \times C_2^n$ with $K$ one of the following types :
\begin{enumerate}
\item [$\mathcal{G}_1$:]  $\langle x,y \mid x^4 = y^4 = 1 \text{ and } y^2, x^2 =[x,y] \text{ central } \rangle,$
\item [$\mathcal{G}_2$:] $\langle x, y_1 , \ldots, y_n \mid x^4 = y_i^2 = [y_i,y_j] = 1 \text{ and } x^2, [x,y_i] \text{ central } \rangle,$
\item[$\mathcal{G}_3$:] $\langle x, y_1 , \ldots, y_n \mid x^4 = y_i^4 = y_i^2[x,y_i]= [y_i,y_j] = 1 \text{ and } x^2, y_i^2 \text{ central } \rangle,$
\item[$\mathcal{G}_4$:] $\langle x, y_1 , \ldots, y_n \mid x^2 = y_i^2 = [y_i,y_j] =[[x,y_i],y_j] = [x,y_i]^2=1\rangle,$
\item[$\mathcal{G}_5$:] $\langle x, y_1 , \ldots, y_n \mid x^2 = y_i^4 = y_i^2[x,y_i] = [y_i,y_j] = [[x,y_i],x] = 1 \rangle,$
\item[$\mathcal{G}_6$:] $\langle x, y_1 , \ldots, y_n \mid x^4 = y_i^4 = x^2 y_1^2 = y_i^2[x,y_i] = [y_i,y_j]=[y_i^2,x]=1 \rangle,$
\item[$\mathcal{G}_7$:] $\langle x, y_1 , \ldots, y_n \mid x^4 = x^2y_i^4 = y_i^{-2}[x,y_i]=[y_i,y_j] = 1 \rangle.$
\end{enumerate}

\begin{remark}
In \cite{JesRioCrelle} the relation $y_i^2[x,y_i]=1$ is written for the groups in class $\mathcal{G}_7$ (which are the groups of `type (g)'in loc. cit.). Inspection of the proof however shows that it must be $y_i^{-2}[x,y_i]=1.$ 
\end{remark}

The group in $\mathcal{G}_1$ is simply $C_4 \rtimes C_4$ with $C_4$ acting by inversion. It is known \cite[Corollary 12.7.2]{EricAngel1} that $\Bic(C_4 \rtimes C_4)$ is of infinite index in $\SL_1(\Z G)$. 

For groups $G$  as above, \cite[Theorem 1.3]{JesRioCrelle} tells that the non-division algebra components of  $\Q G$  are isomorphic to 
$\Ma_2(\Q)$. Moreover, $\Q G$ does not  have exceptional components of type I. Thus we can use \Cref{H-units for 2-grps give finite index} 
to obtain that $\mathcal{H}(G)$ is a so-called congruence subgroup of $\SL_1(\Z G)$. More importantly, as a byproduct we obtain a more precise version of \cite[Theorem 1.1]{RioRuiz} and a classification-free proof.

\begin{corollary}\label{families of grps with H-unit new}
Let $G = K \times C_2^n$, with $K$ a group in $\mathcal{G}_1 \cup \ldots \cup \mathcal{G}_7$, and let $q$ denote the number of simple components 
of the type $\Ma_2(\Q)$ in $\Q G$. Then, there exist $D_8$-triples $(g_i,h_i,Q_i)$, $1 \leq i \leq q$, such that 
$\langle \mathcal{H}(g_i,h_i,Q_i) \rangle \cong \prod_{i=1}^q  \mathcal{H}(g_i,h_i,Q_i)$ is a finite index normal subgroup of 
$\SL_1(\Z G)$. In particular, the $H$-units $\mathcal{H}(G)$ are of finite index despite that $\Bic(G)$ can be of infinite index. 
Moreover, if $G \ncong D_8$, $\langle \mathcal{H}(g_i,h_i,Q_i) \rangle \cong F_n^{q}$ with $n = 1 + \frac{|G|^3}{2.4^4}$.
\end{corollary}
In fact $\langle \mathcal{H}(g_i,h_i,Q_i) \rangle$ is the largest finite index subgroup in $\SL_1(\Z G)$ which is the direct product of free groups. This follows from \Cref{H-unit contain principal congruence in one comp}, saying that   $\mathcal{H}(g_i,h_i,Q_i)  = \SL_1(\Z G) \cap 1 -e + \Q G e$ for an associated $e \in \PCI(\Q G)$, and the indecomposability of $SL_1(\O)$ for $\O$ any order in some $\Ma_n(D)$ ($n\neq 1$) \cite{KleRio}.
That $\prod_{e \in \PCI(\Q G)} \SL_1(\Z G) \cap 1-e + \Q Ge$ is the largest direct product of free groups was already obtained by del R\'io and Ruiz in \cite[Theorem 1.1]{RioRuiz}. However our proof is uniform, i.e. we do not use the classification for $G$, and yield more explicit generators.

\begin{proof}[Proof of \Cref{families of grps with H-unit new}]
Let $e \in \PCI(\Q G)$. By \cite[Theorem 1.3]{JesRioCrelle}, either $\Q Ge$ is a totally definite quaternion algebra or it is isomorphic 
to $\Ma_2(\Q)$. In the former case $\SL_1(\Z G e)$ is finite, so we only need to consider the case $\Q Ge \cong \Ma_2(\Q).$ 
As pointed out in \Cref{extend Jespers-Leal remark}, the proof of \Cref{H-units for 2-grps give finite index} gives a $D_8$-triple $(g_e,h_e,Q_e)$ 
such that $\mathcal{H}(g_e,h_e,Q_e)$ is of finite index in $1-e + \mathcal{U}(\Z Ge)$. In particular, picking one such triple for every $e$ 
now yields a subgroup of $\SL_1(\Z G)$ which is even of finite index in the overgroup $\prod_e \SL_1( \Z Ge)$ such that 
$\langle \mathcal{H}(g_i,h_i,Q_i) \rangle \cong \prod_{i=1}^q \mathcal{H}(g_i,h_i,Q_i)$. 
Following \Cref{H-unit contain principal congruence in one comp}, $\mathcal{H}(g_e,h_e,Q_e)$ is isomorphic to a certain group $V_m$ which 
is normal in $\U(\Z Ge)$. This implies that $\prod_{i=1}^q  \mathcal{H}(g_i,h_i,Q_i)$ is normal in $\prod_e \U(\Z Ge)$ and in particular 
in the subgroup $\U (\Z G)$. Now, as mentioned earlier for example in the group $\mathcal{G}_1$ the bicyclic units are of 
infinite index. This finishes the first part. 

That $\mathcal{H}(g_i,h_i,Q_i) \cong F_n$ with $n$ as in the statement is a combination of \Cref{prop H-unit is H-unit} and \Cref{extend Jespers-Leal remark}. More precisely, the latter says that $\mathcal{H}(g_i,h_i,Q_i) \cong V_m$ with $m =o(gQ)|Q| = \frac{|G|}{4}$. Now would $q\neq 0$ and $o(gQ)|Q| =2$, then $Q = 1$ and $o(g)=2.$ This however entails that $G \cong D_8,$ which was excluded. Therefore, \Cref{prop H-unit is H-unit} yields that $\mathcal{H}(g_i,h_i,Q_i)$ is a free group and as $m$ is a $2$-power the product in \Cref{prop H-unit is H-unit} only runs over the prime divisor $2$, yielding the desired formula.
\end{proof}


To finish this section, we exhibit a first surprising application of $H$-units which  indicate a first of many new possible paths of research. 

\subsubsection*{Example of a normal complement via $H$-units}
Consider 
$$\Gamma = C_4 \rtimes C_4 := \langle a,b \mid a^4=b^4=1, \, a^b = a^{-1} \rangle,$$
a group in the class $\mathcal{G}_1$.
Recall that 
$\Q[\Gamma] \cong \Q[C_2 \times C_2] \op 2 \Q(i) \op \qa{-1}{-1}{\Q} \op \Ma_2(\Q)$. In particular we can apply all the results above. Concretely, notice that $(ab, b , \langle b^2 \rangle)$ is a $D_8$-triple fulfilling the non-triviality condition of \Cref{prop H-unit is H-unit} with $o(\overline{ab})|\langle b^2 \rangle | = 4$. As there is only one matrix component and $\mathcal{Z}(\U (\Z \G))$ is finite, \Cref{prop H-unit is H-unit} now yields that
$$\mathcal{H}(ab, b , \langle b^2 \rangle) \cong F_{1 + \frac{4^3.3}{6.4}} = F_{9} \text{ finite index normal subgroup in }\U(\Z \G).$$
Furthermore, by \Cref{H-unit contain principal congruence in one comp}, 
$\mathcal{H}(ab, b , \langle b^2 \rangle)) = \SL_1(\Z \G) \cap (1- e + \Q \G e)$ with $\Q \G e \cong \Ma_2(\Q)$ and 
$\mathcal{H}(ab, b , \langle b^2 \rangle)) = 1 -e + V_4$ is a normal subgroup. Furthermore, 
$$\SL_1(\Z \G) = \langle \mathcal{H}(ab, b , \langle b^2 \rangle)), G \cap \SL_1(\Z \G) \rangle.$$
The latter can be seen from the facts: (i) $\SL_1(\Z \G) \leq 1 \times \SL_1\qa{-1}{-1}{\Z} = Q_8 \times \SL_2(\Z)$ and (ii) $\G \cap \SL_1(\Z \G) = \langle a^2 \rangle $ (all other $g \in \G$ have a non-trivial projection in one of the commutative components) and with $a^2$ mapping to $-1$ in $\SL_1\qa{-1}{-1}{\Z}$. With a bit more of work\footnote{For example, via the methods as in the proof of Claim 4 in the proof of \Cref{Theorem abelianisation of only M2(Q) components} one can prove that $\U(\Z \G) = \langle \pm \G, \SL_1(\Z G) \rangle$. The latter fact can alternatively be deduced from \cite[Theorem 5.1]{JesLea}.} one can proof that $\U (\Z \G) = \langle \mathcal{H}(ab, b , \langle b^2 \rangle), \pm \G \rangle$. All together we recover  \cite[Theorem 5.1.]{JesLea} (also see \cite[Example 5.5]{BMM}):
$$\U(\Z[\Gamma]) \cong F_9 \rtimes \pm \Gamma $$
where $F_9 = \mathcal{H}(ab, b , \langle b^2 \rangle)$. This description of the free group $F_9$ yields a {\it new set of generators.} \vspace{0,2cm}

{\it A full presentation} can also be obtained. A presentation was given in \cite[Example 5.5]{BMM}, 
using \cite[Theorem 5.1.]{JesLea}, but our methods will yield a considerably more symmetric presentation. Indeed, 
as $F_9 = \mathcal{H}(ab, b , \langle b^2 \rangle) = 1 - e + V_4$, it is enough to compute the action of 
$\pi_e(\G ) = \langle \pi_e(ab), \pi_e(b) \rangle \cong D_8$ on $V_4$, all seen as subgroups of $\GL_2(\Z)$. 
To do so, recall that \Cref{prop description unit twisted group ring} and its proof yields a matrix representation of 
all elements and also record the required actions. More precisely, $V_2 = \langle w^2,v^2,w^{-1}v, -1\rangle$ with 
$w= \left( \begin{array}{ll} 1 & 2 \\ 0 & 1  \end{array} \right)$ and $v= \left( \begin{array}{ll} 1 & 0 \\ 2 & 1  \end{array} \right)$. 
Also, $\pi_e(ab) = \left( \begin{array}{lr} 1 & 0 \\ 0 & -1  \end{array} \right)$ and $\pi_e(b) = 
\left( \begin{array}{ll} 0 & 1 \\ 1 & 0  \end{array} \right).$ 

Via a direct application of the Reidemeister-Schreier method one obtains that 
\begin{equation}\label{gens V4}
V_4 = \langle v^4,w^4, w^2 v^{-2}, x, v^2 x v^{-2} \mid x \in S \rangle
\end{equation}
with $S = \{ (w^{-1}v)^2, (wv^{-1})^2, (w^{-1}v) (wv^{-1})^{-1} \}.$
Using the action in \Cref{prop description unit twisted group ring} (or the matrices above) one readily verifies that 
$$
y^{\pi_e(ab)} =y^{-1} \text{ for } y \in \{ v^4, w^4, (w^{-1}v) (wv^{-1})^{-1}\}.
$$
and
$$
z^{\pi_e(b)} = z^{-1} \text{ for } z \in \{w^2v^{-2}, (w^{-1}v)^2, (wv^{-1})^2 \}.
$$
Furthermore, $(w^2v^{-2})^{\pi_e(ab)} = w^{-4}(w^2v^{-2})v^4$ and $$\left((w^{-1}v) (wv^{-1})^{-1}\right)^{\pi_e(b)} = (w^{-1}v)^2 (w^{-1}v) (wv^{-1})^{-1} (wv^{-1})^2.$$ 
The remaining actions are computed similarly, yielding all together the following presentation:
$$V(\Z \G) = \langle s,t,u, x_1, x_2, x_3, y_1, y_2,y_3 \rangle \rtimes \langle a,b \rangle \cong F_9 \rtimes \Gamma$$
where the action is the following:
$$
\begin{array}{lll}
    s^{ab} = s^{-1} & x_1^{ab} = x_2 & y_1^{ab} = s^{-1}y_2s \\ 
    t^{ab} = t^{-1}& x_2^{ab} = x_1 & y_2^{ab} = s^{-1}y_1s \\
    u^{ab} = t^{-1} u s & x_3^{ab} = x_3^{-1} & y_3^{ab} = s^{-1}y_3s \\
    s^{b} = t & x_1^{b} = x_1^{-1} &  y_1^{b} = u y_1^{-1} u^{-1} \\ 
    t^{b} = s&  x_2^{b} = x_2^{-1} & y_2^{b} = u y_2^{-1}u^{-1}\\
    u^{b} = u^{-1} & x_3^{b} = x_1^{-1} x_3 x_2 & y_3^{b} = u y_1^{-1} y_3 y_2 u^{-1}\\
\end{array}
$$\vspace{0,2cm}

\noindent The previous example naturally raises the following question:

\begin{question}\label{question when generic constr normal complement}
When is $\langle \Bic (G), \mathcal{H}(G) \rangle $ a normal complement for the trivial units? Also, when is $\SL_1(\Z G)$ a complement?
\end{question}

\subsection{Applications to the abelianisation conjectures}\label{Sectie ab conjectures}

\hspace{0,2cm}\newline
 \indent  
  Recall that $\U (\Z G)$ is a finitely generated group and hence $\U (\Z G)^{ab}$ is the direct product of a finitely generated free abelian group 
  $\Z^n$ and a finite abelian group. One calls $n$ the rank of the abelian group $\U (\Z G)^{ab}$. Also recall that the center $\ZZ (\mathcal{U}(\Z G))$ of $\mathcal{U}(\Z G)$ is finitely generated, hence its rank is finite. Finally, recall that $\U(\Z G) = \pm V(\Z G)$
  with $V(\Z G)$ the {\it group of invertible elements with augmentation one}.

We finish the article with  another application of $H$-units. We do this by  giving an answer, for the class of groups considered above, 
 to the following recent questions on the abelianisation of the unit group of $\Z G$:
\begin{enumerate}
\item[(R1)] Is the rank of the abelian groups  $\ZZ (\mathcal{U}(\Z G))$ and  $\U(\Z G)^{ab}$ equal? In particular if  $\ZZ (\mathcal{U}(\Z G))$ is finite, is $\U(\Z G)^{ab}$ also finite? (see \cite[Question 7.8 and Proposition 7.9]{BJJKT}, The labelling of the questions is taken from \cite{BMM}.) 

\item[(P)] If $V(\Z G)^{ab}$ contains an element of prime order $p$, then so does  $G^{ab}$?
(see \cite[page 2]{BMM})
\item[(E1)] Is $\exp V(\Z G)^{ab} = \exp G^{ab}$? (see \cite[page 2]{BMM})
\end{enumerate} 

\noindent  The labelling of the questions is taken from \cite{BMM}. In question (E1) by $\exp \G$ of a group $\G$ we mean the least common multiple of the elements of finite order in $\G$. Note that question (E2) implies question (P).

In general, as proven in  \cite[Proposition 6.1]{BJJKT}, 
$$\rank  \U(\Z G)^{ab}  \geq \rank  \ZZ (\mathcal{U}(\Z G)).$$ In case  $\Q G$ does not have  exceptional components the above is 
an equality by \cite[Theorem 6.3]{JeOlvGdR}, and hence question (R1) has a positive answer  in that case. However in \cite[Theorem D]{BMM} 
one group, where $\Q G$ has an exceptional component of the form $\Ma_2(\Q(i))$, was found where the rank of the abelianisation is non-zero but the center of the unit group is finite (thus the rank of $\U(\Z G)^{ab}$ is larger than expected).
Surprisingly, crucially using $\mathcal{H}(G)$, for each group in $\mc{G}_i \times C_2^n$, for $1 \leq i \leq 7$, we obtain an 
expression for the exponent and rank of the abelianisation. 

Recall that for $G \in \mc{G}_i$ and $e \in \PCI(\Q G)$, one has that $\Q G e$ either is a division algebra that is not exceptional of type I or it is a simple component of type $M_2(\Q )$ (and thus of exceptional type II). The set of primitive central idempotents $e$ of the latter type we will denote (as in the proof of Theorem~\ref{H-units for 2-grps give finite index})  as 
$\mathcal{E}_{exc}$. Also recall that for $e\in \mathcal{E}_{exc}$ the group
 $G$ can be written as an extension as follows:
$$1 \rightarrow N \longrightarrow G \xrightarrow{\pi_e} Ge \cong D_8= \langle a : a^4 =1 \rangle \rtimes \langle b : b^2=1 \rangle \rightarrow 1.$$
Consider the following property:
\begin{equation}\label{prop star}
    (\star)\, \forall e \in \mathcal{E}_{exc}, \exists\,  g,h \in G : o(\pi_e(g))= o(ab) \text{ and } o(\pi_e(h)) = o(b).
\end{equation}

Note that this property is satisfied if $G$ is a split extension of $D_8.$ Under the additional property $(\star)$ we will give a positive answer to  (R1) and (P), despite that such groups may have arbitrarily many exceptional components. 

\begin{theorem}\label{Theorem abelianisation of only M2(Q) components}
Let $G = K \times C_2^n$ with $K$ a group in $\mathcal{G}_1 \cup \ldots \cup \mathcal{G}_7$ and $\pi$ the natural  map of $\U(\Z G)$ onto $\U(\Z G)^{ab}$. Then 
$$
\rank  \U(\Z G)^{ab}  = \rank \ZZ (\mathcal{U}(\Z G)) + \rank  \pi(\langle \mc{H}(G)_{un}\rangle) ,
$$
where $\mc{H}(G)_{un} = \{ x \in \mc{H}(G) \mid x \text{ is unipotent } \}.$ Furthermore,
$$ \exp(V(\Z G)^{ab}) = \lcm\left( \exp(G^{ab}), \exp \left(V(\Z G)/\ZZ(V(\Z G))\,  cl_{\U(\Z G)}(\langle \Bic(G), \pm G \rangle\right)^{ab} \right).$$
Moreover, if $G$ satisfies $(\star)$, then $cl_{\U(\Z G)}(\langle \Bic(G), \pm G \rangle)$, i.e. the normal closure of the bicyclic and trivial units, 
together with the centre $\ZZ(\U(\Z G))$ is of finite index in $\U( \Z G)$ and (R1) and (P) have a positive answer.
\end{theorem}
Recall that for the group $\mathcal{G}_1$ the group $\langle \Bic(G), \pm G \rangle$ is of infinite index, which we expect to be a rather general phenomena for the classes of groups considered. Thus it is somehow surprising that under property $(\star)$ their normal closure is of finite index.
 
\begin{remark} (1) The combination of \cite[Lemma 4.4]{JM} with \Cref{H-unit contain principal congruence in one comp} yields that $\langle \mc{H}(G)_{un} \rangle$ is of infinite index in $\SL_1(\Z G)$ whenever $|G| \geq 32$ and $G = K \times C_2^n$ with $K \in \mathcal{G}_1 \cup \ldots \cup \mathcal{G}_7$. Thus for this class of groups the rank-formula reduces the problem of 
determining the abelianisation to a significantly smaller group.

(2) The proof of \Cref{Theorem abelianisation of only M2(Q) components} will furthermore yield that
$$
\rank \U(\Z G)^{ab}  = \rank \ZZ (\mathcal{U}(\Z G)) + \rank \left(\U(\Z G)/\ZZ (\mathcal{U}(\Z G))\, cl_{\U(\Z G)}(\langle \Bic(G), \pm G \rangle)\right)^{ab}.
$$
Clearly $D_8\in \mathcal{G}_{4}$ and  $G=D_8 \times C_2^n$ is a cut group (i.e. $\ZZ (\mathcal{U}(\Z G))$ is finite) and $G$ satisfies $(\star)$, thus \Cref{Theorem abelianisation of only M2(Q) components} yields that 
$$\rank( \U(\Z [D_8 \times C_2^n])^{ab} ) =0$$
for all $n$, answering \Cref{question on abelianisation} in case $G = D_8$.
\end{remark}
\begin{proof}[Proof of the torsion-free part of \Cref{Theorem abelianisation of only M2(Q) components}]
Recall that taking abelianisation is right exact. More precisely, for any group $\Gamma$ and normal subgroup $N$ ons has the short exact sequence
\begin{equation}\label{SES at level of ab}
1 \rightarrow \pi_{\G} (N) \rightarrow \G^{ab} \rightarrow (\G/N)^{ab} \rightarrow 1
\end{equation}
where  $\pi_{\G}: \G \rightarrow \G^{ab}$ is the canonical projection and thus
$\pi_{\G}(N)\cong  N/N\cap [\G,\G]$.

In particular 
\begin{equation}\label{equality rank in SES}
\rank (\G^{ab}) = \rank(\pi_{\G}(N)) + \rank( (\G/N)^{ab} ).
\end{equation}

Furthermore, it follows from the proof of \cite[Proposition 5.5.1]{EricAngel1} that $\ZZ (\U (\Z G)) \cap \SL_1(\Z G)$ is finite. Since $[\U(\Z G), \U( \Z G)] \leq \SL_1(\Z G)$ and thus also $\ZZ (\U (\Z G)) \cap [\U(\Z G), \U( \Z G)]$ is finite, this implies that 
\begin{equation}\label{rank with central part}
\rank \pi( \langle \ZZ (\U (\Z G)), N \rangle ) = \rank \ZZ (\U (\Z G)) + \rank \pi(N)\end{equation} 
for every normal subgroup $N$ in $\U (\Z G)$ such that $\pi(N) \cap \pi(\ZZ (\U (\Z G)))$ is finite. The latter condition is for example satisfied for a subgroup of $\langle \SL_1(\Z G), \pm G \rangle.$ Moreover, if $N \leq \SL_1(\Z G)$, then
\begin{equation*}\label{rank with SL1}
\rank \left( \SL_1(\Z G)/N \right)^{ab} = \rank \left( \SL_1(\Z G)\, \ZZ (\U (\Z G)/N\,  \ZZ (\U (\Z G)\right)^{ab}
\end{equation*}
and thus, 
because $\langle \ZZ (\U (\Z G)), \SL_1(\Z G) \rangle$ is of finite index in $\U (\Z G)$ (see \cite[Proposition 5.5.1]{EricAngel1}),
\begin{equation}\label{rank with SL1 lower bound}
\rank \left( \SL_1(\Z G)/N \right)^{ab}  \geq \rank \left( \U(\Z G)/N\, \ZZ (\U (\Z G)\right)^{ab}.
\end{equation}

Consider $\mathcal{E}_{exc} := \{e \in \PCI(\Q G) \mid \Q Ge \cong \Ma_2(\Q) \}$ and let $e_1, \ldots, e_q$ be its distinct elements.
For the rest of the proof consider the subgroup 
$\langle \pm G, \Bic (G), \prod_{i=1}^q \mc{H}(g_i,h_i,Q_i) \rangle$ delivered by \Cref{families of grps with H-unit new}, where 
 by \Cref{H-unit contain principal congruence in one comp},
$$\mc{H}(g_i,h_i,Q_i) = \SL_1(\Z G) \cap (1 - e_i + \Q Ge_i).$$ 
As explained in the proof of \Cref{H-unit contain principal congruence in one comp}, for each $i$ we have a  morphism as in (\ref{embedding h-unit in matrices}), which we now denote by $\Phi_i$, is the composition of projecting to $\Q Ge_i$ with the isomorphism $\phi$ defined in (\ref{def phi from section 8}) (\Cref{subsection start case elemen ab}). Moreover, by \Cref{remark on triviality H-unit}, one has that $e_i = \widehat{\langle Q_i, -[g_i,h_i] \rangle} = \widehat{Q_i}\,\frac{(1-[g_i,h_i])}{2}$. By \Cref{H-unit contain principal congruence in one comp}, $\mc{H}(g_i,h_i,Q_i) = 1 - e_i + V_{o(g_iQ_i)\, |Q_i|}$. Note that $o(g_iQ_i)\, |Q_i| = 2 |Q_i| = |G|/4$ does not  depend on $i$, and thus  $|Q_i| = 2^{m-1}$ for some positive integer $m$. \vspace{0,2cm}

First we choose $N = \ZZ (\U (\Z G)) \, cl_{\U(\Z G)}(\langle \Bic(G), \pm G \rangle).$ It follows from\footnote{In \cite{BMM} the result is only proven for bicyclic units of the form $1 + (1-h)g \wt{h}$, but the same proof works for those of the form $1 + \wt{h} g (1-h)$.} \cite[Proposition 3.1]{BMM} that any bicyclic unit of $\Z G$ is a product of commutators of elements of $\langle \Bic(G), \pm G \rangle$. Hence,  $\pi(cl_{\U(\Z G)}(\langle \Bic(G), \pm G \rangle))$ is finite.  Consequently, because of (\ref{equality rank in SES}),
and (\ref{rank with central part}) we obtain that: 
\begin{eqnarray}\label{rank bounded via SL_1}
\rank \U(\Z G)^{ab}  &=& \rank \left(\pi (\ZZ (\mathcal{U}(\Z G)) \, cl_{\U(\Z G)}(\langle \Bic(G))\right) \nonumber\\
&&+ \rank  \left(\U(\Z G)/\ZZ (\mathcal{U}(\Z G))\, cl_{\U(\Z G)}(\langle \Bic(G), \pm G \rangle)\right)^{ab}\nonumber\\
&=&\rank \ZZ (\mathcal{U}(\Z G)) + \rank \left(\U(\Z G)/\ZZ (\mathcal{U}(\Z G))\, cl_{\U(\Z G)}(\langle \Bic(G), \pm G \rangle)\right)^{ab}. 
\end{eqnarray}

Thus, if $\ZZ (\mathcal{U}(\Z G))\, cl_{\U(\Z G)}(\langle \Bic(G), \pm G \rangle)$ is of finite index, then (R1) has a positive answer.

\noindent Now consider the bicyclic unit $b_{\widetilde{h_i},g_i} := 1 + \wt{h_i} g_i (1- h_i^{-1})$. Using that $o(h_iQ) = 2$ one sees that 
$$e_i b_{\wt{h_i},g_i} = e_i\left( 1 + \frac{o(h_i)}{o(h_iQ)} (1+h_i)g_i (1-h_i) \right) = e_i \left( 1 + o(h_i) (1+h_i)g_i \right) = e_i v_{(\frac{-o(h_i)}{2}, \frac{o(h_i)}{2}, \frac{-o(h_i)}{2}, \frac{o(h_i)}{2})}.$$
Therefore, using (\ref{embedding h-unit in matrices}), composing with $\phi$ gives that 
$$\Phi_i(b_{\wt{h_i},g_i}) = \left( \begin{array}{ll} 1 + o(h) & - o(h) \\ o(h) & 1-o(h)  \end{array} \right).$$

Next consider the bicyclic units $b_{h_i,\widetilde{g_i}} := 1 + (1-g_i)h_i\wt{g_i}$ and $b_{\widetilde{g_i}, h_i} := 1 + \wt{g_i} h_i (1- g_i)$. Analogously as with the preceding unit one verifies that 
$$\begin{array}{rcl}
e_i b_{h_i,\tilde{g_i}} & = & e_i \left( 1 + \frac{o(g_i)}{o(g_iQ_i)} (1-g_i) h_i (1+ g_i) \right) = e_i v_{(0, o(g_i), 0,0)} \\
e_i b_{\tilde{g_i}, h_i} & = & e_i \left( 1 + \frac{o(g_i)}{o(g_iQ_i)} (1+g_i) h_i (1- g_i) \right) = e_i v_{(o(g_i), 0, 0,0)}.
\end{array}$$
Thus after composing with $\phi$ we obtain the matrices 
$$\Phi_i(b_{h_i,\tilde{g_i}}) = \left( \begin{array}{lc} 1 & 2 o(g_i) \\ 0 & 1  \end{array} \right)  \text{ and } \Phi_i(b_{\tilde{g_i}, h_i})= \left( \begin{array}{lc} 1 & 0 \\ 2 o(g_i) & 1  \end{array} \right).$$ 

{\it Now suppose that $G$ has property $(\star).$} In particular, $o(h_i) = 2 = o(g_i)$ and so the matrices obtained above are simply those already encountered in \Cref{prop description unit twisted group ring}. Namely : 
$$ (\Phi_i(b_{\wt{h_i},g_i}), \Phi_i(b_{h_i,\tilde{g_i}}), \Phi_i(b_{\tilde{g_i}, h_i}) ) = ( vw^{-1}, w^2, v^2 ).$$
Therefore, by (\ref{generators U_1}) we obtain that 
$$\Phi_i(\pm \langle b_{h_i,\tilde{g_i}}, b_{\tilde{g_i}, h_i}, b_{\wt{h_i},g_i} \rangle ) \cong V_2 =  \{ \left(\begin{array}{cc}
1 + 2 \, l_2 & 2 \, t_1  \\
 2 \,t_2 & 1 + 2\, l_1
\end{array} \right) \in \SL_2(\Z) \mid l_1 \equiv l_2 \text{ and } t_1 \equiv t_2 \mod 2 \}.$$
In particular, because of the description given in \eqref{H-units as Vm}, we get that $\mc{H}(g_i,h_i,Q_i) \cong \Phi_i(\mc{H}(g_i,h_i,Q_i) ) \, \triangleleft \, \Phi_i(\pm \langle b_{h_i,\tilde{g_i}}, b_{\tilde{g_i}, h_i}, b_{\wt{h_i},g_i} \rangle )$. Denote 
$$H := cl_{\U(\Z G)}\left( \langle b_{h_i,\tilde{g_i}}, b_{\tilde{g_i}, h_i}, b_{\wt{h_i},g_i}, g_i, h_i \mid 1 \leq i \leq q \rangle \right).$$
Next, note that 
$$[H, \mc{H}(g_i,h_i,Q_i)] = [1-e_i + He_i, \mc{H}(g_i,h_i,Q_i)] \cong [\Phi_i (H), \Phi_i(\mc{H}(g_i,h_i,Q_i))],$$
with $\Phi_i(g_i) = \left( \begin{array}{lr} 1 & 0 \\ 0 & -1  \end{array} \right)$ and $\Phi_i(h_i) = \left( \begin{array}{ll} 0 & 1 \\ 1 & 0  \end{array} \right).$ We will prove that $H$ is of finite index in $\prod_{i=1}^{q} \U(\Z G e_i).$ Since the latter contains $\U(\Z G)$ and $H \leq cl_{\U (\Z G)}(\langle \Bic(G), \pm G \rangle)$, this would finish the proof of the last statement of the Theorem.

Recall that $\mc{H}(g_i,h_i,Q_i) = 1 -e_i + V_{|G|/4}$. Hence, by the above descriptions,
\begin{equation}\label{identification commutator}
[ H , \prod_{i=1}^q \mc{H}(g_i,h_i,Q_i)] = \prod_{i=1}^q 1 - e_i + [\langle V_2,\Phi(g_i), \Phi(h_i) \rangle, V_{|G|/4}].
\end{equation}

We claim that $ V_{|G|/4}/ [\langle V_2,\Phi(g_i), \Phi(h_i) \rangle , V_{|G|/4}] $ is finite. To prove this we need following general group theoretical inequality: \vspace{0,1cm}

\noindent {\it Claim $1$: } Let $N_2 \leq N_1$ be normal finite index subgroups of some group $\Gamma$. Denote $[N_1 / [\Gamma , N_2] : N_2 / [\Gamma , N_2]]= n$ which divides $[N_1 : N_2]$. If $N_1 / [\Gamma , N_1]$ is finite, then $N_2 / [\Gamma, N_2]$ is finite. Furthermore, $|[\G, \G]/[\G, N_2]|$ divides $\left( [\G : N_1] n \right)^{[s/2]+1}$ where $s$ is the product of all $p^{e_p}$ with $e_p$ the maximum exponent of $p$ dividing $n[\G : N_1]$.\vspace{0.2cm}

Indeed, to prove the claim and for notation simplicity we denote $M := [\G, N_2].$ Note that $N_2/M$ is central in 
$\G / M$ and thus the latter is central-by-finite. Hence, by a well-known theorem of Schur, $[\G/M, \G/M]$ is finite. 
Moreover, by \cite[Theorem 1]{Wehr}, $|[\G/M, \G/M]| \mid ([\G/M : \ZZ (G)/M])^{[t/2]+1}$ with $t$ the product of 
all $p^{l_p}$ with $l_p$ the maximum exponent of $p$ dividing $[\G/M : \ZZ (\G)/M]$. As 
$[\G/M : \ZZ (\G)/M] \mid [\G/M : N_1/M]. [N_1/M : N_2/M]$ it also divides the  multiple mentioned in the statement of the claim. 
Now, as $[\G, N_1]/M \leq [\G, \G]/M$ is finite, and by assumption also $N_1/[\G,N_1]$, we obtain that $N_1/M$ is finite. 
In particular $N_2 / M$ is finite as desired. 

Using this we can now prove that:\vspace{0,1cm}

\noindent {\it Claim 2: } $ V_{|G|/4}/ [\langle V_2,\Phi(g_i), \Phi(h_i) \rangle , V_{|G|/4}] $ is finite. \vspace{0,2cm}

\noindent Applying Claim $1$ to $\Gamma = \langle V_2,\Phi(g_i), \Phi(h_i) \rangle, N_2 = V_{|G|/4}$ and $N_1 = V_2$ we see that it is enough to prove that 
$$V_{2}/ [\langle V_2,\Phi(g_i), \Phi(h_i) \rangle , V_{2}] \text{ is finite.}$$

The latter can be seen via an explicit set of generators. Concretely, following (\ref{generators U_1}), $V_2  = \langle w^2,v^2,w^{-1}v, -1 \rangle$ with $w= \left( \begin{array}{ll} 1 & 2 \\ 0 & 1  \end{array} \right)$ and $v= \left( \begin{array}{ll} 1 & 0 \\ 2 & 1  \end{array} \right)$. It is readily computed that 
\begin{equation}\label{actie h en g op de v en w}
v^{\Phi_i(h_i)} = w, v^{\Phi_i(g_i)} = v^{-1} \text{ and } w^{\Phi_i(g_i)} = w^{-1}.
\end{equation}
From this one verifies that the square of each generator is a single commutator. Therefore, $V_2 / [\langle \Phi_i(g_i), \Phi_i(h_i) \rangle, V_2]$ is an elementary abelian $2$-group, finishing the proof of Claim 2.

Thus all together we have proven that when $G$ satisfies property $(\star)$, 
then $\prod_{i=1}^q [H, \mc{H}(g_i,h_i,Q_i)]$ is of finite index in $\prod_{i=1}^q \U( \Z G e_i)$. 
But by the normality of $H$ in $\U (\Z G)$ we have that $[H, \mc{H}(g_i,h_i,Q_i)] \leq H$ for each $i$ and hence $\ZZ(\U (\Z G))\,H$  indeed is of finite index in $\U (\Z G).$ As mentioned earlier, by (\ref{rank bounded via SL_1}) this implies that (R1) holds for $G$ if it satisfies $(\star)$.\vspace{0,2cm}

Now consider again a general $G \in \mc{G}_j \times C_2^n.$ We will now apply (\ref{equality rank in SES}) to $\G = \U(\Z G)$ and $N = \langle \ZZ (\U (\Z G), \mc{H}(G)_{un} \rangle$. Using (\ref{rank with central part}), we see that the first part of the statement follows if $\left( \U( \Z G)/\ZZ (\U (\Z G)). \langle \mc{H}(G)_{un} \rangle \right)^{ab}$ is finite.

For this recall that each $\mc{H}(g_i,h_i,Q_i)$ contains the element $1- e_i + \begin{pmatrix} 1 & 2^{m} \\ 0 & 1
\end{pmatrix}$ (for simplicity, we abuse notation by writing matrices in $e_i$-part). As $V_{o(g_iQ_i).|Q_i|}$ is normal in $\U(\Z G e_i)$, the H-units $\mc{H}(g_i,h_i,Q_i)$ 
contain the group $A_i := 1 - e_i + cl_{\U(\Z G e_i)} (\langle \begin{pmatrix} 1 & 2^{m} \\ 0 & 1
\end{pmatrix} \rangle).$ Notice that $A_i \leq \langle \mc{H}(G)_{un} \rangle.$ Now as abelianisation is right exact and 
by (\ref{rank with SL1}) it is enough to prove that $\left( \SL_1( \Z G)/ \prod_{i} A_i \right)^{ab}$ is finite. To prove the latter
it is sufficient to show that $\prod_{i=1}^q \left( \SL_1( \Z G e_i)/ A_i \right)^{ab}$ is finite. This is directly verified using the presentation in \Cref{prop description unit twisted group ring} as also 
$\begin{pmatrix} 1 & 0 \\ 2^{m} & 1 \end{pmatrix} = 1- e_i + h^{-1}e_i \begin{pmatrix} 1 & 2^{m} \\ 0 & 1
\end{pmatrix} he_i  \in A_i$. This finishes the proof of the torsion-free part of the Theorem.
\end{proof}

Next,

\begin{proof}[Proof of the torsion part of \Cref{Theorem abelianisation of only M2(Q) components}]
For the torsion statement we consider the short exact sequence (\ref{SES at level of ab}) for $\G = V(\Z G)$ and $N = \ZZ (\U (\Z G))\, cl_{V(\Z G)}(\langle \Bic(G), G\rangle)$. Hereby the following is crucial:

\noindent{\it Claim 3:} For $\G = V(\Z G)$ and $N = \ZZ (\U (\Z G))\, cl_{V(\Z G)}(\langle \Bic(G), G\rangle)$ holds:
\begin{equation}\label{SES for exp torsion}
\exp (\G^{ab}) = lcm (\exp (\pi(N)), \exp (\G/N)^{ab} ).
\end{equation}

When $\pi(N)$ is finite Claim 3 is clear, as the claim even holds more generally for any short exact sequence of finitely generated abelian groups with finite kernel. In particular, recalling that $\pi(cl_{V(\Z G)}(\langle \Bic(G), G\rangle))$ is finite, we have that Claim 3 holds when $G \notin \mc{G}_7 \times C_2^n$ as for such groups $G$ we have that $\ZZ(\U (\Z G))$ is finite and hence also so is  $\pi(N)$. 

To handle the family $\mc{G}_7 \times C_2^n$ we need we need to do some more work.
First recall that by \cite[Theorem 6]{JesLealRio} and \cite[beginning of  
Section 5]{JesRioCrelle} the only simple components of $\Q[G]$ are of the form $\Q, \Q (i), \qa{-1}{-1}{\Q}, \qa{-1}{-1}{\Q(\sqrt{2})}$, 
or $\Ma_2(\Q)$. Therefore by \cite[Lemmas 2 and 3]{JesLealRio} and \cite[Lemma 5.3]{JesRioCrelle} $G$ is a subgroup of 
$C_2^{n_1} \times C_4^{n_2} \times Q_8^{n_3} \times D_8^{n_4} \times Q_{16}^{n_5}$ for some $n_1,\ldots,n_5\in \N$ with $n_5 = 0$ 
if $K \notin \mc{G}_7$. Recall that both $Q_8$ and $Q_{16}$ have a unique subgroup of order $2$ which moreover is central. 
Thus the second part of Claim 4, see below, holds for every subgroup of $C_2^{n_1} \times C_4^{n_2} \times Q_8^{n_3} 
\times D_8^{n_4} \times Q_{16}^{n_5}$. Also note that $G$ has a normal subgroup, say $A$, so that both groups $A$ and   $G/A$ are abelian and also  $ \exp (G/A)$ divides $4$.
A result of Cliff, Sehgal and Weiss (see \cite[Theorem 31.1]{SehgalBook93}) yields that $V_{tf}(\Z G) :=V(\Z G)\cap (1+\ker(\omega_A)\ker (\omega_G))$ is a torsion-free normal subgroup of $V(\Z G)$ and $V(\Z G)= V_{tf}(\Z G)\, G \cong V_{tf}(\Z G) \rtimes G$.
\vspace{0,2cm}

\noindent {\it Claim 4:} if $x \in V_{tf}(\Z G)$ is such that $x^m \in \ZZ (V(\Z G))$ for some $m \in \Z_{>0}$, then $x\in \ZZ (V(\Z G))$. Furthermore, $g^{o(g)/2} \in \ZZ(G)$ for every $g \in G$ with $o(g) > 2.$\vspace{0,2cm}

Take $e\in \PCI (\Q G)$. Note that the Cliff-Sehgal-Weiss result also holds for the quotient groups $Ge$ of $G$. 
Thus also $V(\Z[Ge])= V_{tf}(\Z[Ge]) \rtimes Ge$ has such a decomposition. Furthermore, in view of the explicit description 
of $V_{tf}(\Z G)$, the decompositions are compatible in the following sense: the natural epimorphism  $\pi_e: G \rightarrow Ge$ extends
to an epimorphism $\phi_e: \Z G \rightarrow \Z [Ge]$ and its restriction yields a morphism  $\phi_e: V(\Z G) \rightarrow V(\Z[Ge])$. It follows from the description of the basis of the kernel of a relative augmentation map that  $\phi_e(V_{tf}(\Z G)) {\bf \subseteq}  V_{tf}(\Z [Ge])$ and $\phi_e (\ZZ (V(\Z[G]))\subseteq \ZZ (V(\Z[Ge]))$.
This compatibility combined with the assumption yields that 
\begin{equation}\label{image power x}
\phi_e(x)^m \in V_{tf}(\Z[Ge]) \cap \ZZ (V(\Z[Ge])).
\end{equation}

Now if $Ge$ is abelian or isomorphic to $Q_8$ or $D_8$ then $\ZZ (V(\Z[Ge]))$ is finite. Thus 
$\phi_e(x)^m = e$ by (\ref{image power x}) and, since $\phi_e (x)$ belongs to the torsion-free group 
$V_{tf}(\Z [Ge]$, we thus even get $\phi_e(x)=e$ (in particular it is central). Next consider the case that $Ge \cong Q_{16}$ the quaternion group of order $16$. From  the description of the unit group of $\Z Q_{16}$ obtained in \cite[Theorem 4]{JesParment} it follows that $V_{tf}(\Z[Ge])=V_{tf}(\Z Q_{16})$ is the direct product of an infinite cyclic group, which moreover is $V_{tf}(\Z Q_{16})\cap  \ZZ\left( V(\Z [Ge])\right)$, and a non-abelian free group. Thus (\ref{image power x}) can only happen if $\phi_e(x) \in \ZZ\left( V(\Z [Ge])\right).$

So, we have shown that $\phi_e(x)$ is central for each $e\in \PCI (\Q G)$. Therefore $xe$, the projection of $x$ in the simple component $\Q Ge$, is central for each $e$. Hence, $x$ itself is central and the claim follows.
\vspace{0,1cm}

\noindent {\it Proof of Claim 3:} by Claim 4, and because $G\subseteq N$, we can choose a transversal $\mc{T}$ of $N$ in $\G$ such that
$\mc{T}\subseteq V_{tf}(\Z G)$ and 
the only elements in $\mc{T}$ with some power central are the central  elements. Now, for $x\in \G^{ab}$ write  $x = t\, y\, z$, with $t\in \pi(\mc{T}), z \in \pi(\ZZ (V(\Z G)))$ and $y \in \pi(\langle \Bic(G), G\rangle).$ If $x$ is periodic, then so are $t,y,z.$ To see this, recall that $\pi(cl_{V(\Z G)}(\langle \Bic(G), G\rangle))$ is finite. Thus for some positive integer  $n$ we have 
 $t^n$ is a central unit. However, by the choice of the transversal, this implies that  also $t$ is periodic. Consequently, the remaining component $z$ also needs to be periodic. With this Claim 3 now follows directly.\vspace{0,2cm}

The central and bicyclic units contribute as predicted by conjecture (E1):\vspace{0,1cm}

\noindent {\it Claim 5:} $\exp (\pi \left( \ZZ (V (\Z G))\,  cl_{V(\Z G)}(\langle \Bic(G), G \rangle)\right)) = \exp(G^{ab}).$\vspace{0.2cm}

First we note that $\exp(G^{ab}) \mid \exp (\pi(N))$ with $N= \ZZ (V (\Z G))\,  cl_{V(\Z G)}(\langle \Bic(G), G \rangle))$. To see this, consider the relative augmentation $\omega_{G'} : \Z[G] \rightarrow \Z[G/G']$. As $\U(\Z[G/G'])$ is abelian, we have an induced morphism $\wt{\omega_{G'}} : V(\Z G)^{ab} \rightarrow V(\Z[G/G'])$. Since $G'\subseteq [V(\Z G),V(\Z G)]$ and $\wt{\omega_{G'}}(g[V(\Z G),V( \Z G)])=gG'$ for $g\in G$, we have that any periodic element of $G/G'$ is a $\restr{\wt{\omega_{G'}}}{\pi(N)}$-image of a periodic element of $\pi(N)$. Hence it follows that $\exp (G/G') \mid \exp (\pi(N))$.

Thus in order to prove Claim 5, it is enough to show that $o(\pi(\alpha))$ divides $\exp(G^{ab})$ for every element $\alpha$ of $\ZZ (\U (\Z G)),  \Bic(G)$ and $G$. For the latter this trivially is true. Now consider a central unit $\alpha \in \ZZ (\U (\Z G))$. By \cite[Proposition 5.5.1]{EricAngel1} one has that $\ZZ (V(\Z G)) \cap [V(\Z G), V(\Z G)] \leq \ZZ (V(\Z G)) \cap \SL_1(\Z G)$ is finite. Thus if  $o(\pi(\alpha))$ is finite, then $\alpha$ itself is a periodic unit. Consequently, by a result of Berman and Higman \cite[Corollary 7.1.3]{PolSeh} (or see Theorem~\ref{size and comparission of ker}), $\alpha \in G$ and in particular $o(\pi(\alpha)) \mid \exp(G^{ab})$. 

 Next consider $\alpha = 1+ (1-x) y \wt{x} \in \Bic(G)$ whose inverse is $\alpha^{-1}= 1 - (1-x) y \wt{x}$. We will prove that
  \begin{equation}\label{bicyclic in ab}
 o(\pi(\alpha)) \mid 2.
 \end{equation}
 If $o(x) = 2,$ then (\ref{bicyclic in ab}) was obtained in \cite[Proposition 3.1]{BMM}. So suppose that $o(x) > 2$ and thus $o(x)=4$ or $8$. However, $o(x)=8$ only occurs for the class $\mc{G}_7$ and in that case $x$ generates a normal subgroup \cite[Lemma 5.7]{JesRioCrelle}. In particular if $o(x)=8,$ then $\alpha = 1.$ In conclusion, we may suppose that $o(x)=4.$

 It was noticed in the proof of 
 \cite[Proposition 3.1]{BMM} that $[\alpha^{-1}, x^k] = 1 + (1-x) (1- x^{-k}) y \wt{x}$ for any non-negative integer $k$. 
 Consequently, for $\mc{I}$ a subset of $\{ 1, \ldots, o(x) \}$ one has that
 $$\begin{array}{rcl}
    \prod\limits_{k \in \mc{I}} [\alpha^{-1}, x^k]  & = & 1+ \sum\limits_{k \in \mc{I}} (1-x) (1- x^{-k})y \wt{x} \\
      & = & 1 + |\mc{I}| (1-x)y\wt{x} - (1-x) (\sum\limits_{k \in \mc{I}} x^{-k}) y \wt{x}.
 \end{array}$$
 Now take $\mc{I} = \{ o(x)-1, o(x) \}$ and using that $\wt{x} = (1+x)(1+x^2)$ we see that
 $$(1-x) (\sum\limits_{k \in \mc{I}} x^{-k}) y \wt{x} = (1-x) (1+x) y (1+x)(1+x^2) = (1-x^2) (1+x^2) y (1+x) = 0,$$
 where we used that $x^2$ is central (by Claim 4). Altogether we have proven that $\alpha^2 = 1 + 2 (1-x) y \wt{x} \in [V(\Z G), V(\Z G)],$ yielding (\ref{bicyclic in ab}).\vspace{0,1cm}

The statement (\ref{bicyclic in ab}) also holds for the units $1+ \wt{x}y(1-x)$ and follows from an analogue proof. This finishes the proof of Claim 5. \vspace{0.2cm}

Claim 2 together with (\ref{SES for exp torsion}) yields the upper bound for $\exp (V (\Z G)^{ab})$ we were looking for.\vspace{0,2cm}

\noindent Now suppose that $G$ satisfies $(\star)$, then we  already have proven that the group generated by 
$\ZZ (V( \Z G))$ and $cl_{V(\Z G)}(\langle \Bic(G) , G) \rangle$ is of finite index in $V(\Z G)$. More precisely, recall that $\mc{H}(g_i,h_i,Q_i) = 1 - e_i + V_{|G|/4}$ and also the identification from (\ref{identification commutator}). With this we can reformulate Claim 2 saying that
$ \prod_i \mc{H}(g_i,h_i,Q_i)/ [cl_{V(\Z G)}(\langle \Bic(G) , g_i, h_i \rangle ), \mc{H}(g_i,h_i,Q_i)] $ is finite. To control the exponent of the latter we will pass over to an overgroup:
$$| \frac{\mc{H}(g_i,h_i,Q_i)}{[cl_{V(\Z G)}(\langle \Bic(G) , g_i, h_i \rangle ), \mc{H}(g_i,h_i,Q_i)]}| \text{ divides } | \frac{\langle V_2, \Phi(g_i),\Phi(h_i) \rangle}{[\langle V_2, \Phi(g_i),\Phi(h_i) \rangle, V_{|G|/4}]}|.$$

Next note that the proof of Claim 2 entails that $(\langle V_2, \Phi(g_i),\Phi(h_i) \rangle)^{ab}$ is a finite elementary abelian $2$-group. 
Therefore using the explicit bound from\footnote{More precisely we replace $[\G: N_1]n$ by the multiple 
$[\G: N_1].[N_1: N_2] = [\G : N_2]$.} Claim 1 in the setting of Claim 2 yields that 
$$\exp\left(\langle V_2, \Phi(g_i),\Phi(h_i) \rangle / [\langle V_2, \Phi(g_i),\Phi(h_i) \rangle, V_{|G|/4}] \right) 
\text{ divides } 2\, ([\langle V_2, \Phi(g_i),\Phi(h_i) \rangle : V_{|G|/4}])^{[t/2]+1}$$
where $t$ is the product of all $p^{t_p}$ with $t_p$ the maximum exponent of $p$ dividing $[\langle V_2, \Phi(g_i),\Phi(h_i) \rangle :
V_{|G|/4}].$ Now, recalling that $V_{2^i} \cong U_i$ by (\ref{Ui as matrix}), \Cref{the groups U_i for elemen ab} says that $[V_2 : V_{|G|/4}]$ is a power of two. Furthermore, \Cref{prop description unit twisted group ring} yields that $[\langle V_2, \Phi(g_i),\Phi(h_i) \rangle :
V_2]$ is a $2$-power. Summarized we have proven that 
$$\exp\left(\prod_i \frac{\mc{H}(g_i,h_i,Q_i)}{[cl_{V(\Z G)}(\langle \Bic(G) , g_i, h_i \rangle ), \mc{H}(g_i,h_i,Q_i)]} \right) \text{ is a power of } 2.$$
Combining this with the fact that $[\U (\Z G e_i): V_{|G|/4}]$ is $2$-power and
$$\exp \left( \frac{V(\Z G)}{\ZZ(V(\Z G))\, cl_{\U(\Z G)}(\langle \Bic(G), \pm G \rangle)}\right) \text{ divides }  \exp \left(
\prod_i \frac{\U(\Z G e_i)}{[cl_{V(\Z G)}(\langle \Bic(G) , g_i, h_i \rangle ), \mc{H}(g_i,h_i,Q_i)]} \right)$$
we obtain that $2$ is the only prime divisor of the left hand
side quotient group. Therefore, as $\exp(G^{ab}) \mid \exp(G)=4$ and thanks to the value obtained for $\exp V(\Z G)^{ab}$ conjecture (P) holds, finishing the proof.
\end{proof}

The statement that $cl_{\U (\Z G)}(\langle \Bic(G), \pm G \rangle)$ is of finite index when $G$ satisfies $(\star)$ does not mention $H$-units. However we would like to emphasize that the proof needed them and hence the result is truly a combined use of bicyclic and $H$-units. In upcoming work by the first author a systematic study of the abelianisation of $H$-units and their role on the rank of $\U(\Z G)^{ab}$ will be done.

\bibliographystyle{plain}
\bibliography{Twisted}

\end{document}